\documentclass[12pt]{amsart}

\usepackage{amssymb}
\usepackage{amsfonts}
\usepackage{amsmath}
\usepackage{graphicx}

\usepackage{tikz,bm,color}
\usetikzlibrary{shapes,arrows}
\usetikzlibrary{calc}
\usetikzlibrary{scopes}

\usetikzlibrary{matrix}

\newtheorem{theorem}{Theorem}[section]
\newtheorem{lemma}[theorem]{Lemma}
\newtheorem{definition}[theorem]{Definition}
\newtheorem{remark}[theorem]{Remark}
\newtheorem{corollary}[theorem]{Corollary}

\newcommand{\downn}{\downarrow\!\! n}

\newcommand{\F}{\mathfrak F}
\newcommand{\G}{\mathfrak G}
\newcommand{\HH}{\mathfrak H}
\newcommand{\T}{\mathfrak T}
\newcommand{\LL}{\mathfrak L}
\newcommand{\X}{\mathfrak X}
\newcommand{\Y}{\mathfrak Y}

\newcommand{\nR}{R\!\!\!\!\! \diagup}

\begin{document}

\title{Topological completeness of logics above {\bf S4}}
\author{Guram Bezhanishvili, David Gabelaia, Joel Lucero-Bryan}
\date{}



\begin{abstract}
It is a celebrated result of McKinsey and Tarski \cite{MT44} that {\bf S4} is the logic of the closure algebra $X^+$ over any dense-in-itself separable metrizable space. In particular, {\bf S4} is the logic of the closure algebra over the reals $\mathbf R$, the rationals $\mathbf Q$, or the Cantor space $\mathbf C$. By \cite{BG11}, each logic above {\bf S4} that has the finite model property is the logic of a subalgebra of $\mathbf Q^+$, as well as the logic of a subalgebra of $\mathbf C^+$. This is no longer true for $\mathbf R$, and the main result of \cite{BG11} states that each connected logic above {\bf S4} with the finite model property is the logic of a subalgebra of the closure algebra $\mathbf R^+$.

In this paper we extend these results to all logics above {\bf S4}. Namely, for a normal modal logic $L$, we prove that the following conditions are equivalent: (i)~$L$ is above {\bf S4}, (ii)~$L$ is the logic of a subalgebra of $\mathbf Q^+$, (iii)~$L$ is the logic of a subalgebra of $\mathbf C^+$. We introduce the concept of a well-connected logic above {\bf S4} and prove that the following conditions are equivalent: (i)~$L$ is a well-connected logic, (ii)~$L$ is the logic of a subalgebra of the closure algebra $\T_2^+$ over the infinite binary tree, (iii)~$L$ is the logic of a subalgebra of the closure algebra ${\bf L}_2^+$ over the infinite binary tree with limits equipped with the Scott topology. Finally, we prove that a logic $L$ above {\bf S4} is connected iff $L$ is the logic of a subalgebra of $\mathbf R^+$, and transfer our results to the setting of intermediate logics.

Proving these general completeness results requires new tools. We introduce the countable general frame property (CGFP) and prove that each normal modal logic has the CGFP. We introduce general topological semantics for {\bf S4}, which generalizes topological semantics the same way general frame semantics generalizes Kripke semantics. We prove that the categories of descriptive frames for {\bf S4} and descriptive spaces are isomorphic. It follows that every logic above {\bf S4} is complete with respect to the corresponding class of descriptive spaces. We provide several ways of realizing the infinite binary tree with limits, and prove that when equipped with the Scott topology, it is an interior image of both $\mathbf C$ and $\mathbf R$. Finally, we introduce gluing of general spaces and prove that the space obtained by appropriate gluing involving certain quotients of ${\bf L}_2$ is an interior image of $\mathbf R$.
\end{abstract}

\maketitle

\tableofcontents

\section{Introduction}\label{SecIntro}

Topological semantics for modal logic was developed by McKinsey and Tarski in 1930s and 1940s. They proved that if we interpret modal diamond as topological closure, and hence modal box as topological interior, then {\bf S4} is the logic of the class of all topological spaces. Their celebrated topological completeness result states that {\bf S4} is the logic of any dense-in-itself separable metrizable space \cite{MT44}. In particular, {\bf S4} is the logic of the real line $\mathbf R$, rational line $\mathbf Q$, or Cantor space $\mathbf C$.

In 1950s and 1960s, Kripke semantics for modal logic was introduced \cite{Kri59,Kri63} and it started to play a dominant role in the studies of modal logic. In 1970s it was realized that there exist Kripke incomplete logics \cite{Tho72}. To remedy this, general Kripke semantics was developed, and it was shown that each normal modal logic is complete with respect to the corresponding class of descriptive Kripke frames (see, e.g., \cite{Gol76a}).

Kripke frames for {\bf S4} can be viewed as special topological spaces, the so-called Alexandroff spaces, in which each point has a least open neighborhood. So topological semantics for {\bf S4} is stronger than Kripke semantics for {\bf S4}, but as with Kripke semantics, there are topologically incomplete logics above {\bf S4} \cite{Ger75}. It is only natural to generalize topological semantics along the same lines as Kripke semantics.

For a topological space $X$, let $X^+$ be the closure algebra associated with $X$; that is, $X^+=(\wp(X),{\bf c})$, where $\wp(X)$ is the powerset of $X$ and {\bf c} is the closure operator of $X$. We define a general space to be a pair $(X,\mathcal P)$, where $X$ is a topological space and $\mathcal P$ is a subalgebra of the closure algebra $X^+$; that is, $\mathcal P$ is a field of sets over $X$ closed under topological closure. As in general frame semantics, we introduce descriptive general spaces and prove that the category of descriptive spaces is isomorphic to the category of descriptive frames for {\bf S4}. This yields completeness of each logic above {\bf S4} with respect to the corresponding class of descriptive spaces.

Since descriptive spaces are in 1-1 correspondence with descriptive frames for {\bf S4}, it is natural to ask whether we gain much by developing general topological semantics. One of the main goals of this paper is to demonstrate some substantial gains. For a general space $(X,\mathcal P)$, we have that $\mathcal P$ is a subalgebra of the closure algebra $X^+$, thus general spaces over $X$ correspond to subalgebras of $X^+$. In \cite{BG11} it was shown that if $X$ is taken to be $\mathbf Q$ or $\mathbf C$, then every logic above {\bf S4} with the finite model property (FMP) is the logic of some subalgebra of $X^+$. In this paper, we strengthen this result by showing that all logics above {\bf S4} can be captured this way. Put differently, each logic above {\bf S4} is the logic of a general space over either $\mathbf Q$ or $\mathbf C$. Thus, such natural spaces as $\mathbf Q$ or $\mathbf C$ determine entirely the lattice of logics above {\bf S4} in that each such logic is the logic of a general space over $\mathbf Q$ or $\mathbf C$.

We are not aware of similar natural examples of Kripke frames. In fact, one of the most natural examples would be the infinite binary tree $\T_2$. We introduce the concept of a well-connected logic above {\bf S4} and prove that a logic $L$ above {\bf S4} is the logic of a general frame over $\T_2$ iff $L$ is well-connected, so $\T_2$ is capable of capturing well-connected logics above {\bf S4}.

We recall \cite{BG11} that a logic $L$ above {\bf S4} is connected if $L$ is the logic of a connected closure algebra. The main result of \cite{BG11} establishes that each connected logic above {\bf S4} with the FMP is the logic of a subalgebra of $\mathbf R^+$. Put differently, a connected logic $L$ above {\bf S4} with the FMP is the logic of a general space over $\mathbf R$. We strengthen this result by proving that a logic $L$ above {\bf S4} is connected iff $L$ is the logic of a general space over $\mathbf R$. This is equivalent to being the logic of a subalgebra of $\mathbf R^+$, which solves \cite[p.~306, Open Problem 2]{BG11}. We conclude the paper by transferring our results to the setting of intermediate logics.

We discuss briefly the methodology we developed to obtain our results. It is well known (see, e.g., \cite{CZ97}) that many modal logics have neither the FMP nor the countable frame property (CFP). We introduce a weaker concept of the countable general frame property (CGFP) and prove that each normal modal logic has the CGFP. This together with the fact \cite{BLB12a} that countable rooted {\bf S4}-frames are interior images of $\mathbf Q$ yields that a normal modal logic $L$ is a logic above {\bf S4} iff it is the logic of a general space over $\mathbf Q$, which is equivalent to being the logic of a subalgebra of $\mathbf Q^+$.

Every countable rooted {\bf S4}-frame is also a p-morphic image of $\T_2$. However, in the case of $\T_2$, only a weaker result holds. Namely, a normal modal logic $L$ is a logic above {\bf S4} iff it is the logic of a subalgebra of a homomorphic image of $\T_2^+$. Homomorphic images can be dropped from the theorem only for well-connected logics above {\bf S4}. In this case we obtain that a logic $L$ above {\bf S4} is well-connected iff it is the logic of a subalgebra of $\T_2^+$, which is equivalent to being the logic of a general frame over $\T_2$.

Since $\T_2$ is not an interior image of $\mathbf C$, in order to obtain our completeness results for general spaces over $\mathbf C$, we work with the infinite binary tree with limits $\LL_2$. This uncountable tree has been an object of recent interest \cite{Lan12,Kre13}. In particular, Kremer \cite{Kre13} proved that if we equip $\LL_2$ with the Scott topology, and denote the result by ${\bf L}_2$, then $\T_2^+$ is isomorphic to a subalgebra of ${\bf L}_2^+$. Utilizing Kremer's theorem and the CGFP yields that a logic $L$ above {\bf S4} is well-connected iff it is the logic of a general space over ${\bf L}_2$, which is equivalent to being the logic of a subalgebra of ${\bf L}_2^+$. Since ${\bf L}_2$ is an interior image of $\mathbf C$ and the one-point compactification of the countable sum of homeomorphic copies of ${\bf C}$ is homeomorphic to ${\bf C}$, we further obtain that a normal modal logic $L$ is a logic above {\bf S4} iff it is the logic of a general space over $\mathbf C$, which is equivalent to being the logic of a subalgebra of $\mathbf C^+$.

For $\mathbf R$, only a weaker result holds. Namely, a normal modal logic $L$ is a logic above {\bf S4} iff it is the logic of a subalgebra of a homomorphic image of $\mathbf R^+$. Homomorphic images can be dropped from the theorem only for connected logics above {\bf S4}. In this case we obtain that a logic $L$ above {\bf S4} is connected iff it is the logic of a subalgebra of $\mathbf R^+$, which is equivalent to being the logic of a general space over $\mathbf R$. To prove this result we again use the CGFP, a theorem of Kremer that $\T_2^+$ is isomorphic to a subalgebra of ${\bf L}_2^+$, the fact that certain quotients of ${\bf L}_2$ are interior images of $\mathbf R$, and a generalization of the gluing technique developed in \cite{BG11}.

The paper is organized as follows. In Section~\ref{SecPrelims} we provide all the necessary preliminaries. In Section~\ref{SecGenTopSem} we introduce general topological semantics, and show that the category of descriptive spaces is isomorphic to the category of descriptive frames for {\bf S4}. In Section~\ref{SecRationals} we introduce the CGFP, and prove that each normal modal logic has the CGFP. This paves the way for our first general completeness result: a normal modal logic $L$ is a logic above {\bf S4} iff it is the logic of a general space over $\mathbf Q$, which is equivalent to being the logic of a subalgebra of $\mathbf Q^+$. In Section~\ref{BinaryTreeSection} we prove our second general completeness result: a normal modal logic $L$ is a logic above {\bf S4} iff it is the logic of a subalgebra of a homomorphic image of $\T_2^+$. We also introduce well-connected logics above {\bf S4} and prove that a logic above {\bf S4} is well-connected iff it is the logic of a general frame over $\T_2$, which is equivalent to being the logic of a subalgebra of $\T_2^+$. In Section~\ref{SecL2} we give several characterizations of ${\bf L}_2$ and prove our third general completeness result: a normal modal logic $L$ is a logic above {\bf S4} iff it is the logic of a subalgebra of a homomorphic image of $\mathbf L_2^+$. We also show that a logic above {\bf S4} is well-connected iff it is the logic of a general space over ${\bf L}_2$, which is equivalent to being the logic of a subalgebra of ${\bf L}_2^+$. In Section~\ref{CantorSection} we prove that ${\bf L}_2$ is an interior image of $\mathbf C$. This yields our fourth general completeness result: a normal modal logic $L$ is a logic above {\bf S4} iff it is the logic of a general space over $\mathbf C$, which is equivalent to being the logic of a subalgebra of $\mathbf C^+$. In Section~\ref{SecOpenSubspacesOfRealLine} we prove that ${\bf L}_2$ is an interior image of $\mathbf R$. From this we derive our fifth general completeness result: a normal modal logic $L$ is a logic above {\bf S4} iff it is the logic of a subalgebra of a homomorphic image of $\mathbf R^+$. In Section~\ref{SecRealLineAndConnLogs} we generalize the gluing technique of \cite{BG11} and prove the main result of the paper: a logic $L$ above {\bf S4} is connected iff it is the logic of a general space over $\mathbf R$, which is equivalent to being the logic of a subalgebra of $\mathbf R^+$. This solves \cite[p.~306, Open~Problem~2]{BG11}. Finally, in Section~\ref{SecIntLogs} we transfer our results to the setting of intermediate logics.

\section{Preliminaries}\label{SecPrelims}

In this section we recall some of the basic definitions and facts, and fix the notation. Some familiarity with modal logic and its Kripke semantics is assumed; see, e.g., \cite{CZ97} or \cite{BRV01}.

Modal formulas are built recursively from the countable set of propositional letters $\textrm{Prop}=\{p_1,p_2,\dots\}$ with the help of usual Boolean connectives $\land, \lor, \neg, \to, \leftrightarrow$, the constants $\top,\bot$, and the unary modal operators $\Diamond, \Box$. We denote the set of all modal formulas by $\mathfrak{Form}$. A set of modal formulas $L\subseteq\mathfrak{Form}$ is called a \emph{normal modal logic} if it contains all tautologies, the schemata $\Diamond(\varphi\vee\psi)\leftrightarrow(\Diamond\varphi\vee\Diamond\psi)$ and $\Box\varphi\leftrightarrow\neg\Diamond\neg\varphi$, the formula $\Diamond\bot\leftrightarrow\bot$, and is closed under Modus Ponens and Necessitation $\varphi / \Box\varphi$. The least normal modal logic is denoted by ${\mathbf K}$. Among the many normal extensions of $\mathbf K$, our primary interest is in {\bf S4} and its normal extensions. The logic $\mathbf{S4}$ is axiomatized by adding the following schemata to $\mathbf K$:
\[
\Diamond\Diamond \varphi\to\Diamond \varphi\ \ \ \ \ \ \ \ \ \varphi\to\Diamond \varphi.
\]
We will refer to normal extensions of {\bf S4} as {\em logics above ${\bf S4}$}.

The algebraic semantics for modal logic is provided by modal algebras. A \emph{modal algebra} is a structure $\mathfrak A=(A,\Diamond)$, where $A$ is a Boolean algebra and $\Diamond: A\to A$ is a unary function satisfying $\Diamond (a\vee b)=\Diamond a\vee\Diamond b$ and $\Diamond 0=0$. The unary function $\Box:A\to A$ is defined as $\Box a=\neg\Diamond\neg a$. It is quite obvious how modal formulas can be seen as polynomials over a modal algebra (see, e.g., \cite[Sec.~5.2]{BRV01} wherein polynomials are referred to as terms). We will say that a modal formula $\varphi(p_1,\dots,p_n)$ is \emph{universally true} (or \emph{valid}) in a modal algebra $\mathfrak A$ if $\varphi(a_1,\dots,a_n)=1$ for any tuple of elements $a_1,\dots,a_n\in A$ (in other words, when the polynomial $\varphi$ always evaluates to $1$ in $\mathfrak A$). In such a case, we may write $\mathfrak A\models\varphi$. We may also view an equation $\varphi=\psi$ in the signature of modal algebras as the corresponding modal formula $\varphi\leftrightarrow\psi$. Then the equation holds in a modal algebra $\mathfrak A$ iff the corresponding modal formula is valid in $\mathfrak A$. This yields a standard fact in (algebraic) modal logic that normal modal logics correspond to \emph{equational classes} of modal algebras, i.e.~classes of modal algebras defined by equations. By the celebrated Birkhoff theorem (see, e.g., \cite[Thm.~11.9]{BS81}), equational classes correspond to \emph{varieties}, i.e.~classes of algebras closed under homomorphic images, subalgebras, and products. For a normal modal logic $L$, we denote by $\mathcal V(L)$ the corresponding variety of modal algebras: $\mathcal V(L)=\{\mathfrak A : \mathfrak A\models L\}$. Conversely, for a class $\mathcal K$ of modal algebras, we denote by $L(\mathcal K)$ the modal logic corresponding to this class: $L(\mathcal K)=\{\varphi : \mathcal K\models \varphi\}$. The adequacy of algebraic semantics for modal logic can then be succinctly expressed as the equality $L=L(\mathcal V(L))$.

Of particular importance for us are modal algebras corresponding to {\bf S4}. These are known as closure algebras (or interior algebras or topological Boolean algebras or {\bf S4}-algebras). A modal algebra $\mathfrak A=(A,\Diamond)$ is a \emph{closure algebra} if $a\leq \Diamond a$ and $\Diamond\Diamond a\le\Diamond a$ for each $a\in A$.

Natural examples of modal algebras are provided by Kripke frames. A \emph{Kripke frame} is a pair $\F=(W,R)$, where $W$ is a set and $R$ is a binary relation on $W$. The binary relation $R$ gives rise to the modal operator acting on the Boolean algebra $\wp(W)$: for $U\subseteq W$, set $R^{-1}(U)=\{w\in W : wRv\textrm{ for some }v\in U\}$. We denote the modal algebra $(\wp(W), R^{-1})$ arising from $\F$ by $\F^+$. This enables one to interpret modal formulas in Kripke frames. Namely, to compute the meaning of a modal formula $\varphi(p_1,\dots,p_n)$ in a Kripke frame $\mathfrak F$, when the meaning of the propositional letters is specified by assigning a subset $U_i$ to the letter $p_i$, we simply compute the corresponding element $\varphi(U_1,\dots,U_n)$ in the modal algebra $\F^+$. A mapping $\nu:\textrm{Prop}\to\wp(W)$ is called a \emph{valuation} and the tuple $\mathfrak M=(\mathfrak F,\nu)$ is called a \emph{Kripke model}. A valuation $\nu$ extends naturally to $\mathfrak{Form}$, specifically $\nu(\varphi(p_1,\dots,p_n))$ is the element $\varphi(\nu(p_1),\dots,\nu(p_n))$ in $\F^+$. The notion of validity in a frame is defined as dictated by the corresponding notion for algebras. Given a normal modal logic $L$, we say that a frame $\mathfrak F$ is a \emph{frame for $L$} if all theorems of $L$ are valid in $\mathfrak F$ (notation: $\mathfrak F\models L$). It is easy to check that $\F^+$ is a closure algebra exactly when $R$ is a quasi-order; that is, when $R$ is reflexive and transitive.

It is well known that Kripke semantics is not fully adequate for modal logic. There exist modal logics (including logics above $\mathbf{S4}$) that are not complete with respect to Kripke frames (see, e.g., \cite[Sec.~6]{CZ97}). Algebraically this means that some varieties of modal algebras are not generated by their members of the form $\F^+$. To overcome this shortcoming, it is customary to augment Kripke frames with additional structure by specifying a subalgebra $\mathcal P$ of $\mathfrak F^+$. This brings us to the notion of a general frame. We recall that a \emph{general frame} is a tuple $\F=(W,R,\mathcal P)$ consisting of a Kripke frame $(W,R)$ and a set of possible values $\mathcal P\subseteq \wp(W)$ which forms a subalgebra of $(W,R)^+$. In particular, a Kripke frame $(W,R)$ is also viewed as the general frame $(W,R,\wp(W))$, and so we use the same notation $\F,\G,\HH,\dots$ for both Kripke frames and general frames. Valuations in general frames take values in the modal algebra $\mathcal P$ of possible values, so for a general frame $\F$ and a modal formula $\varphi$, we have $\F\models\varphi$ iff $\mathcal P\models\varphi$. We say that $\F$ is a \emph{general frame for} a normal modal logic $L$ if $\F\models L$. If $L$ is exactly the set of formulas valid in $\F$, we write $L=L(\F)$. It is well known that general frames provide a fully adequate semantics for modal logic (see, e.g., \cite[Sec.~8]{CZ97}). Namely, for every normal modal logic $L$, there is a general frame $\F$ such that $L=L(\F)$.

The gist of this theorem becomes evident once we extend the celebrated Stone duality to modal algebras. Let $\mathfrak A=(A,\Diamond)$ be a modal algebra and let $X$ be the Stone space of $A$ (i.e.~$X$ is the set of ultrafilters of $A$ topologized by the basis $A^*=\{a^*:a\in A\}$, where $a^*=\{x\in X : a\in x\}$). Define $R$ on $X$ by
\[
xRy \ \ \ \ \mathrm{ iff }\ \ \ \ (\forall a\in A)(a\in y\Rightarrow \Diamond a\in x).
\]
Then $(X,R,A^*)$ is a general frame. It is a special kind of general frame, called \emph{descriptive}.

For a set $X$, we recall that a {\em field} of sets over $X$ is a Boolean subalgebra $\mathcal P$ of the powerset $\wp(X)$. A field of sets $\mathcal P$ is {\em reduced} provided $x\ne y$ implies there is $A\in\mathcal P$ with $x\in A$ and $y\notin A$ and {\em perfect} provided for each family $\mathcal F\subseteq\mathcal P$ with the finite intersection property, $\bigcap\mathcal F\ne\varnothing$ (see, e.g., \cite{Sik60}).

\begin{definition} [see, e.g., \cite{CZ97}]
{\em
Let $\F=(W,R,\mathcal P)$ be a general frame.
\begin{enumerate}
\item We call $\F$ \emph{differentiated} if $\mathcal P$ is reduced.
\item We call $\F$ \emph{compact} if $\mathcal P$ is perfect.
\item We call $\F$ \emph{tight} if $w\nR v$ implies there is $A\in\mathcal P$ with $v\in A$ and $w\notin R^{-1}(A)$.
\item We call $\F$ \emph{descriptive} if $\F$ is differentiated, compact, and tight.
\end{enumerate}
}
\end{definition}

It is well known that descriptive frames provide a full duality for modal algebras much as in the case of Stone spaces and Boolean algebras. In fact, if we generate the topology $\tau_{\mathcal P}$ on $W$ from $\mathcal P$, then $(W,\tau_{\mathcal P})$ becomes the Stone space of $\mathcal P$ precisely when $(W,R,\mathcal P)$ is differentiated and compact. Thus, a descriptive frame $(W,R,\mathcal P)$ can be viewed as the Stone space of $\mathcal P$ equipped with a binary relation $R$ which satisfies the condition of tightness. This is equivalent to the $R$-image $R(w)=\{v\in W:wRv\}$ of each $w\in W$ being closed in the Stone topology. Consequently, descriptive frames can equivalently be viewed as pairs $(W,R)$ such that $W$ is a Stone space, $R(w)$ is closed for each $w\in W$, and $R^{-1}(A)$ is clopen for each clopen $A$ of $W$.

Given general frames $\F=(W,R,\mathcal P)$ and $\G=(V,S,\mathcal Q)$, a map $f:W\to V$ is called a \emph{p-morphism} if (a)~$wRw'$ implies $f(w)Sf(w')$; (b)~$f(w)S v$ implies there exists $w'$ with $wRw'$ and $f(w'){=}v$; and (c)~$A\in\mathcal Q$ implies $f^{-1}(A)\in\mathcal P$. It is well known that conditions (a) and (b) are equivalent to the condition $f^{-1}\left(S^{-1}(v)\right)=R^{-1}\left(f^{-1}(v)\right)$ for each $v\in V$. It is also well known that $f$ is a p-morphism between descriptive frames iff $f^{-1}:\mathcal Q\to\mathcal P$ is a modal algebra homomorphism. In fact, the category of modal algebras and modal algebra homomorphisms is dually equivalent to the category of descriptive frames and p-morphisms (see, e.g.~\cite[Sec.~8]{CZ97}).

Part of this duality survives when we switch to a broader class of general frames. Namely, p-morphic images, generated subframes, and disjoint unions give rise to subalgebras, homomorphic images, and products, respectively. Given general frames $\F=(W,R,\mathcal P)$ and $\G=(V,S,\mathcal Q)$, we say that $\G$ is a \emph{p-morphic} image of $\F$ if there is an onto p-morphism $f:W\to V$. If $f$ is 1-1, then we call the $f$-image of $\F$ a \emph{generated subframe} of $\G$. Generated subframes of $\G$ are characterized by the property that when they contain a world $v$, then they contain $S(v)$. If $f:W\to V$ is a p-morphism, then $f^{-1}:\mathcal Q\to\mathcal P$ is a modal algebra homomorphism. Moreover, if $f$ is onto, then $f^{-1}$ is 1-1 and if $f$ is 1-1, then $f^{-1}$ is onto. Thus, if $\G$ is a p-morphic image of $\F$, then $\mathcal Q$ is isomorphic to a subalgebra of $\mathcal P$, and if $\G$ is a generated subframe of $\F$, then $\mathcal Q$ is a homomorphic image of $\mathcal P$. Lastly, suppose $\F_i=(W_i,R_i,\mathcal P_i)$ are general frames indexed by some set $I$. For convenience, we assume that the $W_i$ are pairwise disjoint (otherwise we can always work with disjoint copies of the $W_i$). The \emph{disjoint union} $\F=(W,R,\mathcal P)$ of the $\F_i$ is defined by setting $W=\bigcup_{i\in I}W_i$, $R=\bigcup_{i\in I}R_i$, and $A\in \mathcal P$ iff $A\cap W_i\in\mathcal P_i$. Then the modal algebra $\mathcal P$ is isomorphic to the product of the modal algebras $\mathcal P_i$. We will utilize these well-known throughout the paper.

\section{General topological semantics}\label{SecGenTopSem}

As we pointed out in the introduction, topological semantics predates Kripke semantics. Moreover, Kripke semantics for {\bf S4} is subsumed in topological semantics. To see this, let $\F=(W,R)$ be an {\bf S4}-frame; that is, $\F$ is a Kripke frame and $R$ is reflexive and transitive. We call an underlying set of a generated subframe $\G$ of $\F$ an \emph{$R$-upset}. So $V\subseteq W$ is an $R$-upset if $w\in V$ and $wRv$ imply $v\in V$. The $R$-upsets form a topology $\tau_R$ on $W$, in which each point $w$ has a least open neighborhood $R(w)$. This topology is called an \emph{Alexandroff topology}, and can equivalently be described as a topology in which the intersection of any family of opens is again open. Thus, {\bf S4}-frames correspond to Alexandroff spaces. Consequently, each logic above {\bf S4} that is Kripke complete is also topologically complete. But as with Kripke semantics, topological semantics is not fully adequate since there exist logics above {\bf S4} that are topologically incomplete \cite{Ger75}.

As we saw in Section~\ref{SecPrelims}, Kripke incompleteness is remedied by introducing general Kripke semantics, and proving that each normal modal logic is complete with respect to this more general semantics. In this section we do the same with topological semantics. Namely, we introduce general topological spaces, their subclass of descriptive spaces, and prove that the category of descriptive spaces is isomorphic to the category of descriptive frames for {\bf S4}. This yields that general spaces provide a fully adequate semantics for logics above {\bf S4}.

For a topological space $X$, we recall that $X^+$ is the closure algebra $(\wp(X),{\bf c})$, where {\bf c} is the closure operator of $X$.

\begin{definition}
{\em
We call a pair $\X=(X,\mathcal P)$ a \emph{general topological space} or simply a \emph{general space} if $X$ is a topological space and $\mathcal P$ is a subalgebra of the closure algebra $X^+$.
}
\end{definition}

In other words, $\X=(X,\mathcal P)$ is a general space if $X$ is a topological space and $\mathcal P$ is a field of sets over $X$ that is closed under topological closure. In particular, we may view each topological space $X$ as the general space $(X,X^+)$. The definition of a general space is clearly analogous to the definition of a general frame. We continue this analogy in the next definition. In the remainder of the paper, when we need to emphasize or specify a certain topology $\tau$ on $X$, we will write $(X,\tau)$ as well as $(X,\tau,\mathcal P)$.
\begin{definition}\label{def:3.2}
{\em
Let $\X=(X,\tau,\mathcal P)$ be a general space.
\begin{enumerate}
\item We call $\X$ \emph{differentiated} if $\mathcal P$ is reduced.
\item We call $\X$ \emph{compact} if $\mathcal P$ is perfect.
\item Let $\mathcal P_\tau=\mathcal P\cap\tau$. We call $\X$ \emph{tight} if $\mathcal P_\tau$ is a basis for $\tau$.
\item We call $\X$ \emph{descriptive} if $\X$ is differentiated, compact, and tight.
\end{enumerate}
}
\end{definition}

\begin{remark}
{\em
Since $X\in\mathcal P_\tau$ and $\mathcal P_\tau$ is closed under finite intersections, $\mathcal P_\tau$ is a basis for some topology, and Definition~\ref{def:3.2}(3) says that this topology is $\tau$.
}
\end{remark}

For a topological space $(X,\tau)$, we let $R_\tau$ denote the \emph{specialization order} of $(X,\tau)$. We recall that $x R_\tau y$ iff $x\in{\bf c}(y)$, and that $R_\tau$ is reflexive and transitive, so $(X,R_\tau)$ is an {\bf S4}-frame.

\begin{lemma}\label{lem:3.3}
Let $(X,\tau,\mathcal P)$ be a compact tight general space. Then $R_\tau(x)=\bigcap\{A\in\mathcal P_\tau:x\in A\}$. Moreover, for each $A\in\mathcal P$ we have ${\bf c}(A)=R_\tau^{-1}(A)$. Consequently, $(X,R_\tau,\mathcal P)$ is a compact tight general {\bf S4}-frame. In particular, if $(X,\tau,\mathcal P)$ is a descriptive space, then $(X,R_\tau,\mathcal P)$ is a descriptive {\bf S4}-frame.
\end{lemma}

\begin{proof}
For $A\in\mathcal P_\tau$ it is obvious that $x\in A$ implies $R_\tau(x)\subseteq A$, so $R_\tau(x)\subseteq\bigcap\{A\in\mathcal P_\tau:x\in A\}$. If $y\notin R_\tau(x)$, then $x\notin{\bf c}(y)$. Since $(X,\tau,\mathcal P)$ is tight, there exists $A\in\mathcal P_\tau$ such that $x\in A$ and $y\notin A$. Thus, $R_\tau(x)=\bigcap\{A\in\mathcal P_\tau:x\in A\}$. Next, if $x\in R_\tau^{-1}(A)$, then there is $y\in A$ with $x R_\tau y$, so $x\in{\bf c}(y)\subseteq{\bf c}(A)$, and so $R_\tau^{-1}(A)\subseteq{\bf c}(A)$ for each $A\subseteq X$. Conversely, if $A\in\mathcal P$ and $x\notin R_\tau^{-1}(A)$, then $R_\tau(x)\cap A=\varnothing$. Therefore, $\bigcap\{U\in\mathcal P_\tau:x\in U\}\cap A=\varnothing$. Thus, by compactness, there are $U_1,\dots,U_n\in\mathcal P_\tau$ such that $x\in U_1\cap\dots\cap U_n$ and $U_1\cap\dots\cap U_n\cap A=\varnothing$. Let $U=U_1\cap\dots\cap U_n$. Then $U$ is an open neighborhood of $x$ missing $A$, so $x\notin{\bf c}(A)$. This yields that ${\bf c}(A)=R_\tau^{-1}(A)$ for each $A\in\mathcal P$. Consequently, $(X,R_\tau,\mathcal P)$ is a compact general {\bf S4}-frame. To see that it is tight, let $x\nR_\tau y$. Then there is $A\in\mathcal P_\tau$ such that $x\in A$ and $y\notin A$. Let $B=-A$, where we use $-$ to denote set-theoretic complement. Clearly $B$ is closed, so ${\bf c}(B)=B$. Therefore, since $A\in\mathcal P$, we have $R_\tau^{-1}(B)=B$. Thus, there is $B\in\mathcal P$ with $y\in B$ and $x\notin R_\tau^{-1}(B)$, and hence $(X,R_\tau,\mathcal P)$ is tight. In particular, if $(X,\tau,\mathcal P)$ is a descriptive space, then $(X,R_\tau,\mathcal P)$ is a descriptive {\bf S4}-frame.
\end{proof}

For a general {\bf S4}-frame $(X,R,\mathcal P)$, let $\mathcal P_R=\{A\in\mathcal P:A$ is an $R$-upset$\}$, and let $\tau_R$ be the topology generated by the basis $\mathcal P_R$. That $\mathcal P_R$ forms a basis is clear because $X\in\mathcal P_R$ and $\mathcal P_R$ is closed under finite intersections. We let ${\bf c}_R$ denote the closure operator in $(X,\tau_R)$.

\begin{lemma}\label{lem:3.4}
Let $(X,R,\mathcal P)$ be a compact tight general {\bf S4}-frame. Then $xRy$ iff $x\in{\bf c}_R(y)$. Moreover, for each $A\in\mathcal P$ we have $R^{-1}(A)={\bf c}_R(A)$. Consequently, $(X,\tau_R,\mathcal P)$ is a compact tight general space. In particular, if $(X,R,\mathcal P)$ is a descriptive {\bf S4}-frame, then $(X,\tau_R,\mathcal P)$ is a descriptive space.
\end{lemma}

\begin{proof}
If $(X,R,\mathcal P)$ is tight, then $R(x)=\bigcap\{A\in\mathcal P_R:x\in A\}$ for each $x\in X$. Thus, $xRy$ iff $x\in{\bf c}_R(y)$. Let $x\notin{\bf c}_R(A)$. Then there is $U\in\mathcal P_R$ such that $x\in U$ and $U\cap A=\varnothing$. As $R(x)\subseteq U$, we have $R(x)\cap A=\varnothing$, which means that $x\notin R^{-1}(A)$. Thus, $R^{-1}(A)\subseteq{\bf c}_R(A)$ for each $A\subseteq X$. Conversely, if $A\in\mathcal P$ and $x\notin R^{-1}(A)$, then $R(x)\cap A=\varnothing$. Since $(X,R,\mathcal P)$ is tight, $R(x)=\bigcap\{U\in\mathcal P_R:x\in U\}$ and so $\bigcap\{U\in\mathcal P_R:x\in U\}\cap A=\varnothing$. By compactness, there exist $U_1,\dots,U_n\in\mathcal P_R$ such that $x\in U_1\cap\dots\cap U_n$ and $U_1\cap\dots\cap U_n\cap A=\varnothing$. Let $U=U_1\cap\dots\cap U_n$. Then $U\in\mathcal P_R$, $x\in U$, and $U\cap A=\varnothing$. Thus, there is an open $\tau_R$-neighborhood of $x$ missing $A$, so $x\notin{\bf c}_R(A)$. This proves that $R^{-1}(A)={\bf c}_R(A)$ for each $A\in\mathcal P$. Consequently, $(X,\tau_R,\mathcal P)$ is a compact general space, and it is tight because $\mathcal P_R \subseteq \mathcal P \cap \tau_R \subseteq \tau_R$. In particular, if $(X,R,\mathcal P)$ is a descriptive {\bf S4}-frame, then $(X,\tau_R,\mathcal P)$ is a descriptive space.
\end{proof}

\begin{lemma}\label{lem:3.5}
\begin{enumerate}
\item[]
\item If $(X,\tau,\mathcal P)$ is a compact tight general space, then $\tau=\tau_{R_\tau}$.
\item If $(X,R,\mathcal P)$ is a compact tight general {\bf S4}-frame, then $R=R_{\tau_R}$.
\end{enumerate}
\end{lemma}

\begin{proof}
(1) It follows from Lemma~\ref{lem:3.3} that $\mathcal P_\tau=\mathcal P_{R_\tau}$ because $\mathbf c(A)=R_\tau^{-1}(A)$ for $A\in\mathcal P$. Since $\mathcal P_\tau$ is a basis for $\tau$ and $\mathcal P_{R_\tau}$ is a basis for $\tau_{R_\tau}$, we obtain that $\tau=\tau_{R_\tau}$.

(2) By definition, $xR_{\tau_R}y$ iff $x\in{\bf c}_R(y)$, and by Lemma~\ref{lem:3.4}, $x\in{\bf c}_R(y)$ iff $xRy$. Thus, $R=R_{\tau_R}$.
\end{proof}

Let {\bf DS} denote the category whose objects are descriptive spaces and whose morphisms are maps $f:X\to Y$ between descriptive spaces $\X=(X,\tau,\mathcal P)$ and $\Y=(Y,\tau,\mathcal Q)$ such that $A\in\mathcal Q$ implies $f^{-1}(A)\in\mathcal P$ and $f^{-1}{\bf c}(y)={\bf c}f^{-1}(y)$ for each $y\in Y$. We also let {\bf DF} denote the category whose objects are descriptive {\bf S4}-frames and whose morphisms are p-morphisms between them.

\begin{theorem}\label{thm:3.7}
{\bf DS} is isomorphic to {\bf DF}.
\end{theorem}

\begin{proof}
Define a functor $F:{\bf DS}\to{\bf DF}$ as follows. For a descriptive space $(X,\tau,\mathcal P)$, let $F(X,\tau,\mathcal P)=(X,R_\tau,\mathcal P)$. For a {\bf DS}-morphism $f:X\to Y$, let $F(f)=f$. By Lemma~\ref{lem:3.3}, $F(X,\tau,\mathcal P)\in{\bf DF}$. Moreover, since $R_\tau^{-1}(x)={\bf c}(x)$, it follows that $F(f)$ is a p-morphism. Thus, $F$ is well-defined.

Define a functor $G:{\bf DF}\to{\bf DS}$ as follows. For a descriptive frame $(X,R,\mathcal P)$, let $G(X,R,\mathcal P)=(X,\tau_R,\mathcal P)$. For a p-morphism $f:X\to Y$, let $G(f)=f$. By Lemma~\ref{lem:3.4}, $G(X,R,\mathcal P)\in{\bf DS}$. Lemma~\ref{lem:3.4} also implies that ${\bf c}_R(x)=R^{-1}(x)$, and hence it follows that $G(f)$ is a {\bf DS}-morphism. Thus, $G$ is well-defined.

Now apply Lemma~\ref{lem:3.5} to conclude the proof.
\end{proof}

\begin{remark}
{\em
Theorem~\ref{thm:3.7} holds true in a more general setting, between the categories of compact tight general spaces and compact tight general {\bf S4}-frames. However, we will not need it in such generality.
}
\end{remark}

Since each logic above {\bf S4} is complete with respect to the corresponding class of descriptive {\bf S4}-frames, as an immediate consequence of Theorem~\ref{thm:3.7}, we obtain:

\begin{corollary}
Each logic above {\bf S4} is complete with respect to the corresponding class of descriptive spaces.
\end{corollary}

Since the category {\bf CA} of closure algebras is dually equivalent to {\bf DF}, another immediate consequence of Theorem~\ref{thm:3.7} is the following:

\begin{corollary}\label{cor:3.10}
{\bf DS} is dually equivalent to {\bf CA}.
\end{corollary}

\begin{remark}
{\em
As we already pointed out, by Stone duality, having a reduced and perfect field of sets amounts to having a Stone space. Therefore, having a descriptive frame amounts to having a Stone space with a binary relation that satisfies the following two conditions: $R(x)$ is closed for each $x\in X$ and $R^{-1}(U)$ is clopen for each clopen $U$ of $X$.

Descriptive spaces can also be treated similarly. Namely, if $(X,\tau,\mathcal P)$ is a descriptive space, then announcing $\mathcal P$ as a basis yields a Stone topology on $X$, which we denote by $\tau_S$. Thus, we arrive at the bitopological space $(X,\tau_S,\tau)$, where $(X,\tau_S)$ is a Stone space. Moreover, since $(X,\tau,\mathcal P)$ is tight, $\mathcal P_\tau$ is a basis for $\tau$, so $\tau\subseteq\tau_S$. As $\mathcal P$ is the clopens of $(X,\tau_S)$, each member of $\mathcal P$ is compact in $(X,\tau_S)$, hence compact in the weaker topology $(X,\tau)$. Consequently, the members of $\mathcal P_\tau$ are compact open in $(X,\tau)$. Conversely, if $U$ is compact open in $(X,\tau)$, then since $\mathcal P_\tau$ is a basis for $\tau$ that is closed under finite unions, it follows easily that $U\in\mathcal P_\tau$. Thus, $\mathcal P_\tau$ is exactly the compact open subsets of $(X,\tau)$, and so the compact opens of $(X,\tau)$ are closed under finite intersections and form a basis for $\tau$. As $\mathcal P_\tau\subseteq\mathcal P$, we see that the compact opens of $(X,\tau)$ are clopens in $(X,\tau_S)$. Finally, $U$ clopen in $(X,\tau_S)$ implies that ${\bf c}(U)$ is clopen in $(X,\tau_S)$.

These considerations yield an equivalent description of descriptive spaces as bitopological spaces $(X,\tau_S,\tau)$ such that $(X,\tau_S)$ is a Stone space, $\tau\subseteq\tau_S$, the compact opens of $(X,\tau)$ are closed under finite intersections and form a basis for $\tau$, the compact opens of $(X,\tau)$ are clopens in $(X,\tau_S)$, and $U$ clopen in $(X,\tau_S)$ implies that ${\bf c}(U)$ is clopen in $(X,\tau_S)$. Thus, our notion of a descriptive space corresponds to the bitopological spaces defined in \cite[Def.~3.7]{BMM08}, Theorem~\ref{thm:3.7} corresponds to \cite[Thm.~3.8(2)]{BMM08}, and Corollary~\ref{cor:3.10} to \cite[Cor.~3.9(1)]{BMM08}. In fact, for a descriptive space $(X,\tau,\mathcal P)$, the closure algebra $X^+$ is the topo-canonical completion of the closure algebra $\mathcal P$ \cite{BMM08}.
}
\end{remark}

We recall that the basic truth-preserving operations for general frames are the operations of taking generated subframes, p-morphic images, and disjoint unions (see, e.g., \cite[Sec.~8.5]{CZ97}). We conclude this section by discussing analogous operations for general spaces. Recall (see, e.g., \cite{Gab01,vBBG03}) that for topological spaces interior maps are analogues of p-morphisms, where a map $f:X\rightarrow Y$ between topological spaces is an \emph{interior map} if it is continuous (inverse images of opens are open) and open (direct images of opens are open). It is well known that $f:X\rightarrow Y$ is an interior map iff $f^{-1}:Y^+\to X^+$ is a homomorphism of closure algebras. Moreover, if $f$ is onto, then $f^{-1}$ is 1-1 and if $f$ is 1-1, then $f^{-1}$ is onto. It follows that for topological spaces, open subspaces correspond to generated subframes and interior images correspond to p-morphic images. In addition, topological sums correspond to disjoint unions.

\begin{definition}
\begin{enumerate}
\item[]
\item Let $\X=(X,\mathcal P)$ and $\Y=(Y,\mathcal Q)$ be general spaces.
\begin{enumerate}
\item We say that a map $f:X\to Y$ is an \emph{interior map between $\X$ and $\Y$} if $f:X\to Y$ is an interior map and $A\in\mathcal Q$ implies $f^{-1}(A)\in\mathcal P$.
\item We call \emph{$\Y$ an open subspace of $\X$} if $Y$ is an open subspace of $X$ and the inclusion map $Y\to X$ is an interior map between the general spaces $\Y$ and $\X$.
\item We say that \emph{$\Y$ is an interior image of $\X$} if there is an onto interior map between the general spaces $\X$ and $\Y$.
\end{enumerate}
\item Let $\X_i=(X_i,\mathcal P_i)$ be general spaces indexed by some set $I$, and for convenience, we assume that the $X_i$ are pairwise disjoint. Let $X$ be the topological sum of the $X_i$. Define $\mathcal P\subseteq\wp(X)$ by $A\in\mathcal P$ iff $A\cap X_i\in\mathcal P_i$. Then it is straightforward to see that $\X=(X,\mathcal P)$ is a general space, which we call the \emph{sum of the general spaces $\X_i=(X_i,\mathcal P_i)$}.
\end{enumerate}
\end{definition}

Observe that an interior map $f$ between descriptive spaces is a {\bf DS}-map because it satisfies $f^{-1}{\bf c}(y)={\bf c}f^{-1}(y)$. Also, given general spaces $\X=(X,\mathcal P)$ and $\Y=(Y,\mathcal Q)$, if $\Y$ is an interior image of $\X$, then $\mathcal Q$ is isomorphic to a subalgebra of $\mathcal P$, and if $\Y$ is an open subspace of $\X$, then $\mathcal Q$ is a homomorphic image of $\mathcal P$. It is also clear that if a general space $\X=(X,\mathcal P)$ is the sum of a family of general spaces $\X_i=(X_i,\mathcal P_i)$, $i\in I$, then $\mathcal P$ is isomorphic to the product $\prod_{i\in I}\mathcal P_i$.

The definitions of a valuation in a general space $\X$, of $\X\models L$, and of $L(\X)$ are the same as in the case of general frames. So if a general space $\Y$ is an interior image of a general space $\X$, then $L(\X)\subseteq L(\Y)$. Similarly, if $\Y$ is an open subspace of $\X$, then $L(\X)\subseteq L(\Y)$. Finally, if $\X$ is the sum of the $\X_i$, then $L(\X)=\bigcap_{i\in I}L(\X_i)$.

\section{Countable general frame property and completeness for general spaces over $\mathbf Q$}\label{SecRationals}

By Theorem~\ref{thm:3.7}, descriptive spaces are the same as descriptive {\bf S4}-frames, but as we will see in what follows, it is the perspective of general spaces (rather than general {\bf S4}-frames) that allows us to obtain some strong general completeness results for logics above {\bf S4}. In this section we introduce one of our key tools for yielding these general completeness results, the countable general frame property. We then consider the rational line $\mathbf Q$, and prove our first general completeness result: a normal modal logic is a logic above {\bf S4} iff it is the logic of some general space over $\mathbf Q$, which is equivalent to being the logic of some subalgebra of~$\mathbf Q^+$.

Let $L$ be a normal modal logic. We recall that $L$ has the \emph{finite model property} (FMP) if each non-theorem of $L$ is refuted on a finite frame for $L$. This property has proved to be extremely useful in modal logic. The existence of sufficiently many finite models makes the study of a particular modal system easier. Unfortunately, a large number of modal logics do not have this property. This can be a major obstacle for investigating a particular modal system, as well as for proving general theorems encompassing all modal logics. A natural weakening of FMP is the \emph{countable frame property} (CFP): each non-theorem is refuted on a countable frame for the logic. But there are modal logics that do not have CFP either (see, e.g., \cite[Sec.~6]{CZ97}). We weaken further CFP to the \emph{countable general frame property} (CGFP) and show that all normal modal logics possess the CGFP.

\begin{definition}
{\em
Let $L$ be a normal modal logic. We say that $L$ has the {\em countable general frame property} (CGFP) provided for each non-theorem $\varphi$ of $L$ there exists a countable general frame $\mathfrak F$ for $L$ refuting $\varphi$.
}
\end{definition}

\begin{theorem}\label{CGFP}
Each normal modal logic $L$ has the {\em CGFP}.
\end{theorem}

\begin{proof}
Suppose $\varphi\notin L$. Then there is a general frame $\F=(W,R,\mathcal P)$ for $L$ refuting $\varphi$. Therefore, there is a valuation $\nu$ on $\F$ and $w\in W$ such that $w\not\in\nu(\varphi)$. We next select a countable subframe $\G$ of $\F$ that refutes $\varphi$. Our selection procedure is basically the same as the one found in \cite[Thm.~6.29]{CZ97}, where the L\"owenheim-Skolem Theorem for modal logic is proved. Set $V_0=\{w\}$. Suppose $V_n$ is defined. For each $\psi\in\mathfrak{Form}$ and $v\in V_n$ with $v\in\nu(\Diamond\psi)$, there is $u_{v,\Diamond\psi}\in R(v)$ with $u_{v,\Diamond\psi}\in\nu(\psi)$. We select one such $u_{v,\Diamond\psi}$ and let $V_{n+1}$ be the set of the selected $u_{v,\Diamond\psi}$. Finally, set $V=\bigcup_{n\in\omega} V_n$. Clearly $V$ is countable. Let $S$ be the restriction of $R$ to $V$ and $\mu$ be the restriction of $\nu$ to $V$. An easy induction on the complexity of modal formulas gives that for each $\psi\in\mathfrak{Form}$ and $v\in V$,
$$
v\in\nu(\psi)\text{ iff }v\in\mu(\psi).
$$
Therefore, $\mathfrak N=(V,S,\mu)$ is a countable submodel of $\mathfrak M=(W,R,\nu)$ such that $\mathfrak N$ is a model for $L$ and $\mathfrak N$ refutes $\varphi$. In fact, $\varphi$ is refuted at $w$. Set $\mathcal Q=\{\mu(\psi):\psi\in\mathfrak{Form}\}$ and $\G=(V,S,\mathcal Q)$. Then $\G$ is a countable general frame, and as $\mathfrak N$ refutes $\varphi$, so does $\G$. It remains to show that $\G$ is a frame for $L$. Let $\lambda$ be an arbitrary valuation on $\G$. It is sufficient to show that each theorem $\chi(p_1,\dots,p_n)$ of $L$ is true in $(\G,\lambda)$. Since each $\lambda(p_i)\in\mathcal Q$, there is $\psi_i\in\mathfrak{Form}$ such that $\lambda(p_i)=\mu(\psi_i)$. As $\chi(p_1,\dots,p_n)\in L$, we have $\chi(\psi_1,\dots,\psi_n)\in L$. Because $\mathfrak N$ is a model for $L$, it follows that $\chi(\psi_1,\dots,\psi_n)$ is true in $\mathfrak N$. Therefore, $\chi(p_1,\dots,p_n)$ is true in $(\G,\lambda)$. Thus, $\G$ is a frame for $L$, and so $L$ has the CGFP.
\end{proof}

\begin{remark}\label{EnlargingV0forCountableSet}
{\em
In the proof of Theorem~\ref{CGFP}, if we start the selection procedure by adding to $V_0$ a fixed countable subset $U$ of $W$, then the resulting countable general frame $\G$ will contain $U$. The details are provided in Theorem~\ref{LowSklWithSet}(1).
}
\end{remark}

\begin{remark}\label{EnlargingV0forL}
{\em
If we start the proof of Theorem~\ref{CGFP} with a general frame $\F$ such that $L=L(\F)$, then it is possible to perform the selection procedure in such a way that we obtain a countable general frame $\G$ with $L=L(\G)$. The details are provided in Theorem~\ref{LowSklWithSet}(2).
}
\end{remark}

\begin{remark}
{\em
Since descriptive frames provide adequate semantics, one may wish to introduce the \emph{countable descriptive frame property} (CDFP), which could be stated as follows: A normal modal logic $L$ has the CDFP provided every non-theorem of $L$ is refuted on a countable descriptive frame for $L$. We leave it as an open problem whether every normal modal logic has the CDFP.
}
\end{remark}

We now turn our attention to $\mathbf Q$. This is our first example of how we can use the CGFP to obtain some general completeness results about logics above {\bf S4}. In fact, we prove that a normal modal logic $L$ is a logic above {\bf S4} iff $L$ is the logic of a general space over $\mathbf Q$, which is equivalent to being the logic of a subalgebra of the closure algebra $\mathbf Q^+$. This is achieved by combining the CGFP with known results concerning interior images of $\mathbf Q$.

\begin{lemma}\label{GenTopFromGFAndIntMap}
Let $X$ and $Y$ be topological spaces and let $f:X\rightarrow Y$ be an onto interior map. For a general space $\mathfrak Y=(Y,\mathcal P)$ over $Y$, set $\mathcal Q=\{f^{-1}(A):A\in\mathcal P\}$. Then $\X=(X,\mathcal Q)$ is a general space over $X$ such that $L(\X)=L(\Y)$.
\end{lemma}

\begin{proof}
Since $f:X\to Y$ is an onto interior map, $f^{-1}:Y^+\to X^+$ is a closure algebra embedding, so the restriction of $f^{-1}$ to $\mathcal P$ is a closure algebra isomorphism from $\mathcal P$ onto $\mathcal Q$. Thus, $\X=(X,\mathcal Q)$ is a general space over $X$ such that $L(\X)=L(\Y)$.
\end{proof}

\begin{remark}
{\em
We will frequently use a special case of the lemma, when $Y$ is the Alexandroff space of an {\bf S4}-frame.
}
\end{remark}

Let $\F=(W,R)$ be an {\bf S4}-frame. We recall that $\F$ is {\em rooted} if there is $w\in W$ such that $W=R(w)$, and that such a $w$ is called a {\em root} of $\F$. A general frame $(W,R,\mathcal P)$ is rooted iff $(W,R)$ is rooted. We next show that the {\bf S4}-version of the Main Lemma in \cite[Lem.~3.1]{BLB12a} gives that each countable rooted {\bf S4}-frame is an interior image of $\mathbf Q$. The lemma is well-known in the finite case (see, e.g., \cite[Sec.~2]{BBCS06}).

\begin{lemma}\label{InteriorImageOfQ}
Each countable rooted {\bf S4}-frame is an interior image of $\mathbf Q$.
\end{lemma}

\begin{proof}
(Sketch) Let $\F=(W,R)$ be a countable rooted {\bf S4}-frame. We briefly describe the recursive construction from \cite{BLB12a}. Since $\F$ is reflexive, the construction yields a homeomorphic copy $X$ of $\mathbf Q$ (rather than of a subspace of $\mathbf Q$, as happens in \cite{BLB12a}) and an onto interior map $f:X\to W$.

Let $l$ be a (horizontal) line in the plane and let $P$ be the open lower half plane below $l$. For each $p\in P$, consider the right isosceles triangle in $P\cup l$ such that the vertex at the right angle is $p$ and the hypotenuse lies along $l$. Viewing the hypotenuse as a closed interval in $l$ gives a bijective correspondence between $P$ and the closed (non-trivial) intervals in $l$.

We start our construction with any fixed $p_0\in P$ together with its corresponding triangle. Orthogonally project $p_0$ to the point $l(p_0)$ in $l$. Using successive triangles we now build two sequences (in $l$) converging to $l(p_0)$ (one increasing and one decreasing). Figure~\ref{FigDefining f from X to F} demonstrates this recursive step in which these sequences are built (notice the orthogonal projection of the vertices into $l$). Since $\F$ is reflexive, this recursive process does not terminate (unlike the setting of \cite{BLB12a}). Let $X$ be the set of points in $l$ that are projections of vertices. Induce an ordering of $X$ by restricting the ordering of $l$ (which one may now wish to view as $\mathbf R$). Since $\F$ is reflexive, $X$ is a countable dense linear ordering without endpoints. By Cantor's theorem (see, e.g., \cite[p.~217,~Thm.~2]{KM76}), $X$ is order-isomorphic to $\mathbf Q$, and hence when equipped with the interval topology, $X$ is homeomorphic to $\mathbf Q$. We now define $f:X\to W$.

It is convenient to identify $l(p)$ in $X$ with the point $p\in P$. Set $f(p_0)$ to be a root of $\F$. Assuming that $f(p)=w$, the vertices of the `next' triangles (shown in the dashed box in the picture below) are mapped onto $R(w)$ so that for each $v\in R(w)$ the set $f^{-1}(v)$ is infinite. This can be achieved by utilizing that $R(w)$ is countable and that there is a sequence $\theta:\omega\to\omega$ such that $\theta^{-1}(n)$ is infinite for each $n\in\omega$. It follows from \cite[Lem.~3.1]{BLB12a} that $f:X\to W$ is an onto interior map. Since $X$ and $\mathbf Q$ are homeomorphic, we conclude that $\F$ is an interior image of $\mathbf Q$.
\begin{figure}[h]
\[
\begin{picture}(320,200)(0,0)
\put(42,122){\makebox(0,0){$\bullet$}}
\put(42,162){\makebox(0,0){{\tiny$\bullet$}}}
\put(102,142){\makebox(0,0){$\bullet$}}
\put(102,162){\makebox(0,0){{\tiny$\bullet$}}}
\put(132,152){\makebox(0,0){$\bullet$}}
\put(132,162){\makebox(0,0){{\tiny$\bullet$}}}
\put(162,162){\makebox(0,0){{\tiny$\bullet$}}}
\put(158,170){$l(p)$}
\put(162,2){\makebox(0,0){$\bullet$}}
\put(150,0){$p$}
\put(170,0){$\longrightarrow w$}
\put(192,152){\makebox(0,0){$\bullet$}}
\put(192,162){\makebox(0,0){{\tiny$\bullet$}}}
\put(222,142){\makebox(0,0){$\bullet$}}
\put(222,162){\makebox(0,0){{\tiny$\bullet$}}}
\put(282,122){\makebox(0,0){$\bullet$}}
\put(282,162){\makebox(0,0){{\tiny$\bullet$}}}
\put(144,153){$\ldots$}
\put(168,153){$\ldots$}
\put(164,162){\vector(-1,0){174}}
\put(-18,160){$l$}
\put(-18,80){$P$}
\put(164,162){\vector(1,0){170}}
\put(162,2){\line(-1,1){160}}
\put(162,2){\line(1,1){160}}
\put(42,122){\line(1,1){40}}
\put(102,142){\line(-1,1){20}}
\put(102,142){\line(1,1){20}}
\put(222,142){\line(-1,1){20}}
\put(222,142){\line(1,1){20}}
\put(282,122){\line(-1,1){40}}
\put(132,152){\line(-1,1){10}}
\put(132,152){\line(1,1){10}}
\put(192,152){\line(-1,1){10}}
\put(192,152){\line(1,1){10}}
\put(37,117){\dashbox{2}(250,40){}}
\put(295,115){$\longrightarrow R(w)$}
\end{picture}
\]
\caption{Defining $X$ and $f:X\to \F$}
\label{FigDefining f from X to F}
\end{figure}
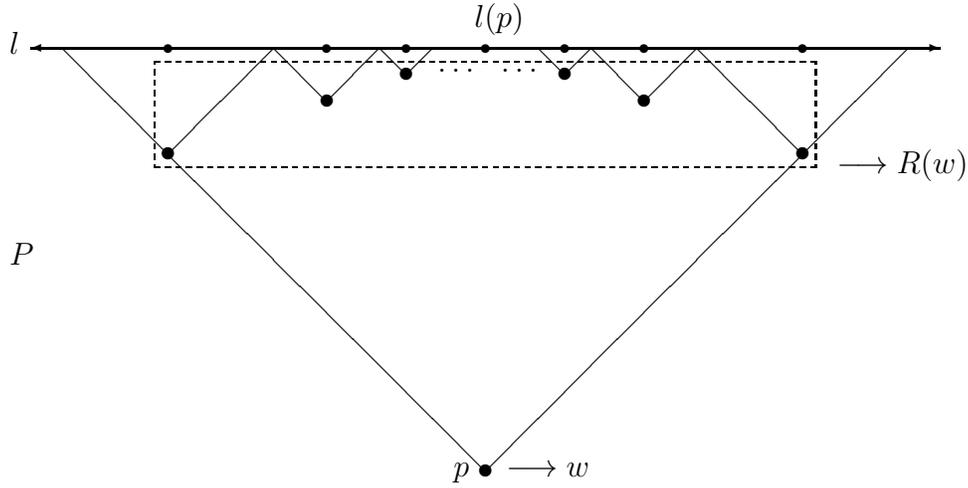
\end{proof}

\begin{lemma}\label{RefutingOneFormInGenTopQ}
Let $L$ be a logic above ${\bf S4}$. If $\varphi\notin L$, then there is a general space over $\mathbf Q$ validating $L$ and refuting $\varphi$.
\end{lemma}

\begin{proof}
It follows from the proof of Theorem~\ref{CGFP} that there is a countable rooted general frame $\mathfrak F=(W,R,\mathcal P)$ for $L$ that refutes $\varphi$. Lemma~\ref{InteriorImageOfQ} gives an onto interior map $f:\mathbf Q\rightarrow W$. Set $\mathcal S=\{f^{-1}(A):A\in\mathcal P\}$. By Lemma~\ref{GenTopFromGFAndIntMap}, $(\mathbf Q,\mathcal S)$ is a general space for $L$ that refutes $\varphi$.
\end{proof}

We are ready to prove our first general completeness result for logics above {\bf S4}. For a closure algebra $\mathfrak A$, let ${\bf S}(\mathfrak A)$ be the collection of all subalgebras of $\mathfrak A$.

\begin{theorem}\label{OneLogicOneGenTopOnQ}
Let $L$ be a normal modal logic. The following are equivalent.
\begin{enumerate}
\item $L$ is a logic above {\bf S4}.
\item There is a general space over $\mathbf Q$ whose logic is $L$.
\item There is $\mathfrak A\in{\bf S}(\mathbf Q^+)$ such that $L=L(\mathfrak A)$.
\end{enumerate}
\end{theorem}

\begin{proof}
(1)$\Rightarrow$(2): Let $\{\varphi_n:n\in\omega\}$ be an enumeration of the non-theorems of $L$. By Lemma~\ref{RefutingOneFormInGenTopQ}, for each $n\in\omega$, there is a general space over $\mathbf Q$ validating $L$ and refuting $\varphi_n$. Let $\X_n=(X_n,\mathcal P_n)$ be a copy of this general space, and without loss of generality we assume that $X_n\cap X_m=\varnothing$ whenever $n\ne m$. Let $\X=(X,\mathcal P)$ be the sum of the $\X_n$. Then $\X$ is a general space. As sums preserve validity, $\X$ validates $L$. Because $\varphi_n$ is refuted in $\X_n$, it is clear that $\varphi_n$ is refuted in the sum $\X$. Therefore, $L=L(\X)$. Since the countable sum of $\mathbf Q$ is homeomorphic to $\mathbf Q$, we have that $X$ is homeomorphic to $\mathbf Q$. Thus, up to homeomorphism, $\X$ is a general space over $\mathbf Q$. Consequently, $L$ is the logic of a general space over $\mathbf Q$.

(2)$\Rightarrow$(3): If $L=L(\mathbf Q,\mathcal P)$, then $L=L(\mathcal P)$ and $\mathcal P$ is a subalgebra of $\mathbf Q^+$.

(3)$\Rightarrow$(1): This is clear since a subalgebra of $\mathbf Q^+$ is a closure algebra and the logic of a closure algebra is a logic above {\bf S4}.
\end{proof}

A key ingredient in the proof of Theorem~\ref{OneLogicOneGenTopOnQ} is the interior mapping provided by Lemma~\ref{InteriorImageOfQ}. An alternative proof, sketched in Section~\ref{BinaryTreeSection}, can be realized by an interior map which factors through the infinite binary tree. With this in mind, we ask whether Theorem~\ref{OneLogicOneGenTopOnQ} holds for the infinite binary tree. Section~\ref{BinaryTreeSection} is dedicated to answering this question.

\section{Well-connected logics and completeness for general frames over $\T_2$}\label{BinaryTreeSection}

For each nonzero $\alpha\in\omega+1$, we view the infinite $\alpha$-ary tree $T_\alpha$ as the set of finite $\alpha$-valued sequences, including the empty sequence. Thus, if $\downn=\{m\in\omega\ : m\leq n\}$, then
$$
T_\alpha = \left\{a:S\rightarrow\alpha\ :\ S=\varnothing \mbox{ or } S=\downn \mbox{ for some } n\in\omega\right\}.
$$
We also consider the infinite $\alpha$-ary tree with limits $L_\alpha$ by setting
$$
L_\alpha = \left\{a:S\rightarrow\alpha\ :\ S=\varnothing, \ S=\downn \mbox{ for some } n\in\omega, \mbox{ or } S=\omega\right\}.
$$
That is, $L_\alpha = T_\alpha\cup\left\{a:\omega\rightarrow\alpha\right\}$. Define a partial order on $L_\alpha$ by
$$
a\le b \text{ iff } \mathrm{dom}(a)\subseteq\mathrm{dom}(b) \text{ and } a(n)=b(n) \text{ for all } n\in\mathrm{dom}(a).
$$
Since $T_\alpha\subseteq L_\alpha$, we also use $\le$ to denote the restriction of this order to $T_\alpha$. We let $\T_\alpha=(T_\alpha,\le)$ and $\LL_\alpha=(L_\alpha,\le)$. We call $\T_\alpha$ the {\em infinite $\alpha$-ary tree}, and we call $\LL_\alpha$ the {\em infinite $\alpha$-ary tree with limits}.

\begin{remark}
{\em
\begin{enumerate}
\item[]
\item The empty sequence, i.e.~the sequence whose domain is empty, is the root of both $\T_\alpha$ and $\LL_\alpha$.
\item In $\LL_\alpha$ each infinite sequence is a leaf.
\item $\T_\alpha$ has no leaves.
\item Each $T_\alpha$ is countable.
\item If $\alpha>1$, then $L_\alpha$ is uncountable.
\item $L_\alpha-T_\alpha$ consists of exactly the infinite $\alpha$-valued sequences.
\end{enumerate}
}
\end{remark}

In this section we are primarily interested in $\T_2$, although our results hold true for any $\alpha\ge 2$. Let $a\leq b$ in $\T_2$. Suppose that $\mathrm{dom}(b)$ has exactly one more element than $\mathrm{dom}(a)$. Call $b$ the {\em left child} of $a$ if the last occurring value in $b$ is $0$, and the {\em right child} of $a$ if the last occurring value in $b$ is $1$. In these cases, we write $b=l(a)$ and $b=r(a)$, respectively. We also put $l^0(a)=a$ and $l^{k+1}(a)=l(l^k(a))$ for $k\in\omega$, as well as $r^0(a)=a$ and $r^{k+1}(a)=r(r^k(a))$.

The next lemma is well known. The finite version of it was proved independently by Gabbay and van Benthem (see, e.g., \cite{gol:dio80}). The countable version of it can be found in Kremer \cite{Kre13}. We give our own proof of the lemma since the technique is useful in later considerations. It is based on the $t$-comb labeling of \cite[Sec.~4]{ABB03}. With careful unpacking, one may realize that our proof is a condensed version of Kremer's proof.

\begin{lemma}\label{ImagesOfT2}
Any countable rooted {\bf S4}-frame $\mathfrak F$ is a p-morphic image of $\T_2$.
\end{lemma}

\begin{proof}
Let $\mathfrak F=(W,R)$ be a countable rooted {\bf S4}-frame. For each $w\in W$ let $\theta_w:\omega\rightarrow R(w)$ be an onto map. Label the elements of $T_2$ as follows. Denote the label of $a\in T_2$ by $L(a)$. Label the root of $\T_2$ by a root of $\mathfrak F$. Suppose $L(a)=w$. For all $n\in\omega$ label $l^n(a)$ by $w$ and $r(l^n(a))$ by $\theta_w(n)$ provided such elements of $T_2$ are not yet labeled; see Figure~\ref{figCombWithLabels}.

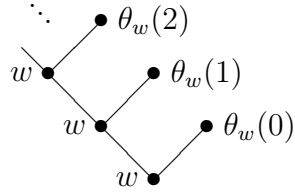
\begin{figure}[h]
\begin{center}
\begin{picture}(100,100)(0,0)
\put(70,0){\makebox(0,0){$\bullet$}}
\put(70,0){\line(-1,1){50}}
\put(70,0){\line(1,1){20}}
\put(90,20){\makebox(0,0){$\bullet$}}
\put(50,20){\makebox(0,0){$\bullet$}}
\put(50,20){\line(1,1){20}}
\put(70,40){\makebox(0,0){$\bullet$}}
\put(30,40){\makebox(0,0){$\bullet$}}
\put(30,40){\line(1,1){20}}
\put(50,60){\makebox(0,0){$\bullet$}}
\put(60,0){\makebox(0,0){$w$}}
\put(40,20){\makebox(0,0){$w$}}
\put(20,40){\makebox(0,0){$w$}}
\put(110,20){\makebox(0,0){$\theta_w(0)$}}
\put(90,40){\makebox(0,0){$\theta_w(1)$}}
\put(70,60){\makebox(0,0){$\theta_w(2)$}}
\put(30,60){\makebox(0,0){.}}
\put(27,63){\makebox(0,0){.}}
\put(24,66){\makebox(0,0){.}}
\end{picture}
\end{center}
\caption{Labeling scheme for a $t$-comb}
\label{figCombWithLabels}
\end{figure}

This labeling induces an onto p-morphism, namely $L:T_2\rightarrow W$. To see that $L$ is a p-morphism, observe that $L(l(a)),L(r(a))\in R(L(a))$ for each $a\in T_2$. Therefore, $a\le b$ in $\T_2$ implies $L(a)RL(b)$ in $\mathfrak F$. Suppose $L(a)Rw$. Then $w\in R(L(a))$ and there is $n\in\omega$ such that $\theta_{L(a)}(n)=w$. Thus, $a\le r(l^n(a))$ and $L(r(l^n(a)))=w$. This shows that $L$ is a p-morphism. That $L$ is onto is obvious because if $w$ is a root of $\mathfrak F$, then $W=R(w)$ and $\theta_w:\omega\to R(w)$ is onto. Consequently, $\mathfrak F$ is a p-morphic image of $\T_2$.
\end{proof}

\begin{remark}
{\em
As promised at the end of Section~\ref{SecRationals}, we give an alternative proof of Lemma~\ref{InteriorImageOfQ}. Let $\mathfrak F=(W,R)$ be a countable rooted {\bf S4}-frame. By Lemma~\ref{ImagesOfT2}, there is an onto p-morphism $f:T_2\rightarrow W$. By \cite[Claim 2.6]{BBCS06}, there is an onto interior map $g:\mathbf Q\rightarrow T_2$. Thus, the composition $f\circ g:\mathbf Q\rightarrow W$ is an onto interior map that factors through $T_2$.
}
\end{remark}

We now show that an analogue of Theorem~\ref{OneLogicOneGenTopOnQ} does not hold for $\T_2$. For this we recall the notion of a connected logic from \cite{BG11}. Let $\mathfrak A=(A,\Diamond)$ be a closure algebra. Call $a\in A$ {\em clopen} if $\Box a=a=\Diamond a$ (that is, $a$ is both open and closed). We say $\mathfrak A$ is {\em connected} provided the only clopen elements are $0$ and $1$, and that a logic $L$ above {\bf S4} is {\em connected} provided $L=L(\mathfrak A)$ for some connected closure algebra $\mathfrak A$.

For a topological space $X$, it is clear that $X^+$ is connected iff $X$ is connected. For an {\bf S4}-frame $\mathfrak F$, it is also well known that $\mathfrak F^+$ is connected iff $\mathfrak F$ is path-connected (see, e.g., \cite[Lem.~3.4]{BG11}). Because $\T_2$ is rooted, $\T_2$ is path-connected. Therefore, $\T_2^+$ is a connected closure algebra, and hence each subalgebra of $\T_2^+$ is also connected. Thus, the logic of any subalgebra of $\T_2^+$ is a connected logic (we will strengthen this result at the end of this section). Since there exist logics above {\bf S4} that are not connected \cite[p.~306]{BG11}, it follows that subalgebras of $\T_2^+$ do not give rise to all logics above {\bf S4}. Therefore, the direct analogue of Theorem~\ref{OneLogicOneGenTopOnQ} obtained by substituting $\T_2$ for $\mathbf Q$ does not hold. But there is a weaker analogue that does hold for $\T_2$.

\begin{lemma}\label{RefuteOneFormInGenFrmOnT2}
Let $L$ be a logic above {\bf S4}. If $\varphi\notin L$, then there is a general frame over $\T_2$ validating $L$ and refuting $\varphi$.
\end{lemma}

\begin{proof}
The proof is the same as the proof of Lemma~\ref{RefutingOneFormInGenTopQ}, with the only differences being $\mathbf Q$ should be replaced with $\T_2$ and Lemma~\ref{InteriorImageOfQ} should be replaced with Lemma~\ref{ImagesOfT2}.
\end{proof}

The following lemma is a weaker version for $\T_2$ of the fact that a countable sum of $\mathbf Q$ is homeomorphic to $\mathbf Q$.

\begin{lemma}\label{CountableDisUnionT2inItself}
A countable disjoint union of $\T_2$ is isomorphic to a generated subframe of $\T_2$.
\end{lemma}

\begin{proof}
Define $a_n:{\downarrow}n\to 2$ by
\[
a_n(k)=\left\{\begin{array}{ll}
0 & k < n,\\
1 & k=n.
\end{array}\right.
\]
Then $\{a_n:n\in\omega\}\subset T_2$. Clearly $\bigcup\nolimits_{n\in\omega}{\uparrow}a_n$ is a generated subframe of $\T_2$ and the family $\{{\uparrow}a_n:n\in\omega\}$ is pairwise disjoint; see Figure~\ref{figCombForDisjointUnion}. Furthermore, the generated subframe of $\T_2$ whose underlying set is ${\uparrow}a_n$ is isomorphic to $\T_2$. To see this observe that the following recursively defined function is a bijective p-morphism from $\T_2$ onto ${\uparrow}a_n$:
$$f(a)=\left\{\begin{array}{ll}
a_n & a\text{ is the root of }\T_2,\\
l(f(b)) & a=l(b),\\
r(f(b)) & a=r(b).
\end{array}\right.$$
\begin{figure}[h]
\begin{center}
\begin{picture}(100,100)(0,0)
\put(70,0){\makebox(0,0){$\bullet$}}
\put(70,0){\line(-1,1){50}}
\put(70,0){\line(1,1){20}}
\put(90,20){\makebox(0,0){$\bullet$}}
\put(94,12){\makebox(0,0){$a_0$}}
\put(50,20){\makebox(0,0){$\bullet$}}
\put(50,20){\line(1,1){20}}
\put(70,40){\makebox(0,0){$\bullet$}}
\put(74,32){\makebox(0,0){$a_1$}}
\put(30,40){\makebox(0,0){$\bullet$}}
\put(30,40){\line(1,1){20}}
\put(50,60){\makebox(0,0){$\bullet$}}
\put(54,52){\makebox(0,0){$a_2$}}
\put(90,20){\line(1,4){10}}
\put(90,20){\line(-1,4){10}}
\put(70,40){\line(1,4){10}}
\put(70,40){\line(-1,4){10}}
\put(50,60){\line(1,4){10}}
\put(50,60){\line(-1,4){10}}
\put(30,60){\makebox(0,0){.}}
\put(27,63){\makebox(0,0){.}}
\put(24,66){\makebox(0,0){.}}
\end{picture}
\end{center}
\caption{Depicting ${\uparrow}a_n$'s}\label{figCombForDisjointUnion}
\end{figure}
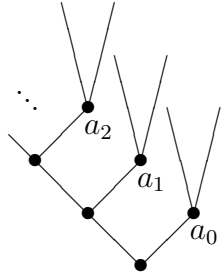
\end{proof}

Let $\mathfrak A$ be a closure algebra. We recall that ${\bf S}(\mathfrak A)$ is the collection of all subalgebras of $\mathfrak A$. We also let ${\bf H}(\mathfrak A)$ be the collection of all homomorphic images of $\mathfrak A$, and ${\bf SH}(\mathfrak A)$ be the collection of all subalgebras of homomorphic images of $\mathfrak A$.

\begin{theorem}\label{OneLogicOneGenSubInT2}
For a normal modal logic $L$, the following conditions are equivalent.
\begin{enumerate}
\item $L$ is a logic above {\bf S4}.
\item There is a general {\bf S4}-frame $\F=(W,R,\mathcal P)$ such that $(W,R)$ is a generated subframe of $\T_2$ and $L=L(\F)$.
\item There is a closure algebra $\mathfrak A\in{\bf SH}(\T_2^+)$ such that $L=L(\mathfrak A)$.
\end{enumerate}
\end{theorem}

\begin{proof}
(1)$\Rightarrow$(2): Let $L$ be a logic above {\bf S4}, and let $\{\varphi_n:n\in\omega\}$ be an enumeration of the non-theorems of $L$. By Lemma~\ref{RefuteOneFormInGenFrmOnT2}, for each $n\in\omega$, there is a general frame over $\T_2$ validating $L$ and refuting $\varphi_n$. Let $\F_n=(W_n,\mathcal P_n)$ be a copy of this general frame, and without loss of generality we assume that $W_n\cap W_m=\varnothing$ whenever $n\ne m$. Let $\F=(W,R,\mathcal P)$ be the disjoint union of the $\F_n$. Because disjoint unions of general frames preserve validity, $\F$ is a general frame for $L$, and clearly $\F$ refutes each $\varphi_n$. Thus, $L=L(\F)$, and by Lemma~\ref{CountableDisUnionT2inItself}, $(W,R)$ is isomorphic to a generated subframe of $\T_2$.

(2)$\Rightarrow$(3): Since $(W,R)$ is a generated subframe of $\T_2$, we have $(W,R)^+\in{\bf H}(\T_2^+)$. As $\mathcal P$ is a subalgebra of $(W,R)^+$, we have $\mathcal P\in{\bf SH}(\T_2^+)$. Finally, $L=L(\F)$ yields $L=L(\mathcal P)$.

(3)$\Rightarrow$(1): This is clear since $\mathfrak A$ is a closure algebra and $L=L(\mathfrak A)$.
\end{proof}

The next natural question is to characterize those logics above {\bf S4} which arise from subalgebras of $\T_2^+$. We recall \cite[Def.~1.10]{MT44} that a closure algebra $\mathfrak A=(A,\Diamond)$ is {\em well-connected} if $\Diamond a\wedge\Diamond b=0$ implies $a=0$ or $b=0$. Equivalently, $\Box a \vee \Box b = 1$ implies $a=1$ or $b=1$. It is easy to see that a well-connected closure algebra is connected, and that a subalgebra of a well-connected closure algebra is also well-connected.

For an example of a connected closure algebra that is not well-connected, let $\F=(W,R)$ be a finite {\bf S4}-frame. Then $\F^+$ is connected iff $\F$ is path-connected and $\F^+$ is well-connected iff $\F$ is rooted (see, e.g., \cite[Sec.~2]{BG05a}). So a finite path-connected $\F$ that is not rooted gives rise to a connected closure algebra that is not well-connected.

\begin{definition}
{\em
We call a logic $L$ above {\bf S4} {\em well-connected} if $L=L(\mathfrak A)$ for some well-connected closure algebra $\mathfrak A$.
}
\end{definition}

It is easy to see that if $\F=(W,R)$ is a rooted {\bf S4}-frame, then $\F^+$ is a well-connected closure algebra. Therefore, since $\T_2$ is rooted, it follows that $\T_2^+$ is well-connected. Thus, each $\mathfrak A\in{\bf S}(\T_2^+)$ is well-connected. This implies that $L(\mathfrak A)$ is a well-connected logic above {\bf S4} for each $\mathfrak A\in{\bf S}(\T_2^+)$. To prove the converse, we need the following lemma.

\begin{lemma}\label{revariablizing}
Let $\{\varphi_n:n\in\omega\}$ be a set of formulas, $\F=(W,R)$ be a frame, and $w\in W$. Suppose that for each $n\in\omega$ there is a valuation $\nu_n$ on $\F$ such that $w\notin\nu_n(\varphi_n)$. Then there is a single valuation $\nu$ on $\F$ such that $w\notin\nu\left(\widehat{\varphi_n}\right)$ for each $n\in\omega$, where $\widehat{\varphi_n}$ is obtained from $\varphi_n$ via substitution involving only propositional letters.
\end{lemma}

\begin{proof}
We build $\widehat{\varphi_n}$ so that distinct formulas in $\{\widehat{\varphi_n}:n\in\omega\}$ have no propositional letters in common. Let $\mathrm P_n$ be the set of propositional letters occurring in $\varphi_n$. Since the disjoint union of countably many finite sets is countably infinite, there is a bijection $\sigma:\bigcup\nolimits_{n\in\omega}\mathrm P_n\times\{\varphi_n\}\to\mathrm{Prop}$. Thus, $\sigma$ assigns each propositional letter $p$ in $\varphi_n$ to a new propositional letter so that no letter in $\varphi_n$ is assigned to the same letter, and no two occurrences of $p$ in distinct formulas are assigned to the same letter. We let $\widehat{\varphi_n}$ be the substitution instance of $\varphi_n$ obtained by substituting each occurrence of $p$ in $\varphi_n$ with $\sigma(p,\varphi_n)$. Then distinct formulas in $\{\widehat{\varphi_n}:n\in\omega\}$ have no propositional letters in common.

Define $\nu$ by $\nu(\sigma(p,\varphi_n))=\nu_n(p)$. Then for any $v\in W$ we have
$$
v\in\nu_n(\varphi_n)\text{ iff }v\in\nu\left(\widehat{\varphi_n}\right).
$$
In particular, $w\notin\nu\left(\widehat{\varphi_n}\right)$ for each $n\in\omega$.
\end{proof}

We are ready to prove the main result of this section.

\begin{theorem}\label{WellConLogAndT2}
Let $L$ be a logic above {\bf S4}. The following conditions are equivalent.
\begin{enumerate}
\item $L$ is well-connected.
\item $L$ is the logic of a general frame over $\T_2$.
\item $L=L(\mathfrak A)$ for some $\mathfrak A\in{\bf S}(\T_2^+)$.
\end{enumerate}
\end{theorem}

\begin{proof}
(1)$\Rightarrow$(2): Let $L$ be well-connected. Then $L=L(\mathfrak A)$ for some well-connected closure algebra $\mathfrak A=(A,\Diamond)$. Let $\F=(W,R,\mathcal P)$ be the dual descriptive frame of $\mathfrak A$. Then $L(\F)=L(\mathfrak A)=L$. Since $\mathfrak A$ is well-connected, $\F$ is rooted (see, e.g., \cite[Sec.~3]{Esa79}). Let $w$ be a root of $\F$. Suppose $\{\varphi_n:n\in\omega\}$ is an enumeration of the non-theorems of $L$. For each $n\in\omega$, there is a valuation $\nu_n$ on $\F$ refuting $\varphi_n$. Since $w$ is a root, $w\notin\nu_n(\Box\varphi_n)$. By Lemma~\ref{revariablizing}, there are a valuation $\nu$ on $\F$ and the set $\{\widehat{\Box\varphi_n}:n\in\omega\}$ such that $w\notin\nu\left(\widehat{\Box\varphi_n}\right)$ for each $n\in\omega$. By Theorem~\ref{CGFP}, there is a general frame $\G=(V,S,\mathcal Q)$ such that $\G$ is a frame for $L$, $V\subseteq W$ is countable and contains $w$, $S$ is the restriction of $R$ to $V$, $\mathcal Q=\{\mu(\varphi):\varphi\in\mathfrak{Form}\}$, where $\mu(p)=\nu(p)\cap V$ for each $p\in\mathrm{Prop}$, and $w\notin\mu\left(\widehat{\Box\varphi_n}\right)$ for each $n\in\omega$. For each propositional letter $p$ occurring in $\varphi_n$, set $\lambda(p)=\mu(\sigma(p,\varphi_n))$. Then $\lambda(\Box\varphi_n)=\mu\left(\widehat{\Box\varphi_n}\right)$, so $w\notin\lambda(\Box\varphi_n)$ since $w\notin\mu\left(\widehat{\Box\varphi_n}\right)$. Thus, each $\varphi_n$ is refuted on $\G$, so $L=L(\G)$. Since $(V,S)$ is a countable rooted {\bf S4}-frame, Lemma~\ref{ImagesOfT2} gives that there is an onto p-morphism $f:T_2\to V$. By Lemma~\ref{GenTopFromGFAndIntMap}, there is a general frame over $\T_2$ whose logic is $L$.

(2)$\Rightarrow$(3): Let $L=L(\T_2,\mathcal P)$. Then $L=L(\mathcal P)$ and $\mathcal P\in{\bf S}(\T_2^+)$.

(3)$\Rightarrow$(1): This is obvious since each $\mathfrak A\in{\bf S}(\T_2^+)$ is well-connected.
\end{proof}

\section{Completeness for general spaces over ${\bf L}_2$}\label{SecL2}

In this section we take a more careful look at the infinite binary tree with limits $\LL_2$, equip it with the Scott topology, denote the result by ${\bf L}_2$, and show that in the completeness results of Section~\ref{BinaryTreeSection}, $\T_2$ can be replaced by ${\bf L}_2$.

We begin by pointing out that $\LL_2$ is obtained from $\T_2$ by adding leaves, which we realize as limit points via multiple topologies, the first of which is the Scott topology for a directed complete partial order (DCPO). Recall (see, e.g., \cite{GHKLMS03}) that a poset is a DCPO if every directed subset has a sup, and that an upset $U$ in a DCPO is \emph{Scott open} provided for each directed set $S$, we have $S\cap U\ne\varnothing$ whenever $\sup(S)\in U$. The collection of Scott open sets forms the \emph{Scott topology}.

For each $\alpha$, it is easy to see that a directed set in $\LL_\alpha$ is a chain whose sup exists in $\LL_\alpha$. Therefore, each $\LL_\alpha$ is a DCPO. Moreover, since ${\downarrow}a$ is a finite chain for each $a\in T_\alpha$, we have that ${\uparrow}a$ is Scott open for each $a\in T_\alpha$, and so $\{{\uparrow}a:a\in T_\alpha\}$ forms a basis for the Scott topology $\tau$ on $L_\alpha$. We denote $(L_\alpha,\tau)$ by ${\bf L}_\alpha$. Kremer \cite{Kre13} proved that ${\bf S4}$ is strongly complete with respect to ${\bf L}_2$.

\begin{lemma}\label{CantorInL2}
The Cantor space ${\bf C}$ is homeomorphic to the subspace $L_2-T_2$ of ${\bf L}_2$.
\end{lemma}

\begin{proof}
It is well known that $\bf C$ is homeomorphic to the space whose underlying set $X$ consists of infinite sequences $s=\{s_n:n\in\omega\}$ in $\T_2$ such that $s_0$ is the root and $s_{n+1}$ is a child of $s_n$, and whose topology is generated by the basic open sets $B_s^n=\{t\in X:s_k=t_k \ \ \forall k\in\downn\}$ for $s\in X$ and $n\in\omega$. With each $a\in L_2-T_2$ we associate $s\in X$ as follows:
\begin{eqnarray*}
s_0&=&\varnothing\text{ (the sequence with empty domain; i.e.~the root),}\\
s_{n+1}&=&a|_{\downarrow n}\text{ (the restriction of } a\text{ to } {\downarrow}n).
\end{eqnarray*}
Then the correspondence $a\mapsto s$ is a well-defined bijection from $L_2-T_2$ to {\bf C}, under which the basic open of $L_2-T_2$ arising from the Scott open set ${\uparrow}(a|_{\downarrow n})$ corresponds to $B_s^n$. Thus, $L_2-T_2$ is homeomorphic to ${\bf C}$.
\end{proof}

Utilizing the technique similar to the one presented in Section~\ref{SecRationals}, it is convenient to embed $L_2$ in the lower half plane as shown in Figure~\ref{figL2inPlane}, where the closed intervals formed in constructing ${\bf C}$ are depicted at the top of Figure~\ref{figL2inPlane}. The elements of $T_2$ are realized as vertices of isosceles right triangles whose hypotenuse coincides with the closed intervals and whose other sides depict the relation $\le$. Projecting the picture onto the line with arrows gives a realization of $L_2$ as a subset of $\mathbf R$ by adding to ${\bf C}$ the midpoint of each open middle third that is removed in constructing ${\bf C}$.

\begin{figure}[h]
\begin{center}
\begin{picture}(270,200)
\put(135,0,){\line(1,1){135}}
\put(135,0){\makebox(0,0){$\bullet$}}
\put(135,0){\line(-1,1){135}}
\put(225,90){\makebox(0,0){$\bullet$}}
\put(225,90){\line(-1,1){45}}
\put(45,90){\makebox(0,0){$\bullet$}}
\put(45,90){\line(1,1){45}}
\put(15,120){\makebox(0,0){$\bullet$}}
\put(15,120){\line(1,1){15}}
\put(75,120){\makebox(0,0){$\bullet$}}
\put(75,120){\line(-1,1){15}}
\put(195,120){\makebox(0,0){$\bullet$}}
\put(195,120){\line(1,1){15}}
\put(255,120){\makebox(0,0){$\bullet$}}
\put(255,120){\line(-1,1){15}}
\put(135,135){\vector(-1,0){150}}
\put(135,135){\vector(1,0){150}}
\put(0,200){\line(1,0){270}}
\put(0,200){\makebox(0,0){[}}
\put(270,200){\makebox(0,0){]}}
\put(0,185){\line(1,0){90}}
\put(0,185){\makebox(0,0){[}}
\put(90,185){\makebox(0,0){]}}
\put(180,185){\line(1,0){90}}
\put(180,185){\makebox(0,0){[}}
\put(270,185){\makebox(0,0){]}}
\put(0,170){\line(1,0){30}}
\put(0,170){\makebox(0,0){[}}
\put(30,170){\makebox(0,0){]}}
\put(60,170){\line(1,0){30}}
\put(60,170){\makebox(0,0){[}}
\put(90,170){\makebox(0,0){]}}
\put(180,170){\line(1,0){30}}
\put(180,170){\makebox(0,0){[}}
\put(210,170){\makebox(0,0){]}}
\put(240,170){\line(1,0){30}}
\put(240,170){\makebox(0,0){[}}
\put(270,170){\makebox(0,0){]}}
\put(30,185){\makebox(0,0){(}}
\put(90,200){\makebox(0,0){(}}
\put(210,185){\makebox(0,0){(}}
\put(60,185){\makebox(0,0){)}}
\put(180,200){\makebox(0,0){)}}
\put(240,185){\makebox(0,0){)}}
\put(15,155){\makebox(0,0){$\vdots$}}
\put(75,155){\makebox(0,0){$\vdots$}}
\put(195,155){\makebox(0,0){$\vdots$}}
\put(255,155){\makebox(0,0){$\vdots$}}
\end{picture}
\end{center}
\caption{Embedding $\mathfrak L_2$ in the lower half plane}\label{figL2inPlane}
\end{figure}
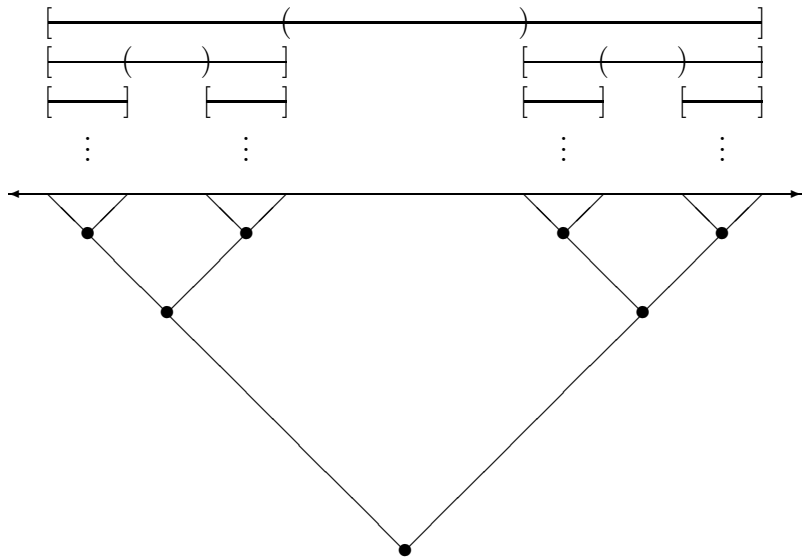

Since the Pelczynski compactification \cite{Pel65} of a countable discrete space $X$ is the compactification of $X$ whose remainder is homeomorphic to {\bf C}, we can realize $L_2$ as the Pelczynski compactification of $T_2$ viewed as a discrete space.

\begin{lemma}
Viewing $L_2$ as a subspace of $\mathbf R^2$ (or equivalently as a subspace of $\mathbf R$) gives the Pelczynski compactification of the discrete space $T_2$.
\end{lemma}

\begin{proof}
Clearly the image of $L_2$ under the above described embedding into the plane (or line) is closed and bounded. Therefore, if we give the image of $L_2$ the subspace topology, then it is a compact Hausdorff space. It is also clear that each point of $T_2$ is isolated in the image, and that the image of $L_2$ is the closure of the image of $T_2$. Thus, the image of $L_2$ is a compactification of the image of $T_2$, which is a countable discrete space. Finally, by Lemma~\ref{CantorInL2}, the remainder is homeomorphic to {\bf C}, so the image of $L_2$ is the Pelczynski compactification of the image of $T_2$.
\end{proof}

We denote this new topology on $L_2$ by $\tau_S$. Since the Pelczynski compactification of a countable discrete space is a Stone space, $(L_2,\tau_S)$ is a Stone space. It is clear from the figure that for $a\in T_2$, both ${\uparrow}a$ and $\{a\}$ are clopen in $(L_2,\tau_S)$, and that each open set of $(L_2,\tau_S)$ is the union of clopen sets of this form. As $\{a\}={\uparrow}a-\left({\uparrow}l(a)\cup{\uparrow}r(a)\right)$, we also see that the Boolean algebra $\mathcal P$ of all clopens of $(L_2,\tau_S)$ is generated by $\{{\uparrow}a:a\in T_2\}$, which is a basis for the Scott topology $\tau$ on $L_2$. Moreover, since for each $a\in T_2$, we have that ${\downarrow}a$ is finite and ${\downarrow}({\uparrow}a)={\uparrow}a\cup{\downarrow}a$, we see that $(\mathcal P,{\downarrow})$ is a closure algebra. Consequently, $(L_2,\le,\mathcal P)$ is a descriptive frame. Let $\tau_\le$ be the Alexandroff topology on $\LL_2$. Then, by \cite[Thm. 2.12]{BMM08}, $\tau=\tau_S\cap\tau_\le$.

\begin{remark}
{\em
As a result, we have several ways of thinking about $\LL_2$. The first way is to think about $\LL_2$ as a DCPO leading to the Scott topology $\tau$. The second way is a geometrically motivated approach that realizes $\LL_2$ as a subspace of $\mathbf R^2$, which gives the Stone topology $\tau_S$. The third way connects the first and second ways by realizing the Scott topology (which, by the way is the McKinsey-Tarski topology introduced in \cite{BMM08}) as the intersection of the Stone and Alexandroff topologies. Moreover, the Stone topology is the patch topology of the Scott topology. In fact, there is also a forth way of thinking about $\LL_2$. Let $D$ be the bounded distributive lattice generated by $\{{\uparrow}a:a\in T_2\}$. Then $(L_2,\tau_S,\le)$ is (up to homeomorphism) the Priestley space of $D$. Consequently, $(L_2,\tau_S,\le)$ is a Priestley order-compactification of the poset $(T_2,\le)$ (see \cite{BM11}).
}
\end{remark}

We are ready to prove completeness results that are similar to the ones proved in Section~\ref{BinaryTreeSection} but involve ${\bf L}_2$. For this we will take advantage of Kremer's theorem that $\T_2^+$ is isomorphic to a subalgebra of ${\bf L}_2^+$ \cite[Lem.~6.4]{Kre13}.

\begin{lemma}\label{RefuteOneFormInGenSpOnL_2}
Let $L$ be a logic above ${\bf S4}$. If $\varphi\notin L$, then there is a general space over ${\bf L}_2$ validating $L$ and refuting $\varphi$.
\end{lemma}

\begin{proof}
By Lemma~\ref{RefuteOneFormInGenFrmOnT2}, there is a general frame $(\T_2,\mathcal P)$ for $L$ refuting $\varphi$. By \cite[Lem.~6.4]{Kre13}, $\T_2^+$ is isomorphic to a subalgebra of ${\bf L}_2^+$, so $\mathcal P$ is isomorphic to some $\mathcal Q\in{\bf S}({\bf L}_2^+)$. Thus, there is a general space $({\bf L}_2,\mathcal Q)$ for $L$ refuting $\varphi$.
\end{proof}

\begin{theorem}\label{OneLogicOneGenSpaceOnL2Scott}
Let $L$ be a normal modal logic. The following are equivalent.
\begin{enumerate}
\item $L$ is a logic above {\bf S4}.
\item There is a general space over a Scott open subspace of ${\bf L}_2$ whose logic is $L$.
\item There is a closure algebra $\mathfrak A\in{\bf SH}({\bf L}_2^+)$ such that $L=L(\mathfrak A)$.
\end{enumerate}
\end{theorem}

\begin{proof}
(1)$\Rightarrow$(2): Let $\{\varphi_n:n\in\omega\}$ be an enumeration of the non-theorems of $L$. By Lemma~\ref{RefuteOneFormInGenSpOnL_2}, for each $n\in\omega$, there is a general space over ${\bf L}_2$ validating $L$ and refuting $\varphi_n$. Let $\X_n=(X_n,\mathcal P_n)$ be a copy of this general space, and without loss of generality we assume that $X_n\cap X_m=\varnothing$ whenever $n\ne m$. Thus, the sum $\X=(X,\mathcal P)$ of the general spaces $\X_n$ is a general space whose logic is $L$. The proof will be complete provided that $X$ is homeomorphic to a Scott open subspace of ${\bf L}_2$. Consider $a_n$ as in the proof of Lemma~\ref{CountableDisUnionT2inItself}. Then $\{{\uparrow}a_n:n\in\omega\}$ is a pairwise disjoint family of subsets of $L_2$ such that each ${\uparrow}a_n$ is isomorphic to $\LL_2$. To see the isomorphism, extend the map $f$ defined in the proof of Lemma~\ref{CountableDisUnionT2inItself} to $L_2-T_2$ by setting $f(a)=\sup\{f(a|_{\downarrow n}):n\in\omega\}$. Since $a_n\in T_2$, we see that $\bigcup_{n\in\omega}{\uparrow}a_n$ is Scott open in ${\bf L}_2$. As $X$ is homeomorphic to $\bigcup_{n\in\omega}{\uparrow}a_n$, we conclude that $X$ is homeomorphic to a Scott open subspace of ${\bf L}_2$, thus finishing the proof.

(2)$\Rightarrow$(3): Suppose $L=L(X,\mathcal P)$, where $X$ is a Scott open subspace of ${\bf L}_2$. Then $L=L(\mathcal P)$, $\mathcal P\in{\bf S}(X^+)$, and $X^+\in{\bf H}({\bf L}_2^+)$. Thus, there is a closure algebra $\mathcal P\in{\bf SH}({\bf L}_2^+)$ such that $L=L(\mathcal P)$.

(3)$\Rightarrow$(1): This is obvious since $\mathfrak A$ is a closure algebra.
\end{proof}

\begin{theorem}
Let $L$ be a logic above {\bf S4}. The following are equivalent.
\begin{enumerate}
\item $L$ is well-connected.
\item There is a general space over ${\bf L}_2$ whose logic is $L$.
\item $L=L(\mathfrak A)$ for some $\mathfrak A\in{\bf S}({\bf L}_2^+)$.
\end{enumerate}
\end{theorem}

\begin{proof}
(1)$\Rightarrow$(3): By Theorem~\ref{WellConLogAndT2}, $L=L(\mathfrak B)$ for some $\mathfrak B\in{\bf S}(\T_2^+)$. By \cite[Lem.~6.4]{Kre13}, $\mathfrak B$ is isomorphic to a subalgebra $\mathfrak A$ of ${\bf L}_2^+$. Thus, $L=L(\mathfrak B)=L(\mathfrak A)$.

(3)$\Rightarrow$(1): Since ${\bf L}_2^+$ is a well-connected closure algebra, it follows that every $\mathfrak A\in{\bf S}({\bf L}_2^+)$ is also well-connected. Thus, $L=L(\mathfrak A)$ is a well-connected logic.

Consequently, (1) and (3) are equivalent, and obviously (2) and (3) are equivalent.
\end{proof}

\section{Completeness for general spaces over $\mathbf C$}\label{CantorSection}

The key ingredient in proving that each logic above {\bf S4} is the logic of a general space over $\mathbf Q$ is that each countable rooted {\bf S4}-frame is an interior image of $\mathbf Q$. This is no longer true if we replace $\mathbf Q$ by the Cantor space $\mathbf C$ or the real line $\mathbf R$ \cite[Sec.~6]{BG02}. In this section we show that nevertheless there is an interior map from $\mathbf C$ onto ${\bf L}_2$, and utilize this fact to prove that each logic above {\bf S4} is the logic of some general space over $\mathbf C$.

In fact, we prove that ${\bf L}_\alpha$ is an interior image of ${\bf C}$ for each nonzero $\alpha\in\omega+1$. This we do by first constructing an onto interior map $f:L_2\rightarrow L_\alpha$. Then restricting $f$ to $L_2-T_2$ and applying Lemma~\ref{CantorInL2} realizes each $\mathbf L_\alpha$ as an interior image of ${\bf C}$. That ${\bf L}_2$ is an interior image of $\mathbf C$ also follows from Kremer's result \cite[Lem 8.1]{Kre13} that $\mathbf L_2$ is an interior image of any complete dense-in-itself metric space. However, our proof is different. Our approach utilizes the way that ${\bf C}$ sits inside the DCPO structure of $\LL_2$ and the aforementioned $f:L_2\to L_\alpha$ is defined by utilizing the supremum of directed sets. Whereas Kremer's method decomposes ${\bf C}$ (or any complete dense-in-itself metric space) into equivalence classes that are indexed by $L_2$ so that mapping each point in an equivalence class to the corresponding index gives an interior map onto ${\bf L}_2$.

We use the proof of Lemma~\ref{ImagesOfT2} to label the nodes of $\T_2$ by the nodes of $\T_\alpha$. Recall that we denote the labeling of $a\in T_2$ by $L(a)$, and that for each $b\in T_\alpha$ we have an onto map $\theta_b:\omega\rightarrow {\uparrow}b$. Then the root $\mathrm{r}_2$ of $\T_2$ is labeled by the root $\mathrm{r}_\alpha$ of $\T_\alpha$, and we write $L(\mathrm{r}_2)=\mathrm{r}_\alpha$. Also, if $L(a)$ is defined, then for $n\in\omega$, we have $L(l^n(a))=L(a)$ and $L(r(l^n(a)))=\theta_{L(a)}(n)$. Therefore, for each $a\in L_2-T_2$, the sequence $\{L(a|_{\downarrow n}):n\in\omega\}$ is increasing in $\T_\alpha$, and hence also increasing in the DCPO $\LL_\alpha$. Set
$$
f(a)=\left\{\begin{array}{ll}
L(a)&\mbox{ if }a\in T_2,\\
\sup\{L(a|_{\downarrow n}):n\in\omega\}&\mbox{ if }a\in L_2-T_2.
\end{array}\right.
$$
Then $f$ is a well-defined map from $L_2$ to $L_\alpha$.

\begin{lemma}\label{LalphaOntoL2}
$f$ is an interior map from ${\bf L}_2$ onto ${\bf L}_\alpha$.
\end{lemma}

\begin{proof}
The proof consists of three claims.

\medskip

\noindent{\bf Claim 1:} $f$ is open.

\smallskip

\noindent{\em Proof:} We show $f({\uparrow}a)={\uparrow}f(a)$ for each $a\in T_2$. Let $b\in {\uparrow}a$. If $b\in T_2$, then an inductive argument based on the labeling scheme gives $f(a)=L(a) \le L(b)=f(b)$. If $b\in L_2-T_2$, then $b|_{\mathrm{dom}(a)}=a$, and so $f(a)=L(a)=L(b|_{\mathrm{dom}(a)}) \le \sup\{L(b|_{\downarrow n})\}=f(b)$. Therefore, $f({\uparrow}a)\subseteq {\uparrow}f(a)$.

Conversely let $b\in {\uparrow}f(a)$. Then $L(a)=f(a) \le b$. If $b\in T_\alpha$, then $\theta_{L(a)}(n)=b$ for some $n\in\omega$ and $r(l^n(a))\in {\uparrow}a$ with
$$
f(r(l^n(a)))=L(r(l^n(a)))=\theta_{L(a)}(n)=b.
$$
Suppose $b\in L_\alpha-T_\alpha$. We build an increasing sequence $\{a_n\}$ such that $a_n\in {\uparrow}a\cap T_2$ for each $n\in\omega$ and $c=\sup\{a_n\}\in L_2-T_2$ satisfies $f(c)=b$. Let $a_0=a$. Set $b_0=f(a)$. Let $b_{n+1}\in{\downarrow}b$ be such that
$$
\mathrm{dom}(b_{n+1})=\left\{\begin{array}{ll}
  {\downarrow}(m+1) & \text{if } \mathrm{dom}(b_n)={\downarrow}m,\\
  {\downarrow}0     & \text{if } \mathrm{dom}(b_n)=\varnothing.
\end{array}\right.
$$
Then $b_{n}=b|_{\mathrm{dom}(b_{n})}$ and $b_n\in T_\alpha$ for each $n\in\omega$. Furthermore, $\sup\{b_n\}=b$ and
$$
L(a)=b_0 \le b_n \le b_{n+1}.
$$
There is $m_n$ such that $\theta_{b_n}(m_n)=b_{n+1}$. Let $a_{n+1}=r(l^{m_n}(a_n))$. Clearly $L(a_0)=L(a)=f(a)=b_0$, and assuming $L(a_n)=b_n$ it follows that
$$
L(a_{n+1})=L(r(l^{m_n}(a_n)))=\theta_{L(a_n)}(m_n)=\theta_{b_n}(m_n)=b_{n+1}.
$$
Thus, for $c=\sup\{a_n\}$, we have
$$
b=\sup\{b_n\}=\sup\{L(a_n)\} \le \sup\{L(c|_{\downarrow n})\}=f(c),
$$
giving $f(c)=b$ since $b$ is a leaf of $\LL_\alpha$. As $c\in{\uparrow}a$, we have shown $f({\uparrow}a)\supseteq {\uparrow}f(a)$.

Lastly since each $a\in T_2$ is labeled by an element of $T_\alpha$, we have that $f(a)\in T_\alpha$ whenever $a\in T_2$, giving that $f({\uparrow}a)={\uparrow}f(a)\in \tau$. Thus, $f$ sends basic opens of ${\bf L}_2$ to basic opens of ${\bf L}_\alpha$, hence $f$ is open.

\medskip

\noindent{\bf Claim 2:} $f$ is onto.

\smallskip

\noindent{\em Proof:} $f(L_2)=f({\uparrow}\mathrm{r}_2)={\uparrow}f(\mathrm{r}_2)={\uparrow}L(\mathrm{r}_2)={\uparrow}\mathrm{r}_\alpha=L_\alpha$.

\medskip

\noindent{\bf Claim 3:} $f$ is continuous.

\smallskip

\noindent{\em Proof:} We show $f^{-1}({\uparrow}b)=\bigcup\{{\uparrow}a:b\le L(a)\}$ for each $b\in T_\alpha$. Let $c\in f^{-1}({\uparrow}b)$. Then $b \le f(c)$. If $c\in T_2$, then $b \le f(c)=L(c)$ and $c\in {\uparrow}c$, giving that $c\in\bigcup\{{\uparrow}a:b\le L(a)\}$. Suppose $c\in L_2-T_2$. Then $\sup\{L(c|_{\downarrow n})\}=f(c)\in {\uparrow}b$. Since $b\in T_\alpha$, we have ${\uparrow}b$ is Scott open, so there is $n\in\omega$ such that $f(c|_{\downarrow n})=L(c|_{\downarrow n})\in {\uparrow}b$. Therefore, $b \le L(c|_{\downarrow n})$ and $c\in {\uparrow}(c|_{\downarrow n})$. Thus, $c\in \bigcup\{{\uparrow}a:b\le L(a)\}$, showing that $f^{-1}({\uparrow}b)\subseteq\bigcup\{{\uparrow}a:b\le L(a)\}$.

Conversely, let $c\in\bigcup\{{\uparrow}a:b\le L(a)\}$. Then there is $a\in T_2$ such that $b \le L(a)$ and $c\in {\uparrow}a$. Therefore,
$$
f(c)\in f({\uparrow}a)={\uparrow}f(a)={\uparrow}L(a)\subseteq {\uparrow}b.
$$
Thus, $c\in f^{-1}({\uparrow}b)$, giving $f^{-1}({\uparrow}b)\supseteq\bigcup\{{\uparrow}a:b\le L(a)\}$. This proves that $f$ is continuous.
\end{proof}

\begin{theorem}\label{LalphaInteriorImgCantor}
For each nonzero $\alpha\in\omega+1$, the space ${\bf L}_\alpha$ is an interior image of ${\bf C}$.
\end{theorem}

\begin{proof}
By Lemma~\ref{CantorInL2}, ${\bf C}$ is homeomorphic to the subspace $L_2-T_2$ of ${\bf L}_2$. So it is enough to show that $g=f|_{L_2-T_2}$ is an onto interior map, where $f:L_2\to L_\alpha$ is the map of Lemma~\ref{LalphaOntoL2}. Clearly $g$ is continuous since it is the restriction of a continuous map. That $g$ is onto and open follows from the next claim since $f({\uparrow}a)={\uparrow}f(a)$ for each $a\in T_2$.

\medskip

\noindent{\bf Claim:} For each $a\in T_2$, we have $g\left({\uparrow}a\cap (L_2-T_2)\right)=f({\uparrow}a)$.

\smallskip

\noindent{\em Proof:} Clearly $g\left({\uparrow}a\cap (L_2-T_2)\right)=f\left({\uparrow}a\cap (L_2-T_2)\right)\subseteq f({\uparrow}a)$. Let $b\in f({\uparrow}a)$. Then there is $c\in {\uparrow}a$ such that $f(c)=b$. If $c\in L_2-T_2$, then there is nothing to prove. Suppose $c\in T_2$. Define $d\in L_2-T_2$ by $d(n)=c(n)$ when $n\in\mathrm{dom}(c)$ and $d(n)=0$ otherwise. So $d$ is the limit of the sequence $\{l^n(c):n\in\omega\}$ of the left ancestors of $c$. Then $d\in {\uparrow}a$ and since $L(l^n(c))=L(c)$, we have
$$
g(d)=f(d)=\sup\{L(d|_{\downarrow n})\}=\sup\{L(l^n(c))\}=\sup\{L(c)\}=L(c)=f(c)=b.
$$
Thus, $g\left({\uparrow}a\cap (L_2-T_2)\right)\supseteq f({\uparrow}a)$, and equality follows.
\end{proof}

As an immediate consequence, we obtain:

\begin{corollary}\label{L2InteriorImgCantor}
The space ${\bf L}_2$ is an interior image of ${\bf C}$.
\end{corollary}

In order to prove the main result of this section, we need the following lemma.

\begin{lemma}\label{RefuteOneFormInGenSpOnC}
Let $L$ be a logic above ${\bf S4}$. If $\varphi\notin L$, then there is a general space over $\mathbf C$ validating $L$ and refuting $\varphi$.
\end{lemma}

\begin{proof}
By Lemma~\ref{RefuteOneFormInGenSpOnL_2}, there is a general space $({\bf L}_2,\mathcal P)$ for $L$ refuting $\varphi$. By Corollary~\ref{L2InteriorImgCantor}, ${\bf L}_2$ is an interior image of $\mathbf C$, so ${\bf L}_2^+$ is isomorphic to a subalgebra of $\mathbf C^+$. Therefore, $\mathcal P$ is isomorphic to some $\mathcal Q\in{\bf S}(\mathbf C^+)$. Thus, there is a general space $(\mathbf C,\mathcal Q)$ for $L$ refuting $\varphi$.
\end{proof}

We are ready to prove the main result of this section.

\begin{theorem}\label{CompletenessForAnyLwrtCantor}
Let $L$ be a normal modal logic. The following conditions are equivalent.
\begin{enumerate}
\item $L$ is a logic above {\bf S4}.
\item $L$ is the logic of a general space over ${\bf C}$.
\item $L=L(\mathfrak A)$ for some $\mathfrak A\in{\bf S}({\bf C}^+)$.
\end{enumerate}
\end{theorem}

\begin{proof}
(1)$\Rightarrow$(2): Suppose $L$ is a logic above {\bf S4}. Let $\{\varphi_n:n\in\omega\}$ be an enumeration of the non-theorems of $L$. By Lemma~\ref{RefuteOneFormInGenSpOnC}, for each $n\in\omega$, there is a general space over $\mathbf C$ validating $L$ and refuting $\varphi_n$. Let $\X_n=(X_n,\mathcal P_n)$ be a copy of this general space, and without loss of generality we assume that $X_n\cap X_m=\varnothing$ whenever $n\not= m$. Let $\X=(X,\mathcal P)$ be the sum of the $\X_n$. Then $L(\X)=\bigcap_{n\in\omega}L(\X_n)=L$. Although $X$ is not homeomorphic to $\mathbf C$, the one-point compactification of $X$ is homeomorphic to $\mathbf C$ (see, e.g., \cite[Lem.~7.2]{BG11}).\footnote{We may realize $X$ geometrically as $\bigcup\nolimits_{n\in\omega}X_n$, where $X_n=\mathbf C\cap\left[\frac{2}{3^{n+1}},\frac{1}{3^n}\right]$. Note each $X_n$ is the portion of ${\bf C}$ that is contained in the right closed third of the iteration of constructing ${\bf C}$ as the `leftovers' of removing open middle `thirds' and hence each $X_n$ is homeomorphic to ${\bf C}$. The only point of ${\bf C}$ not in $X$ is $0$, which is clearly a limit point of $X$ as viewed as a subset of $\mathbf R$ (depicted below). So it is intuitively clear that the one-point compactification of the sum of $\omega$ copies of ${\bf C}$ is homeomorphic to ${\bf C}$.
\begin{center}
  \begin{picture}(270,15)(0,0)
    \put(270,5){\makebox(0,0){$]$}}
    \put(270,-10){\makebox(0,0){$1$}}
    \put(180,5){\makebox(0,0){$[$}}
    \put(180,-10){\makebox(0,0){$\frac23$}}
    \put(225,15){\makebox(0,0){$X_0$}}
    \put(180,5){\dashbox(90,0){}}
    \put(90,5){\makebox(0,0){$]$}}
    \put(90,-10){\makebox(0,0){$\frac13$}}
    \put(60,5){\makebox(0,0){$[$}}
    \put(60,-10){\makebox(0,0){$\frac29$}}
    \put(75,15){\makebox(0,0){$X_1$}}
    \put(60,5){\dashbox(30,0){}}
    \put(30,5){\makebox(0,0){$]$}}
    \put(30,-10){\makebox(0,0){$\frac19$}}
    \put(20,5){\makebox(0,0){$[$}}
    \put(20,-10){\makebox(0,0){$\frac2{27}$}}
    \put(25,15){\makebox(0,0){$X_2$}}
    \put(20,5){\dashbox(10,0){}}
    \put(10,5){\makebox(0,0){$\dots$}}
    \put(0,5){\makebox(0,0){$\bullet$}}
    \put(0,-10){\makebox(0,0){$0$}}
  \end{picture}
\end{center}
}

Let $\alpha X=X\cup\{\infty\}$ be the one-point compactification of $X$. Then $\alpha X$ is homeomorphic to ${\bf C}$. Define $\mathcal Q$ on $\alpha X$ by $A\in\mathcal Q$ iff $A\cap X_n\in\mathcal P_n$ and either $\infty\notin A$ and $\{n\in\omega:A\cap X_n\ne\varnothing\}$ is finite or $\infty\in A$ and $\{n\in\omega:A\cap X_n\ne X_n\}$ is finite.

\medskip

\noindent{\bf Claim 1:} $\mathcal Q$ is a closure algebra.

\smallskip

\noindent{\em Proof:} Let $A,B\in \mathcal Q$. Then $A\cap X_n,B\cap X_n\in \mathcal P_n$ for each $n\in\omega$. Clearly we have $(A\cap B)\cap X_n=(A\cap X_n)\cap (B\cap X_n)\in \mathcal P_n$. Suppose $\infty\not\in A$ or $\infty\not\in B$. Then $\infty\not\in A\cap B$ and
$$
\{n:(A\cap B)\cap X_n\ne\varnothing\}\subseteq\{n:A\cap X_n\ne\varnothing\}\cap\{n:B\cap X_n\ne\varnothing\}
$$
is finite since either $\{n:A\cap X_n\ne\varnothing\}$ is finite or $\{n:B\cap X_n\ne\varnothing\}$ is finite. Suppose $\infty\in A$ and $\infty \in B$. Then $\infty\in A\cap B$ and
$$
\{n:(A\cap B)\cap X_n\ne X_n\}\subseteq\{n:A\cap X_n\ne X_n\}\cup\{n:B\cap X_n\ne X_n\}
$$
is finite since both $\{n:A\cap X_n\ne X_n\}$ and $\{n:B\cap X_n\ne X_n\}$ are finite. Thus, $\mathcal Q$ is closed under $\cap$.

For complement, we clearly have $(-A)\cap X_n=X_n-(A\cap X_n)\in\mathcal P_n$. For every $C\in\mathcal Q$, we have
$$
\{n:(-C)\cap X_n\ne X_n\}=\{n:C\cap X_n\ne\varnothing\}.
$$
So if $\infty\not\in A$, then $\infty\in-A$ and $\{n:(-A)\cap X_n\ne X_n\}$ is finite since $\{n:A\cap X_n\ne\varnothing\}$ is finite. On the other hand, if $\infty\in A$, then $\infty\not\in -A$ and $\{n:(-A)\cap X_n\ne\varnothing\}$ is finite as $\{n:A\cap X_n\ne X_n\}$ is finite. Thus, $\mathcal Q$ is closed under complement.

Let {\bf c} be closure in $\alpha X$. It is left to show that $\mathcal Q$ is closed under {\bf c}. For $A\subseteq\alpha X$, we show that $\mathbf c(A)\cap X_n=\mathbf c_n(A\cap X_n)$, where $\mathbf c_n$ is closure in $X_n$. As $\mathbf c_n(A\cap X_n)=\mathbf c(A\cap X_n)\cap X_n$, one inclusion is clear. For the other inclusion, let $x\in\mathbf c(A)\cap X_n$ and let $U$ be an open neighborhood of $x$ in $X_n$. Since $X_n$ is a clopen subset of $\alpha X$, we see that $U$ is open in $\alpha X$. As $x\in\mathbf c(A)$, we have $A\cap U\ne\varnothing$. Therefore, $(A\cap X_n)\cap U=A\cap(U\cap X_n)=A\cap U\ne\varnothing$, and so $x\in\mathbf c_n(A\cap X_n)$. Now, since $\mathbf c_n(A\cap X_n)\in\mathcal P_n$, we see that $\mathbf c(A)\cap X_n\in\mathcal P_n$. Suppose $\infty\not\in A$. Then $\{n:\mathbf c(A)\cap X_n\ne\varnothing\}=\{n:A\cap X_n\ne\varnothing\}$ is finite, and $\infty\not\in\mathbf c(A)$ because $\{\infty\}\cup\bigcup_{\{n:A\cap X_n=\varnothing\}}X_n$ is an open neighborhood of $\infty$ disjoint from $A$. Suppose $\infty\in A$. Then $\{n:A\cap X_n\ne X_n\}$ is finite. Clearly $\infty\in\mathbf c(A)$. Since $\{n:\mathbf c(A)\cap X_n\ne X_n\}\subseteq\{n:A\cap X_n\ne X_n\}$, it follows that $\{n:\mathbf c(A)\cap X_n\ne X_n\}$ is finite. Thus, $\mathbf c(A)\in\mathcal Q$, and hence $\mathcal Q$ is a closure algebra.

\medskip

\noindent{\bf Claim 2:} $\mathcal Q$ is isomorphic to a subalgebra of the closure algebra $\mathcal P$.

\smallskip

\noindent{\em Proof:} Define $\eta:\mathcal Q\to\mathcal P$ by $\eta(A)=A\cap X=A-\{\infty\}$. Note that $\eta$ is the identity map when $\infty\not\in A$. Clearly $\eta$ is well defined. Let $A,B\in \mathcal Q$. Then
$$
\eta(A\cap B)=(A\cap B)\cap X=(A\cap X)\cap(B\cap X)=\eta(A)\cap\eta(B).
$$
Moreover,
$$
\eta(\alpha X-A)=(\alpha X-A)\cap X=X-A=X-(A\cap X)=X-\eta(A).$$
Therefore, $\eta$ is a Boolean homomorphism. We recall that {\bf c} is closure in $\alpha X$. Let ${\bf c}_X$ be closure in $X$. If $\infty\not\in A$, then $A\subseteq X$ and ${\bf c}(A)={\bf c}_X(A)$, so
$$
\eta(\mathbf c(A))=\eta(\mathbf c_X(A))=\mathbf c_X(A)=\mathbf c_X(A\cap X)=\mathbf c_X\eta(A).
$$
Suppose $\infty\in A$. Then $\mathbf c(A)=\mathbf c_X(A\cap X)\cup\{\infty\}$. Therefore,
$$
\eta(\mathbf c(A))=\eta\left(\mathbf c_X(A\cap X)\cup\{\infty\}\right)=\mathbf c_X(A\cap X)=\mathbf c_X(\eta(A)).
$$
Thus, $\eta$ is a closure algebra homomorphism.

To see that $\eta$ is an embedding, let $\eta(A)=X$. Then $\{n:A\cap X_n\ne X_n\}=\varnothing$, so $\infty\in A$, and hence $A=\alpha X$. Thus, $\eta:\mathcal Q\to\mathcal P$ is a closure algebra embedding, and so $\mathcal Q$ is isomorphic to a subalgebra of $\mathcal P$.

\medskip

Now, since $\mathcal P$ validates $L$, we have that $\mathcal Q$ validates $L$. Furthermore, since $\alpha_n:\mathcal Q\to\mathcal P_n$ given by $\alpha_n(A)=A\cap X_n$ is an onto closure algebra homomorphism, each $\mathcal P_n$ is a homomorphic image of $\mathcal Q$. Thus, $\mathcal Q$ refutes each $\varphi_n$, and hence $L=L(\mathcal Q)$. Consequently, $L$ is the logic of the general space $(\alpha X,\mathcal Q)$, and as $\alpha X$ is homeomorphic to ${\bf C}$, we conclude that $L$ is the logic of a general space over {\bf C}.

(2)$\Rightarrow$(3): If $L=L(\mathbf C,\mathcal Q)$, then $L=L(\mathcal Q)$ and $\mathcal Q\in{\bf S}({\bf C}^+)$.

(3)$\Rightarrow$(1): This is obvious since $L$ is the logic of a closure algebra.
\end{proof}

\begin{remark}
{\em
The closure algebra $\mathcal Q$ constructed in the proof of Theorem~\ref{CompletenessForAnyLwrtCantor} is a weak product \cite[Appendix, \S 3]{Esa85} of the closure algebras $\mathcal P_n$.
}
\end{remark}

Most of the remainder of the paper is dedicated to logics associated with the real line $\mathbf R$.

\section{Completeness for general spaces over open subspaces of $\mathbf R$}\label{SecOpenSubspacesOfRealLine}

As we have seen, general spaces over both $\mathbf Q$ and $\mathbf C$ characterize all logics above {\bf S4}. This is no longer true for $\mathfrak T_2$ and ${\bf L}_2$. In fact, general frames over $\mathfrak T_2$ and general spaces over ${\bf L}_2$ characterize all well-connected logics above {\bf S4}, and in order to characterize all logics above {\bf S4}, we need to work with general frames over generated subframes of $\mathfrak T_2$ or general spaces over open subspaces of ${\bf L}_2$. In this section we show that a similar result is also true for the reals. Namely, we prove that a normal modal logic is a logic above {\bf S4} iff it is the logic of a general space over an open subspace of $\mathbf R$. In the next section we address the logics above {\bf S4} that arise from general spaces over $\mathbf R$ and show that each connected logic arises this way.

We recall that in proving that a logic $L$ above {\bf S4} is the logic of a general space over $\mathbf Q$, we enumerated all the non-theorems of $L$ as $\{\varphi_n:n\in\omega\}$, used the CGFP to find a countable rooted general {\bf S4}-frame $\mathfrak F_n=(W_n,R_n,\mathcal P_n)$ for $L$ that refuted $\varphi_n$, and obtained each $\mathfrak F_n$ as an interior image of $\mathbf Q$ via an onto interior map $f_n:\mathbf Q\to\mathfrak F_n$. We then used $f_n^{-1}$ to obtain $\mathcal Q_n\in{\bf S}(\mathbf Q^+)$ isomorphic to $\mathcal P_n$, thus producing a general space $(\mathbf Q,\mathcal Q_n)$. Finally, we took the sum of disjoint copies of the general spaces $(\mathbf Q,\mathcal Q_n)$ to obtain a general space whose underlying topological space was homeomorphic to $\mathbf Q$, and whose logic was indeed $L$.

What can go wrong with this technique when we switch to the reals? One obvious obstacle is that the sum of countably many copies of $\mathbf R$ is no longer homeomorphic to $\mathbf R$ because it is no longer connected. However, even before this `summation' stage, we have no guarantee that a rooted countable {\bf S4}-frame is an interior image of $\mathbf R$. Quite the contrary, it is a consequence of the Baire category theorem that rooted {\bf S4}-frames with infinite ascending chains cannot be obtained as interior images of the reals \cite[Sec.~6]{BG02}. We overcome this obstacle by switching from general frames $\F_n$ to general spaces over ${\bf L}_2$. This can be done as follows. As we saw in Section~\ref{BinaryTreeSection}, we can realize each $\F_n$ as a general frame over $\mathfrak T_2$. By Kremer's result \cite[Lem.~6.4]{Kre13}, $\mathfrak T_2^+$ is embeddable in ${\bf L}_2^+$, hence each general frame over $\mathfrak T_2$ can be realized as a general space over ${\bf L}_2$. We next build an interior map from any (non-trivial) real interval onto ${\bf L}_2$. Such maps have already appeared in the literature; see \cite{Lan12,Kre13}. The sum of disjoint open intervals produces an open subspace $X$ of $\mathbf R$ and a general space over $X$ whose logic is $L$. Thus, general spaces over open subspaces of $\mathbf R$ give rise to all logics above {\bf S4}.

In realizing ${\bf L}_2$ as an interior image of a non-trivial real interval $I$, our method differs from both Kremer's \cite{Kre13} and Lando's \cite{Lan12} methods. Kremer utilizes a decomposition of $I$ (indeed of any complete dense-in-itself metrizable space) into equivalence classes indexed by $L_2$, and maps each element in a class to its index. Lando defines an interior mapping of $I$ onto ${\bf L}_2$ by successively labeling and relabeling the points of $I$ by points in $L_2$. Our construction is similar to Lando's construction in that we utilize a labeling scheme; although we make a sequence of labels that form a directed set. In line with earlier comments in Section~\ref{CantorSection}, our method utilizes the DCPO structure of $\mathfrak L_2$. All three methods are similar in that each utilizes a dissection of intervals into nowhere dense `borders' that separate two collections of intervals to be used in the next stage of the construction.

We start by recalling the construction of the Cantor set in a (non-trivial) real interval $I\subseteq\mathbf R$ with endpoints $x<y$ (note $I$ may be closed, open, or neither).

\smallskip

{\bf Step $0$:} Set $C_{0,1}=I$.

\smallskip

{\bf Step $n>0$:} Start with $2^{n-1}$ intervals $C_{n-1,1},\dots,C_{n-1,2^{n-1}}$, each of which is a closed subset of $I$ of length $\frac{y-x}{3^{n-1}}$. For each $m\in\{1,\dots,2^{n-1}\}$, remove the open middle third $U_{n-1,m}$ of $C_{n-1,m}$. Thus, we end step $n$ with $2^{n-1}$ removed open intervals $U_{n-1,1},\dots,U_{n-1,2^{n-1}}$, each of length $\frac{y-x}{3^n}$, and the remaining portion of $I$, specifically $2^n$ intervals $C_{n,1},\dots,C_{n,2^n}$, each of which is closed in $I$ and of length $\frac{y-x}{3^n}$.

\smallskip

The Cantor set in $I$ is
$$
{\bf C}^I=\bigcap\limits_{n\in\omega}\bigcup\limits_{m=1}^{2^n}C_{n,m}.
$$
Note that ${\bf C}={\bf C}^{[0,1]}$ and, as subspaces of $\mathbf{R}$, $\mathbf{C}^I$ is homeomorphic to ${\bf C}$ whenever $I$ is closed. When we need to keep track of $I$, we write $C^I_{n,m}$ and $U^I_{n,m}$ for the intervals involved in the construction of ${\bf C}^I$.

We now define a map $f:[0,1]\rightarrow L_2$. The technique is similar to that of Section~\ref{CantorSection}, when we built an interior map from the Cantor space $\mathbf C$ onto ${\bf L}_2$. For each $x\in[0,1]$, define recursively a strictly increasing (but possibly finite) sequence in $\T_2$; we refer to the values in the sequence as labels and write $\{L_n(x)\}$ for the sequence of labels associated with $x$. The following generalizes the main construction of \cite{BG05a}.

\smallskip

{\bf Step $0$:} For $x\in[0,1]$, set $L_0(x)$ equal to the root of $\T_2$, $\mathcal C_0=\{{\bf C}^{[0,1]}\}$, and
$$
\mathcal U_0=\{U^{[0,1]}_{i,j}:i\in\omega,j=1,\dots,2^i\}.
$$

\smallskip

{\bf Step $n+1$:} For $x\in\bigcup\mathcal U_n$, set
$$
L_{n+1}(x)=\left\{
\begin{array}{cl}
lL_n(x) & \text{if } x\in U^I_{i,j}\text{ for some }U^I_{i,j}\in\mathcal U_n \text{ with }j\text{ even},\\
rL_n(x) & \text{if } x\in U^I_{i,j}\text{ for some }U^I_{i,j}\in\mathcal U_n \text{ with }j\text{ odd}.
\end{array}\right.
$$
Note that $L_{n+1}(x)$ is undefined if $x\not\in\bigcup\mathcal U_n$. Set $\mathcal C_{n+1}=\{{\bf C}^I:I\in\mathcal U_n\}$, and
$$
\mathcal U_{n+1}=\{U^I_{i,j}:I\in\mathcal U_n,\ i\in\omega, j=1,\dots,2^i\}.
$$
We define $f:[0,1]\rightarrow L_2$ by $f(x)=\sup\{L_n(x)\}$. By the construction, $L_{n+1}(x)$ is a child of $L_n(x)$ (whenever $L_{n+1}(x)$ is defined), so $\{L_n(x)\}$ is a strictly increasing (possibly finite) sequence in $\T_2$. Since $\LL_2$ is a DCPO, it follows that $f$ is a well-defined map.

The following is the intuitive idea of the construction. We send the Cantor set to the root of $\T_2$. In the remaining open intervals, we send `half' of the Cantor set occuring in these intervals to the left child of the root and the other `half' to the right child. In the next remaining open intervals, we again send the Cantor set to appropriate left or right children. So any point occurring in one of these Cantor sets is sent to an element of $T_2$. But there are only countably many such Cantor sets and by the Baire category theorem, there must be something left over in $[0,1]$ (see Fact~4 below). This `left over' portion of the interval is sent to $L_2-T_2$.

We present some useful facts concerning the construction. The proof is straightforward.

\smallskip

\noindent{\bf Facts:}
\begin{enumerate}
\item $\forall n\in\omega$, $\mathcal U_n$ is a countable pairwise disjoint family of open intervals and the maximum length of an interval in $\mathcal U_n$ is $\frac{1}{3^{n+1}}$.
\item $\forall n\in\omega$, $\forall C\in\mathcal C_{n+1}$, $\exists I\in\mathcal U_n$, $C\subset I$.
\item $\mathcal C=\{C:\exists n\in\omega,\ C\in\mathcal C_n\}$ is a countable pairwise disjoint family.
\item $\bigcap\nolimits_{n\in\omega}\left(\bigcup\mathcal U_n\right)=[0,1]-\bigcup\mathcal C$.
\item $\{L_n(x)\}$ is an infinite sequence iff $x\in \bigcap\nolimits_{n\in\omega}\left(\bigcup\mathcal U_n\right)$ iff $x\not\in \bigcup\mathcal C$ iff $f(x)\in L_2-T_2$.
\item $\{L_n(x)\}$ is a finite sequence iff $x\in \bigcup\mathcal C$ iff $f(x)\in T_2$.
\item $\forall x\in[0,1]$, if $L_n(x)$ is defined, then $L_n(x)\in T_2$ is a sequence in $\{0,1\}$ consisting of $n$ values; more specifically, $\mathrm{dom}(L_0(x))=\varnothing$ and $\forall n\in\omega$, $\mathrm{dom}(L_{n+1}(x))=\downn$.
\item $\forall x\in[0,1]$, if $\{L_n(x)\}$ is infinite, then for all $n\in\omega$ there is a unique $I_{x,n}\in\mathcal U_n$ such that $x\in I_{x,n}$ and the length of $I_{x,n}$ is at most $\frac{1}{3^{n+1}}$.
\item $\forall x\in[0,1]$, if $\{L_n(x)\}$ is infinite, then $\{I_{x,n}:n\in\omega\}$ is a local basis for $x$.
\item The set $\bigcup\mathcal C$ is dense in $[0,1]$.
\end{enumerate}

\smallskip

\begin{lemma}\label{RealsOntoL2isOpen}
The map $f$ into $\mathbf L_2$ is open.
\end{lemma}

\begin{proof}
Let $U\subseteq[0,1]$ be an open interval. First we show that $f(U)$ is an $\le$-upset. Let $x\in U$, $a\in L_2$, and $f(x)\le a$. We show that $a\in f(U)$. If $f(x)\in L_2-T_2$, then $a=f(x)\in f(U)$. So assume that $f(x)\in T_2$. Then the sequence of labels $\{L_n(x)\}$ is finite, so $f(x)=L_n(x)$, where $n$ is the largest integer for which $L_n(x)$ is defined. Therefore, there is a unique $I\in\mathcal U_{n-1}$ (note $I=[0,1]$ in case $n=0$) such that $x\in{\bf C}^I=\bigcap\nolimits_{i\in\omega}\bigcup\nolimits_{j=1}^{2^i}C^I_{i,j}$. Thus, there are $i\in\omega$ and $j\in\{1,\dots,2^i\}$ such that $x\in C^I_{i,j}\subseteq U$.

Assume $a\in L_2-T_2$. Put $a_n=a|_{\downarrow(n-1)}$ (note if $n=0$, then $a_n$ is the root). Then $a_n=f(x)=L_n(x)$. Define recursively a sequence of open intervals $\{I_m:m\in\omega\}$. Let $x_0<y_0$ be the endpoints of $C^I_{i,j}$ and put $I_0=(x_0,y_0)$. There is an odd $k\in\{1,\dots,2^{i+1}\}$ such that $U^I_{i+1,k},U^I_{i+1,k+1}$ are the open middle thirds removed in the two closed thirds $C^I_{i+1,k},C^I_{i+1,k+1}\subset C^I_{i,j}$ which remain after removing $U^I_{i,j}$ from $C^I_{i,j}$. Set
$$
I_1=
\left\{
\begin{array}{ll}
U^I_{i+1,k}&\mbox{ if }a_{n+1}=ra_n,\\
U^I_{i+1,k+1}&\mbox{ if }a_{n+1}=la_n.
\end{array}
\right.
$$
For $m\ge 1$, set
$$
I_{m+1}=
\left\{
\begin{array}{ll}
U^{I_m}_{1,1}&\mbox{ if }a_{n+m+1}=ra_{n+m},\\
U^{I_m}_{1,2}&\mbox{ if }a_{n+m+1}=la_{n+m}.
\end{array}
\right.
$$
It follows by induction that for $m\ge 0$, we have $I_{m+1}\in\mathcal U_{n+m}$ and for all $y\in I_m$, we have $L_{n+m}(y)=a_{n+m}$. Therefore, for $m\ge 1$, we have $f(y)=a_{n+m}$ for all $y\in {\bf C}^{I_m}\subset U$. Moreover, if $y\in\bigcap\nolimits_{m\in\omega}I_m$, then
$$
f(y)=\sup\{L_i(y):i\in\omega\}=\sup\{L_{n+m}(y):m\in\omega\}=\sup\{a_{n+m}:m\in\omega\}=a.
$$
To see that $\bigcap\nolimits_{m\in\omega}I_m\not=\varnothing$, let $x_m<y_m$ be the endpoints of $I_m$. Set
$$
K_m=\left[x_m+\frac{y_m-x_m}{9},y_m-\frac{y_m-x_m}{9}\right].
$$
Then $I_{m+1}\subset K_m\subset I_m$. Since $\{K_m:m\in\omega\}$ is a strictly decreasing family of closed intervals,
$$
\{y\}=\bigcap\limits_{m\in\omega}K_m\subset\bigcap\limits_{m\in\omega}I_m\subset I_0\subset U
$$
for some $y\in[0,1]$. Thus, $a,a_{n+m}\in f(U)$ for all $m\in\omega$. Now, since for all $b\in {\uparrow}f(x)\cap T_2$ there is $a\in L_2-T_2$ such that $b\le a$ (and hence $b=a_{n+m}$ for some $m$), it follows that $f(U)$ is an $\le$-upset.

Next, if $f(x)\in L_2-T_2$, then there is $b\in f(U)\cap T_2$ such that $b\le f(x)$. To see this, since $\{I_{x,n}\in\mathcal U_n:n\in\omega,\ x\in I_{x,n}\}$ is a local basis for $x$, there is $n\in\omega$ such that $I_{x,n}\subset U$. By the construction, each $y\in I_{x,n}$ has the same $(n+1)^{\text{th}}$-label, namely $L_{n+1}(y)=L_{n+1}(x)$. Taking $y\in{\bf C}^{I_{x,n}}(\not=\varnothing)$, we get $y\in U$, $f(y)=L_{n+1}(y)=L_{n+1}(x)\le f(x)$, and $f(y)\in T_2$. Thus, $f(U)$ is Scott open, which completes the proof that $f$ is an open map.
\end{proof}

\begin{lemma}\label{RealsOntoL2}
The map $f$ is onto.
\end{lemma}

\begin{proof}
Let $U$ be an open subset of $[0,1]$ with $U\cap{\bf C}\not=\varnothing$. Then the root of $L_2$ belongs to $f(U)$. By the proof of Lemma~\ref{RealsOntoL2isOpen}, $f(U)$ is an $\le$-upset. Thus, $f(U)=L_2$, and so $f$ is onto.
\end{proof}

Let $I\in\mathcal U_n$ for some $n\in\omega$. Since each $x\in I$ has the same label, namely $L_{n+1}(x)$, setting $L(I)=L_{n+1}(x)$ for some $x\in I$ is well defined. An inductive argument gives that for all $a\in T_2$ except the root, there is $n\in\omega$ such that $\mathrm{dom}(a)=\downn$ and there is $I\in\mathcal U_n$ such that $L(I)=a$. (In fact, $\mathrm{dom}(L(I))=\downn$ iff $I\in\mathcal U_n$.)

\begin{lemma}\label{fImgIinUn}
For each $n\in\omega$ and $I\in \mathcal U_n$, we have $f(I)={\uparrow}L(I)$.
\end{lemma}

\begin{proof}
Since $I$ is an open interval containing ${\bf C}^I$ and $f(x)=L_{n+1}(x)=L(I)$ for each $x\in{\bf C}^I$, we have $L(I)\in f(I)$. Lemma~\ref{RealsOntoL2isOpen} gives that $f(I)$ is an $\le$-upset, hence ${\uparrow}L(I)\subseteq f(I)$. Conversely, $L(I)=L_{n+1}(x)\le f(x)$ for every $x\in I$. Thus, $f(I)\subseteq {\uparrow}L(I)$.
\end{proof}

\begin{lemma}\label{RealsOntoL2isCTS}
The map $f$ onto $\mathbf L_2$ is continuous.
\end{lemma}

\begin{proof}
It is sufficient to show that $f^{-1}({\uparrow}a)$ is open for each $a\in T_2$. If $a$ is the root, then $f^{-1}({\uparrow}a)=[0,1]$. Therefore, we may assume that $a$ is not the root. Then $\mathrm{dom}(a)=\downn$ for some $n\in\omega$. We show that
$$
f^{-1}({\uparrow}a)=\bigcup\{I\in\mathcal U_n:L(I)=a\}.
$$
The $\supseteq$ direction follows from Lemma~\ref{fImgIinUn} as
\begin{eqnarray*}
f\left(\bigcup\{I\in\mathcal U_n:L(I)=a\}\right) &=& \bigcup\{f(I):I\in\mathcal U_n \ \& \ L(I)=a\}\\
&=&\bigcup\{{\uparrow}L(I): I\in\mathcal U_n \ \& \ L(I)=a\}\\
&=& {\uparrow}a.
\end{eqnarray*}
Let $x\in f^{-1}({\uparrow}a)$. Then $a\le f(x)$ (and so $f(x)$ extends $a$), which gives $\mathrm{dom}(f(x))\supseteq\downn$ and $f(x)|_{\downarrow n}=a$, giving $L_{n+1}(x)=a$. Since $L_{n+1}(x)$ is defined, $x\in\bigcup\mathcal U_n$. So there is $I\in\mathcal U_n$ such that $x\in I$. Furthermore, $L(I)=L_{n+1}(x)=a$. So $x\in\bigcup\{I\in\mathcal U_n:L(I)=a\}$, and the $\subseteq$ direction holds. Thus, $f$ is continuous.
\end{proof}

Putting together Lemmas~\ref{RealsOntoL2isOpen}, \ref{RealsOntoL2}, and~\ref{RealsOntoL2isCTS} yields:

\begin{theorem}\label{L2isIntrImageR}
$\mathbf L_2$ is an interior image of any (non-trivial) interval in $\mathbf R$.
\end{theorem}

\begin{proof}
Since $f$ sends the entire Cantor set ${\bf C}^{[0,1]}$ to the root of $\LL_2$, it is straightforward that both $f|_{(0,1)}$ and $f|_{[0,1)}$ are interior maps from $(0,1)$ and $[0,1)$ onto $\mathbf L_2$, respectively. The result follows since any (non-trivial) real interval is homeomorphic to either $(0,1)$, $[0,1)$, or $[0,1]$.
\end{proof}

In order to prove the main result of this section, we need the following lemma.

\begin{lemma}\label{RefuteOneFormInGenSpOnR}
Let $L$ be a logic above ${\bf S4}$. If $\varphi\notin L$, then there is a general space over any (non-trivial) interval in $\mathbf R$ validating $L$ and refuting $\varphi$.
\end{lemma}

\begin{proof}
By Lemma~\ref{RefuteOneFormInGenSpOnL_2}, there is a general space $({\bf L}_2,\mathcal P)$ for $L$ refuting $\varphi$. By Theorem~\ref{L2isIntrImageR}, ${\bf L}_2$ is an interior image of any (non-trivial) interval $I$ in $\mathbf R$, so ${\bf L}_2^+$ is isomorphic to a subalgebra of $I^+$. Therefore, $\mathcal P$ is isomorphic to some $\mathcal Q\in{\bf S}(I^+)$. Thus, there is a general space $(I,\mathcal Q)$ for $L$ refuting $\varphi$.
\end{proof}

We are ready to prove the main result of this section.

\begin{theorem}\label{OneLogicOneGenSpaceInSubsOfReals}
For a normal modal logic $L$, the following conditions are equivalent.
\begin{enumerate}
\item $L$ is a logic above {\bf S4}.
\item $L$ is the logic of a general space over an open subspace of $\mathbf R$.
\item There is a closure algebra $\mathfrak A\in{\bf SH}(\mathbf R^+)$ such that $L=L(\mathfrak A)$.
\end{enumerate}
\end{theorem}

\begin{proof}
(1)$\Rightarrow$(2): Let $L$ be a logic above {\bf S4}, and let $\{\varphi_n:n\in\omega\}$ be an enumeration of the non-theorems of $L$. By Lemma~\ref{RefuteOneFormInGenSpOnR}, for each $n\in\omega$, there is a general space $\X_n=((n,n+1),\mathcal P_n)$ for $L$ refuting $\varphi_n$. Taking the sum of $\X_n$ gives the general space $\X=(\bigcup_{n\in\omega}(n,n+1),\mathcal P)$, where $A\in\mathcal P$ iff $A\cap(n,n+1)\in\mathcal P_n$. Clearly $\bigcup_{n\in\omega}(n,n+1)$ is an open subspace of $\mathbf R$ and $L(\X)=L$.

(2)$\Rightarrow$(3): Suppose that $X$ is an open subspace of $\mathbf R$, $\X=(X,\mathcal P)$ is a general space over $X$, and $L=L(\X)$. Since $X$ is an open subspace of $\mathbf R$, we have $X^+\in{\bf H}(\mathbf R^+)$. As $\mathcal P$ is a subalgebra of $X^+$, we have $\mathcal P\in{\bf SH}(\mathbf R^+)$. Finally, $L=L(\X)$ yields $L=L(\mathcal P)$.

(3)$\Rightarrow$(1): This is clear since $\mathfrak A$ is a closure algebra and $L=L(\mathfrak A)$.
\end{proof}

Putting together what we have established so far yields:

\begin{corollary}
Let $L$ be a normal modal logic. The following are equivalent.
\begin{enumerate}
\item $L$ is a logic above {\bf S4}.
\item $L=L(\mathfrak A)$ for some $\mathfrak A\in{\bf S}(\mathbf Q^+)$.
\item $L=L(\mathfrak B)$ for some $\mathfrak B\in{\bf S}({\bf C}^+)$.
\item $L=L(\mathfrak C)$ for some $\mathfrak C\in{\bf SH}(\mathbf R^+)$.
\item $L=L(\mathfrak D)$ for some $\mathfrak D\in{\bf SH}({\bf L}_2^+)$.
\item $L=L(\mathfrak E)$ for some $\mathfrak E\in{\bf SH}(\T_2^+)$.
\end{enumerate}
\end{corollary}

In the proof of Theorem \ref{OneLogicOneGenSpaceInSubsOfReals}, instead of taking $X=\bigcup_{n\in\omega}(n,n+1)$, we could have taken $X$ to be larger. But in general, the `largest' $X$ can get is a countably infinite union of pairwise disjoint open intervals that is dense in $\mathbf R$. In the next section we characterize the logics above {\bf S4} that arise from general spaces over the entire real line.

\section{Connected logics and completeness for general spaces over $\mathbf R$}\label{SecRealLineAndConnLogs}

In this section we characterize the logics above {\bf S4} that arise from general spaces over $\mathbf R$. Since $\mathbf R^+$ is connected, so is each subalgebra of $\mathbf R^+$, so if $L$ is the logic of a general space over $\mathbf R$, then $L$ is connected. In \cite{BG11} it was shown that if $L$ is a connected logic above {\bf S4} that has the FMP, then $L$ is the logic of some subalgebra of $\mathbf R^+$. We strengthen this result by proving that a logic $L$ above {\bf S4} is connected iff $L$ is the logic of some general space over $\mathbf R$, which is equivalent to $L$ being the logic of some subalgebra of $\mathbf R^+$. This is the main result of the paper and solves \cite[p.~306, Open Problem 2]{BG11}. The proof requires several steps. For a connected logic $L$, we show that each non-theorem $\varphi_n$ of $L$ is refuted on a general space over an interior image $X_n$ of ${\bf L}_2$. For this we use the CGFP and Kremer's embedding of $\T_2^+$ into ${\bf L}_2^+$. We also utilize that ${\bf L}_2$ is an interior image of $\mathbf R$, which allows us to obtain each $X_n$ as an interior image of $\mathbf R$. We generalize the technique of `gluing' from \cite{BG11} to glue the $X_n$ accordingly, and design an interior map from $\mathbf R$ onto the glued copies of the $X_n$, which yields a general space over $\mathbf R$ whose logic is $L$.

For the readers' convenience, we first state the main result and provide the sketch of the proof. The various technical tools that are utilized in the proof, as well as the rigorous definitions of some of the constructions are provided later in the section.

\begin{theorem}[Main Result]\label{MainTheorem}
Let $L$ be a logic above ${\bf S4}$. The following are equivalent.
\begin{enumerate}
\item $L$ is connected.
\item $L$ is the logic of a countable path-connected general {\bf S4}-frame.
\item $L$ is the logic of a countable connected general space.
\item $L$ is the logic of a general space over $\mathbf R$.
\item $L=L(\mathfrak A)$ for some $\mathfrak A\in{\bf S}(\mathbf R^+)$.
\end{enumerate}
\end{theorem}

Some of the implications of the Main Result are easy to prove. Indeed, to see that (2)$\Rightarrow$(3), if $(W,R,\mathcal P)$ is a general {\bf S4}-frame, then $(W,\tau_R,\mathcal P)$ is a general space. Now since $(W,\tau_R)$ is connected iff $(W,R)$ is path-connected (see, e.g., \cite[Lem.~3.4]{BG11}), the result follows. The implications (3)$\Rightarrow$(1) and (5)$\Rightarrow$(1) are clear because a subalgebra of a connected closure algebra is connected and $X^+$ is connected iff $X$ is connected \cite[Thm.~3.3]{BG11}. The equivalence (4)$\Leftrightarrow$(5) is obvious. Thus, to complete the proof of the Main Result it is sufficient to establish (1)$\Rightarrow$(2) and (1)$\Rightarrow$(5). Below we give an outline of the proof of these implications. Full details are given in Sections~\ref{ss:1->2} and~\ref{ss:1->5}. In both proof sketches we will distinguish two cases depending on whether the logic $L$ is above ${\bf S4.2}$ or not, where we recall that ${\bf S4.2}={\bf S4}+\Diamond\Box\varphi\to\Box\Diamond\varphi$.

\medskip

\noindent{\bf Proof sketch of (1)$\Rightarrow$(2):}

Since $L$ is connected, it is the logic of a connected closure algebra $\mathfrak A$. Let $\F=(W,R,\mathcal P)$ be the dual descriptive frame of $\mathfrak A$. In general, $\F$ does not have to be path-connected \cite[Sec.~3]{BG11}.

\smallskip

\noindent\underline{{\bf Case 1:} $L$ is above ${\bf S4.2}$.}

\smallskip

\textbf{Step 1.1:} Show that $\F$ contains a unique maximal cluster accessible from each point of~$W$, so $\F$ is path-connected.

\textbf{Step 1.2:} Extract a countable path-connected general frame $\G$ from $\F$ by using a modified version of the CGFP alluded to in Remarks~\ref{EnlargingV0forL} and \ref{EnlargingV0forCountableSet} so that $L=L(\G)$.

\smallskip

\noindent\underline{{\bf Case 2:} $L$ is not above ${\bf S4.2}$.} In this case $\F$ may not be path-connected, so we employ a different strategy.

\smallskip

\textbf{Step 2.1:} Introduce a family of auxiliary frames, which we refer to as forks; see Figure~\ref{f:forks}.

\textbf{Step 2.2:} Choose a countable family of countable rooted `refutation frames' for $L$ via the modification of CGFP that yields for each non-theorem $\varphi_n$ of $L$ a countable rooted general {\bf S4}-frame $\mathfrak G_n$ for $L$ that refutes $\varphi_n$ at a root and contains a maximal cluster.

\textbf{Step 2.3:} For each refutation frame $\G_n$ there is a corresponding fork that has a maximal cluster isomorphic to a maximal cluster of $\G_n$. Gluing the two frames along the maximal clusters gives a countable family of `attached frames', say $\HH_n$; see Figure~\ref{f:glue-1}.

\textbf{Step 2.4:} Each of the attached frames $\HH_n$ has a maximal point. Gluing the family $\HH_n$ along their maximal points yields a countable path-connected general frame $\HH$ such that $L=L(\HH)$; see Figure~\ref{f:gluing-2}.

\medskip

\noindent{\bf Proof sketch of (1)$\Rightarrow$(5):}

We utilize the frames occurring in the proof sketch of (1)$\Rightarrow$(2).

\smallskip

\noindent\underline{{\bf Case 1:} $L$ is above ${\bf S4.2}$.} Let $\G$ be as in Step 1.2 of (1)$\Rightarrow$(2).

\smallskip

\textbf{Step 1.1:} For each non-theorem $\varphi_n$, select a rooted generated subframe $\G_n$ of $\G$ so that $\G_n$ refutes $\varphi_n$.

\textbf{Step 1.2:} Construct an interior image $X_n$ of ${\bf L}_2$ so that there is a general space $\X_n=(X_n,\mathcal Q_n)$ satisfying $L(\X_n)=L(\G_n)$. Gluing the family $\X_n$ yields a general frame $\X$ whose logic is $L$.

\textbf{Step 1.3:} Realize each $X_n$ as an interior image of any non-trivial real interval.

\textbf{Step 1.4:} Produce an interior map $f:\mathbf R\to\X$ via the interior mappings of Step 1.3 and, utilizing $f^{-1}$, obtain a subalgebra of $\mathbf R^+$ whose logic is $L$.

\smallskip

\noindent\underline{{\bf Case 2:} $L$ is not above ${\bf S4.2}$.} Let $\G_n$ be as in Step 2.2 of (1)$\Rightarrow$(2).

\smallskip

\textbf{Step 2.1:} Build general spaces $\X_n$ as described in Step 1.2 so that $L(\X_n)=L(\G_n)$.

\textbf{Step 2.2:} Gluing $\X_n$ with the corresponding fork gives the general space $\Y_n$. Gluing the $\Y_n$ along the appropriate isolated points yields the general space $\Y$ whose logic is $L$.

\textbf{Step 2.3:} Produce an interior map $f:\mathbf R\to\Y$ using that each $\X_n$ and each fork is an interior image of any non-trivial real interval and, utilizing $f^{-1}$, obtain a subalgebra of $\mathbf R^+$ whose logic is $L$.

\medskip

The next two subsections are dedicated to developing in full detail the two proof sketches just presented. We end the section with an easy but useful corollary of the Main Result.

\subsection{Proof of (1)$\Rightarrow$(2)}\label{ss:1->2}

Let $L$ be a connected logic above {\bf S4}. Then $L=L(\mathfrak A)$ for some connected closure algebra $\mathfrak A$. Let $\F=(W,R,\mathcal P)$ be the dual descriptive frame of $\mathfrak A$. We distinguish two cases.

\smallskip

\noindent\underline{{\bf Case 1:} $L$ is above {\bf S4.2}.}

\smallskip

\noindent{\bf Step 1.1:} We show that in this case $\F$ contains a unique maximal cluster. It is important that $L$ is connected and above {\bf S4.2}.

\begin{lemma}\label{ConnectedAboveS42}
$\F$ has a unique maximal cluster.
\end{lemma}

\begin{proof}
Suppose $C_1$ and $C_2$ are distinct maximal clusters of $\F$. As $\F$ is a descriptive {\bf S4}-frame, $C_1,C_2$ are closed, and $R(C_1)\cap R^{-1}(C_2)=\varnothing$, there is an $R$-upset $U\in\mathcal P$ such that $C_1\subseteq U$ and $U\cap C_2=\varnothing$. Since $\mathfrak A$ is connected and $U$ is neither $\varnothing$ nor $W$, it cannot be simultaneously an $R$-upset and an $R$-downset. As $U$ is an $R$-upset, $U$ then is not an $R$-downset. Set $\nu(p)=U$. Since $U$ is an $R$-upset, we have $\nu(\Box p)=U$, so $\nu(\Diamond\Box p)=R^{-1}(U)$. On the other hand, $C_1\subseteq R^{-1}(U)$ and $C_2\cap R^{-1}(U)=\varnothing$ imply that $R^{-1}(U)$ is neither $\varnothing$ nor $W$. Therefore, as $R^{-1}(U)$ is an $R$-downset, $R^{-1}(U)$ is not an $R$-upset. Thus, since $\nu(\Box\Diamond p)$ is the largest $R$-upset contained in $R^{-1}(U)$, it is strictly contained in $R^{-1}(U)$. Consequently, $\Diamond\Box p\rightarrow\Box\Diamond p$ is refuted in $\F$, and hence in $\mathfrak A$. The obtained contradiction proves that $\F$ has a unique maximal cluster.
\end{proof}

Since the unique maximal cluster of $\F$ is accessible from every point of $\F$, it follows that $\F$ is path-connected.

\smallskip

\noindent{\bf Step 1.2:} As indicated in the proof sketch, we develop the modified version of the CGFP (Theorem~\ref{CGFP}) outlined in Remarks~\ref{EnlargingV0forL} and \ref{EnlargingV0forCountableSet}.

\begin{theorem}\label{LowSklWithSet}
Let $L$ be a normal modal logic, $\F=(W,R,\mathcal P)$ be a general frame for $L$, and $U$ be a countable subset of $W$.
\begin{enumerate}
\item If $\F$ refutes a non-theorem $\varphi$ of $L$, then there is a countable general frame $\G=(V,S,\mathcal Q)$ such that $\G$ is a subframe of $\F$, $U\subseteq V$, $\G$ is a frame for $L$, and $\G$ refutes~$\varphi$.
\item If $L=L(\F)$, then $\G$ may be selected so that $U\subseteq V$ and $L=L(\G)$.
\end{enumerate}
\end{theorem}

\begin{proof}
(1) The proof of the CGFP given in Theorem~\ref{CGFP} needs to be adjusted only slightly. Since $\varphi$ is refuted in $\F$, there is a valuation $\nu$ and $w\in W$ such that $w\not\in\nu(\varphi)$. Now proceed as in the proof of Theorem~\ref{CGFP} but set the starting set $V_0$ to be equal to $U\cup\{w\}$. The resulting countable general frame $\G=(V,S,\mathcal Q)$ is a subframe of $\F$, $U\subseteq V$, $\G$ is a frame for $L$, and $\G$ refutes $\varphi$. Thus, (1) is established.

(2) For each non-theorem $\varphi_n$ of $L$, there is a valuation $\nu_n$ and $w_n\in W$ such that $w_n\notin\nu_n(\varphi_n)$. We use Lemma~\ref{revariablizing} to make the propositional letters occurring in substitution instances of $\varphi_n$ and $\varphi_m$ distinct whenever $n\neq m$. This gives the set $\{\widehat{\varphi_n}:n\in\omega\}$. Let $\nu$ be a single valuation that refutes all the $\widehat{\varphi_n}$, and let $V_0=\{w_n\in W:w_n\notin\nu(\widehat{\varphi_n})\}\cup U$. Then proceed precisely as in the proof of (1) to obtain a countable general frame $\G=(V,S,\mathcal Q)$ such that $\G$ is a subframe of $\F$, $U\subseteq V$, and $\G$ is a frame for $L$. To see that each $\varphi_n$ is refuted in $\G$, note that by our construction, $\widehat{\varphi_n}$ is refuted in $\G$, and $\widehat{\varphi_n}$ is obtained from $\varphi_n$ by substituting propositional letters with other propositional letters. By a straightforward adjustment of the valuation according to the substitution, we obtain a refutation of $\varphi_n$. Thus, $L=L(\G)$.
\end{proof}

We use the modified CGFP. Take any point $m$ from the unique maximal cluster of $\F$ provided by Lemma~\ref{ConnectedAboveS42}, and set $U=\{m\}$. By Theorem~\ref{LowSklWithSet}(2), there is a countable general frame $\G$ such that $\G$ is a subframe of $\F$, it contains $U$, and its logic is $L$. Moreover, $\G$ is path-connected since $m$ is accessible from every point of $\G$, thus finishing the proof of Case~1.

\medskip

\noindent\underline{{\bf Case 2:} $L$ is not above {\bf S4.2}.}

\smallskip

\noindent{\bf Step 2.1:} We introduce some very simple auxiliary frames that will be used later on for gluing refutation frames for $L$ into one connected general frame whose logic is $L$. Let $\alpha\in\omega+1$ be nonzero. We let $\mathfrak C_\alpha=(W,R)$ denote the {\em $\alpha$-cluster}; that is, $\mathfrak C_\alpha$ is the {\bf S4}-frame consisting of a single cluster of cardinality $\alpha$, so $W=\{w_n:n\in\alpha\}$ and $R=W\times W$. We also let $\F_\alpha=(W_\alpha,R_\alpha)$ denote the {\em $\alpha$-fork}; that is, the {\bf S4}-frame obtained by adding two points to $\mathfrak C_\alpha=(W,R)$, a root below the cluster and a maximal point unrelated to the cluster. So $W_\alpha=\{r_\alpha,m_\alpha\}\cup W$ and
$$
R_\alpha=\{(r_\alpha,r_\alpha),(m_\alpha,m_\alpha),(r_\alpha,m_\alpha),(r_\alpha,w_n),(w_n,w_m):n,m\in\alpha\}.
$$
How $\mathfrak C_\alpha$ sits inside $\F_\alpha$ is depicted in Figure~\ref{f:forks} below.

\begin{figure}[h]
\centering
\begin{tikzpicture}
\draw[rounded corners]  (-4,2) rectangle (-3,1.5);
\node (C0) at (-3.5,1.75) {\tiny $\mathfrak C_{\alpha}$};
\node[circle, fill, inner sep=1.5] (r1) at (-2.5,0) {};
\node[circle, fill, inner sep=1.5] (m) at (-1.5,1.5) {};
\node at (-2.5,-0.4) {\small $r_\alpha$};
\node (mnew) at (-1.5,1.8) {\small $m_\alpha$};
\draw[->=latex, shorten <=2pt, shorten >=2pt] (r1) -- (m);
\draw[->=latex, shorten <=2pt, shorten >=2pt] (r1) -- (-3.5,1.47);
\end{tikzpicture}
\caption{The $\alpha$-fork $\F_\alpha$}\label{f:forks}
\end{figure}
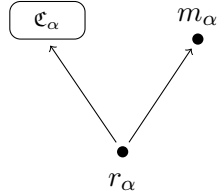

\begin{lemma}\label{lem:8.6}
Let $L$ be a logic above {\bf S4}, $\F=(W,R,\mathcal P)$ be a descriptive {\bf S4}-frame for $L$, and $\F$ have an infinite maximal cluster $\mathfrak C$. Then
\begin{enumerate}
\item $\mathfrak C_n\models L$ for each nonzero $n\in\omega$.
\item $\mathfrak C_\omega\models L$.
\end{enumerate}
\end{lemma}

\begin{proof}
(1) Since $\mathfrak C$ is a maximal cluster of $\F$, it is clear that $\mathfrak C$ is an $R$-upset of $\F$. In fact, $\mathfrak C=R(w)$ for each $w$ in $\mathfrak C$, and as $R(w)$ is closed, $\mathfrak C$ is a closed $R$-upset of $\F$. Therefore, $\mathfrak C$ is a descriptive {\bf S4}-frame for $L$ (see, e.g, \cite[Lem.~III.4.11]{Esa85}). Since $\mathfrak C$ is an infinite Stone space, for each nonzero $n\in \omega$, there is a partition of $\mathfrak C$ into $n$-many clopens $U_0,\dots,U_{n-1}$. Define $f:\mathfrak C\to\mathfrak C_n$ by sending all points of the clopen $U_i$ to $w_i$ in $\mathfrak C_n$. It is straightforward to verify that $f$ is a p-morphism. Thus, as $\mathfrak C\models L$, we have $\mathfrak C_n\models L$.

(2) If $\mathfrak C_\omega\not\models L$, then there are $\varphi\in L$, $n\in\omega$, and a valuation $\nu$ on $\mathfrak C_\omega$ such that $w_n\not\in\nu(\varphi)$. Define an equivalence relation $\equiv$ on $\mathfrak C_\omega$ by
$$
w_i\equiv w_j\text{ iff }(\forall\psi\in\mathrm{Sub}\varphi)\left(w_i\in\nu(\psi)\Leftrightarrow w_j\in\nu(\psi)\right),
$$
where $\mathrm{Sub}\varphi$ is the set of subformulas of $\varphi$. Since $\mathrm{Sub}\varphi$ is finite, so is the set of equivalence classes, and we let $\{C_k:k\in m\}$ be this set. Then $f:\mathfrak C_\omega\rightarrow\mathfrak C_m$, given by $f(w_i)=w_k$ whenever $w_i\in C_k$, is an onto p-morphism. Moreover, $\mathfrak C_m$ refutes $\varphi$ at $f(w_n)$ under the valuation $\mu=f\circ\nu$. But this contradicts (1), completing the proof of (2).
\end{proof}

\begin{lemma}\label{OmegaFork}
Let $L$ be a logic above ${\bf S4}$. If $\F_k\models L$ for each $k\in\omega$, then $\F_\omega\models L$.
\end{lemma}

\begin{proof}
Suppose that $\F_\omega\not\models L$. Then there are $\varphi\in L$, a valuation $\nu$, and $v\in W_\omega$ such that $v\not\in\nu(\varphi)$. Let ${\equiv}$ be the equivalence relation on $\mathfrak C_\omega$ defined in the proof of Lemma~\ref{lem:8.6}(2), and let $\{C_k:k\in n\}$ be the set of $\equiv$-equivalence classes. Define $f:\F_\omega\rightarrow\F_n$ by
$$f(w)=\left\{
\begin{array}{ll}
r_n&\mbox{ if }w=r_\omega,\\
m_n&\mbox{ if }w=m_\omega,\\
w_k&\mbox{ if }w\in C_k.
\end{array}\right.
$$
Then $f$ is an onto p-morphism. Moreover, $\F_n$ refutes $\varphi$ at $f(v)$ under the valuation $\mu=f\circ\nu$. This contradicts to $\F_n\models L$. Thus, $\F_\omega\models L$.
\end{proof}

\smallskip

\noindent{\bf Step 2.2:} For each non-theorem $\varphi_n$ of $L$, there is a valuation $\nu_n$ and $w_n\in W$ such that $w_n\notin\nu_n(\varphi_n)$. Let $m_n\in R(w_n)$ be a maximal point of $\F$. (Such $m_n$ exists because $\F$ is a descriptive {\bf S4}-frame; see, e.g., \cite[Sec.~III.2]{Esa85}.) By Theorem~\ref{LowSklWithSet}(1), there is a countable general frame $\G_n=(W_n,R_n,\mathcal P_n)$ containing $\{w_n,m_n\}$ that validates $L$ and refutes $\varphi_n$. Clearly $w_n$ is a root of $\G_n$. Let $C_n$ be the maximal cluster of $\G_n$ generated by $m_n$. If $\alpha_n$ is the cardinality of $C_n$, then we identify $C_n$ with $\mathfrak C_{\alpha_n}$. Let $C$ be the maximal cluster of $\F$ from which $C_n$ was selected. If $C$ is finite, then $C\models L$, so $\mathfrak C_{\alpha_n}\models L$. If $C$ is infinite, then as $\alpha_n\in\omega+1$ is nonzero, by Lemma~\ref{lem:8.6}, $\mathfrak C_{\alpha_n}\models L$.

\smallskip

\noindent{\bf Step 2.3:} Since $L$ is not above {\bf S4.2}, it is well known (see, e.g., \cite[Sec.~6.1]{ZWC01}) that $\F_1\models L$. If $\alpha_n$ is finite, \cite[Lem.~4.2]{BG11} gives that $\F_{\alpha_n}\models L$. If $\alpha_n=\omega$, we get that $\mathfrak C_m\models L$ for each nonzero $m\in\omega$ because each $\mathfrak C_m$ is a p-morphic image of $\mathfrak C_\omega$. By \cite[Lem.~4.2]{BG11}, each $\F_m\models L$. Therefore, by Lemma~\ref{OmegaFork}, $\F_\omega\models L$. Thus, $\F_{\alpha_n}\models L$.

For our next move, we need to introduce the operation of gluing for general {\bf S4}-frames, which generalizes the gluing of finite {\bf S4}-frames introduced in \cite{BG11}. However, later on we will also need to glue general spaces. Because of this, we introduce the operation of gluing for general spaces, which is similar to the operation of \emph{attaching space} or \emph{adjunction space}, a particular case of which is the \emph{wedge sum}. Both constructions are used in algebraic topology. Since general {\bf S4}-frames are a particular case of general spaces, we will view gluing of general {\bf S4}-frames as a particular case of gluing of general spaces. We start by defining gluing of topological spaces.

\begin{definition}\label{def:gluing}
{\em
Let $X_i$ be a family of topological spaces indexed by $I$. Without loss of generality we may assume that $\{X_i:i\in I\}$ is pairwise disjoint. Let $Y$ be a topological space disjoint from each $X_i$ and such that for each $i\in I$ there is an open subspace $Y_i$ of $X_i$ homeomorphic to $Y$. Let $f_i:Y\to Y_i$ be a homeomorphism. Define an equivalence relation $\equiv$ on $\bigcup_{i\in I}X_i$ by
$$
x\equiv z\text{ iff } x=z \text{ or } x\in Y_i, \ z\in Y_j, \text{ and } (\exists y\in Y)(x=f_i(y) \text{ and } z=f_j(y)).
$$
We call the quotient space $X=\bigcup_{i\in I}X_i/{\equiv}$ the {\em gluing of the $X_i$ along $Y$}.
}
\end{definition}

\begin{lemma}\label{lem:gluing}
Let the $X_i$ and $Y$ be as in Definition~\ref{def:gluing}, and let $X$ be the gluing of the $X_i$ along $Y$. We let $\rho:\bigcup_{i\in I}X_i\to X$ be the quotient map. Then $\rho$ is an onto interior map.
\end{lemma}

\begin{proof}
Since a quotient map is always continuous and onto, we only need to check that $\rho$ is open. By \cite[Cor.~2.4.10]{Eng89}, it is sufficient to show that $\rho^{-1}\rho(U)$ is open in $\bigcup_{i\in I}X_i$ for each $U$ open in $\bigcup_{i\in I}X_i$. We have $U=\bigcup_{i\in I}U_i$, where each $U_i$ is open in $X_i$. Therefore, $U_i\cap Y_i$ is open in $Y_i$, and hence $f_i^{-1}(U_i\cap Y_i)$ is open in $Y$. Thus, $V=\bigcup_{i\in I} f_i^{-1}(U_i\cap Y_i)$ is open in $Y$. This implies $V_i=f_i(V)$ is open in $X_i$, yielding $\rho^{-1}(\rho(U))=\bigcup_{i\in I}(U_i\cup V_i)$. This shows that $\rho$ is indeed open.
\end{proof}

We note in passing that the gluing operation is actually a pushout in the category of topological spaces with interior maps as morphisms. We next generalize Definition~\ref{def:gluing} to general spaces.

\begin{definition}\label{def:gluing1}
{\em
Let $\X_i=(X_i,\mathcal P_i)$ be a family of general spaces indexed by $I$. Without loss of generality we may assume that $\{X_i:i\in I\}$ is pairwise disjoint. Let $\Y=(Y,\mathcal Q)$ be a general space such that $Y$ is disjoint from each $X_i$ and for each $i\in I$ there is an open subspace $\Y_i=(Y_i,\mathcal Q_i)$ of $\X_i$ homeomorphic to $\Y$. Suppose $f_i:Y\to Y_i$ is a homeomorphism. Let $X$ be the gluing of the $X_i$ along $Y$, and let $\rho:\bigcup_{i\in I}X_i\to X$ be the quotient map. Define $\mathcal P=\{A\subseteq X:\rho^{-1}(A)\cap X_i\in\mathcal P_i \ \ \forall i\}$. Lemma~\ref{lem:gluing} yields that $\mathcal P$ is a subalgebra of $X^+$, hence $\X=(X,\mathcal P)$ is a general space. We call $\X$ the {\em gluing of the $\X_i$ along $\Y$}.
}
\end{definition}

We now produce a new frame by gluing $\F_{\alpha_n}$ and $\G_n$ along the cluster $\mathfrak C_{\alpha_n}$. Since $\mathfrak C_{\alpha_n}$ is a maximal cluster in both $\G_n$ and $\mathfrak F_{\alpha_n}$, if we view $\G_n$ and $\mathfrak F_{\alpha_n}$ as general Alexandroff spaces, $\mathfrak C_{\alpha_n}$ becomes an open subspace of both. Let $\HH_n$ be the general {\bf S4}-frame obtained by gluing $\G_n$ and $\mathfrak F_{\alpha_n}$ along $\mathfrak C_{\alpha_n}$; see Figure~\ref{f:glue-1}. Then $\HH_n$ is a p-morphic image of the disjoint union of $\G_n$ and $\mathfrak F_{\alpha_n}$. As both validate $L$, so does $\HH_n$. Also, since $\G_n$ is (isomorphic to) a generated subframe of $\HH_n$ and $\G_n$ refutes $\varphi_n$, so does $\HH_n$.

\begin{figure}[h]
\centering
\begin{tikzpicture}[x=1mm, y=1mm, inner xsep=0pt, inner ysep=-1.2pt]
\path[line width=0mm] (34.79,-254.01) rectangle +(121.81,37.12);
\definecolor{L}{rgb}{0,0,0}
\path[line width=0.30mm, draw=L] (114.04,-223.49) -- (123.18,-242.61) -- (132.29,-223.49) -- cycle;
\definecolor{F}{rgb}{1,1,1}
\path[line width=0.30mm, draw=L, fill=F] (127.93,-224.82) .. controls (127.93,-225.10) and (128.15,-225.32) .. (128.43,-225.32) .. controls (131.60,-225.32) and (131.60,-225.32) .. (134.77,-225.32) .. controls (135.05,-225.32) and (135.27,-225.10) .. (135.27,-224.82) .. controls (135.27,-223.28) and (135.27,-223.28) .. (135.27,-221.73) .. controls (135.27,-221.46) and (135.05,-221.23) .. (134.77,-221.23) .. controls (131.60,-221.23) and (131.60,-221.23) .. (128.43,-221.23) .. controls (128.15,-221.23) and (127.93,-221.46) .. (127.93,-221.73) .. controls (127.93,-223.28) and (127.93,-223.28) .. (127.93,-224.82) -- cycle;
\draw(129.11,-224) node[anchor=base west]{\fontsize{7.4}{10.24}\selectfont $\mathfrak C_{\alpha_n}$\strut};
\definecolor{F}{rgb}{0,0,0}
\path[line width=0.30mm, draw=L, fill=F] (149.80,-224.23) circle (0.40mm);
\path[line width=0.30mm, draw=L, fill=F] (140.99,-241.81) circle (0.40mm);
\path[line width=0.30mm, draw=L] (140.26,-240.91) -- (132.78,-225.81);
\path[line width=0.30mm, draw=L, fill=F] (132.78,-225.81) -- (132.69,-227.21) -- (132.78,-225.81) -- (133.94,-226.59) -- (132.78,-225.81) -- cycle;
\path[line width=0.30mm, draw=L] (141.67,-240.91) -- (149.15,-225.81);
\path[line width=0.30mm, draw=L, fill=F] (149.15,-225.81) -- (147.99,-226.59) -- (149.15,-225.81) -- (149.24,-227.21) -- (149.15,-225.81) -- cycle;
\path[line width=0.30mm, draw=L] (37.74,-223.49) -- (46.89,-242.61) -- (55.99,-223.49) -- cycle;
\definecolor{F}{rgb}{1,1,1}
\path[line width=0.30mm, draw=L, fill=F] (51.63,-224.98) .. controls (51.63,-225.26) and (51.86,-225.48) .. (52.13,-225.48) .. controls (55.31,-225.48) and (55.31,-225.48) .. (58.48,-225.48) .. controls (58.76,-225.48) and (58.98,-225.26) .. (58.98,-224.98) .. controls (58.98,-223.44) and (58.98,-223.44) .. (58.98,-221.89) .. controls (58.98,-221.61) and (58.76,-221.39) .. (58.48,-221.39) .. controls (55.31,-221.39) and (55.31,-221.39) .. (52.13,-221.39) .. controls (51.86,-221.39) and (51.63,-221.61) .. (51.63,-221.89) .. controls (51.63,-223.44) and (51.63,-223.44) .. (51.63,-224.98) -- cycle;
\draw(52.82,-224) node[anchor=base west]{\fontsize{7.4}{10.24}\selectfont $\mathfrak C_{\alpha_n}$\strut};
\path[line width=0.30mm, draw=L, fill=F] (62.97,-224.98) .. controls (62.97,-225.26) and (63.19,-225.48) .. (63.47,-225.48) .. controls (66.64,-225.48) and (66.64,-225.48) .. (69.82,-225.48) .. controls (70.09,-225.48) and (70.32,-225.26) .. (70.32,-224.98) .. controls (70.32,-223.44) and (70.32,-223.44) .. (70.32,-221.89) .. controls (70.32,-221.61) and (70.09,-221.39) .. (69.82,-221.39) .. controls (66.64,-221.39) and (66.64,-221.39) .. (63.47,-221.39) .. controls (63.19,-221.39) and (62.97,-221.61) .. (62.97,-221.89) .. controls (62.97,-223.44) and (62.97,-223.44) .. (62.97,-224.98) -- cycle;
\draw(64.35,-224) node[anchor=base west]{\fontsize{7.4}{10.24}\selectfont $\mathfrak C_{\alpha_n}$\strut};
\definecolor{F}{rgb}{0,0,0}
\path[line width=0.30mm, draw=L, fill=F] (84.84,-224.39) circle (0.40mm);
\path[line width=0.30mm, draw=L, fill=F] (76.03,-241.97) circle (0.40mm);
\path[line width=0.30mm, draw=L] (75.30,-241.07) -- (67.82,-225.97);
\path[line width=0.30mm, draw=L, fill=F] (67.82,-225.97) -- (67.73,-227.37) -- (67.82,-225.97) -- (68.98,-226.75) -- (67.82,-225.97) -- cycle;
\path[line width=0.30mm, draw=L] (76.71,-241.07) -- (84.20,-225.97);
\path[line width=0.30mm, draw=L, fill=F] (84.20,-225.97) -- (83.03,-226.75) -- (84.20,-225.97) -- (84.29,-227.37) -- (84.20,-225.97) -- cycle;
\path[line width=0.15mm, draw=L, dash pattern=on 0.60mm off 0.50mm] (49.00,-224.96) .. controls (49.00,-226.63) and (50.33,-227.96) .. (52.00,-227.96) .. controls (60.90,-227.96) and (60.90,-227.96) .. (69.79,-227.96) .. controls (71.46,-227.96) and (72.79,-226.63) .. (72.79,-224.96) .. controls (72.79,-223.43) and (72.79,-223.43) .. (72.79,-221.90) .. controls (72.79,-220.23) and (71.46,-218.90) .. (69.79,-218.90) .. controls (60.90,-218.90) and (60.90,-218.90) .. (52.00,-218.90) .. controls (50.33,-218.90) and (49.00,-220.23) .. (49.00,-221.90) .. controls (49.00,-223.43) and (49.00,-223.43) .. (49.00,-224.96) -- cycle;
\path[line width=0.1mm, draw=L] (36.79,-246.97) -- (59.12,-246.98);
\path[line width=0.1mm, draw=L][->] (36.87,-247.01) -- (36.87,-245.10);
\path[line width=0.1mm, draw=L][->] (59.03,-247.06) -- (59.03,-245.15);
\path[line width=0.1mm, draw=L] (62.97,-246.86) -- (85.30,-246.87);
\path[line width=0.1mm, draw=L][->] (63.05,-246.90) -- (63.05,-244.98);
\path[line width=0.1mm, draw=L][->] (85.21,-246.95) -- (85.21,-245.04);
\draw(46.16,-251) node[anchor=base west]{\fontsize{9}{10.24}\selectfont $\mathfrak G_n$\strut};
\draw(74.94,-251.36) node[anchor=base west]{\fontsize{9}{10.24}\selectfont $\mathfrak F_{\alpha_n}$\strut};
\draw(130,-251.14) node[anchor=base west]{\fontsize{9}{10.24}\selectfont $\mathfrak H_n$\strut};
\path[line width=0.15mm, draw=L] (114.89,-246.52) -- (150.02,-246.55);
\path[line width=0.15mm, draw=L][->] (114.96,-246.56) -- (114.96,-244.65);
\path[line width=0.15mm, draw=L] (149.89,-246.61) -- (149.89,-244.70);
\path[line width=0.15mm, draw=L][->] (149.89,-246.56) -- (149.89,-244.65);
\draw(75.20,-245) node[anchor=base west]{\fontsize{9}{13.66}\selectfont $r_{\alpha_n}$\strut};
\draw(86.39,-223.43) node[anchor=base west]{\fontsize{9}{13.66}\selectfont $m_{\alpha_n}$\strut};
\draw(151.49,-223.33) node[anchor=base west]{\fontsize{9}{13.66}\selectfont $m_{\alpha_n}$\strut};
\draw(140.30,-245) node[anchor=base west]{\fontsize{9}{13.66}\selectfont $r_{\alpha_n}$\strut};
\end{tikzpicture}
\caption{Gluing of $\mathfrak G_n$ and $\mathfrak F_{\alpha_n}$}\label{f:glue-1}
\end{figure}
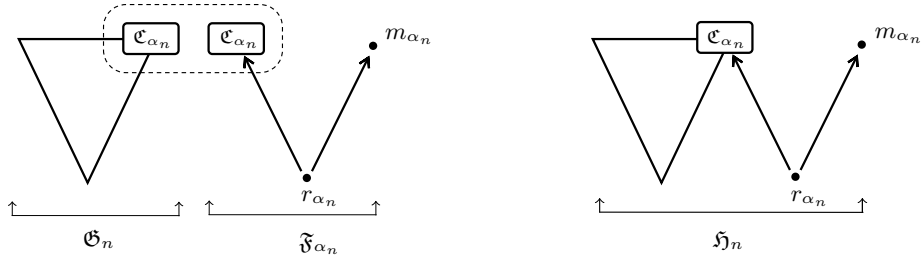

\smallskip

\noindent{\bf Step 2.4:} In this final step we glue the $\HH_n$ along the maximal element $m_{\alpha_n}$ as depicted in Figure~\ref{f:gluing-2}.

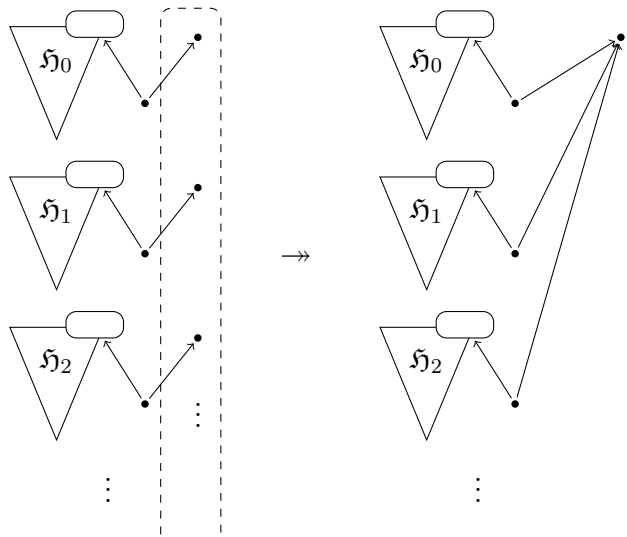
\begin{figure}[h]
\centering
\tikzstyle{trigl}=[
   isosceles triangle,
   draw,
   shape border rotate=90,
   inner sep=8,
   font=\small\sffamily\bfseries,
   shape border rotate =-90,
   isosceles triangle apex angle=45,
   isosceles triangle stretches,
   text depth=-3ex]
\tikzstyle{cluster}=[
   fill=white,
   draw,
   inner sep=5,
   rounded corners]
\tikzstyle{pnt}=[
	circle,
	inner sep=1,
	fill]
\begin{tikzpicture}
\node [trigl, anchor=right side] (f1)  at (16, 0) {$\mathfrak H_0$};
\node[cluster, right of=f1, xshift=-14, yshift=5](c1) {$\ \ \ $ };
\node[pnt, right of=f1, xshift=5, yshift=-25] (r1) {};
\node[pnt, right of=f1, xshift=25] (m1) {};
\draw[->, shorten <=1pt, shorten >=1pt] (r1) -- (m1);
\draw[->, shorten <=1pt, shorten >=1pt] (r1) -- (c1);
\node [trigl, anchor=right side] (f2)  at (16, -2) {$\mathfrak H_1$};
\node[cluster, right of=f2, xshift=-14, yshift=5](c2) {$\ \ \ $ };
\node[pnt, right of=f2, xshift=5, yshift=-25] (r2) {};
\node[pnt, right of=f2, xshift=25] (m2) {};
\draw[->, shorten <=1pt, shorten >=1pt] (r2) -- (m2);
\draw[->, shorten <=1pt, shorten >=1pt] (r2) -- (c2);
\node [trigl, anchor=right side] (f3)  at (16, -4) {$\mathfrak H_2$};
\node[cluster, right of=f3, xshift=-14, yshift=5](c3) {$\ \ \ $ };
\node[pnt, right of=f3, xshift=5, yshift=-25] (r3) {};
\node[pnt, right of=f3, xshift=25] (m3) {};
\draw[->, shorten <=1pt, shorten >=1pt] (r3) -- (m3);
\draw[->, shorten <=1pt, shorten >=1pt] (r3) -- (c3);
\node[draw=none, rectangle] at  (17,-5.3) {$\vdots$};
\node[draw=none, rectangle] at  (18.2,-4.3) {$\vdots$};
\draw[dashed, rounded corners] (17.7,-6) -- (17.7,1) -- (18.5,1) -- (18.5,-6);
\node[draw=none, rectangle] at  (19.5,-2.3) {$\twoheadrightarrow$};
\begin{scope}[xshift=140]
\node [trigl, anchor=right side] (f1)  at (16, 0) {$\mathfrak H_0$};
\node[cluster, right of=f1, xshift=-14, yshift=5](c1) {$\ \ \ $ };
\node[pnt, right of=f1, xshift=5, yshift=-25] (r1) {};
\node[pnt, right of=f1, xshift=45] (m1) {};
\draw[->, shorten <=1pt, shorten >=1pt] (r1) -- (m1);
\draw[->, shorten <=1pt, shorten >=1pt] (r1) -- (c1);
\node [trigl, anchor=right side] (f2)  at (16, -2) {$\mathfrak H_1$};
\node[cluster, right of=f2, xshift=-14, yshift=5](c2) {$\ \ \ $ };
\node[pnt, right of=f2, xshift=5, yshift=-25] (r2) {};
\draw[->, shorten <=1pt, shorten >=1pt] (r2) -- (m1);
\draw[->, shorten <=1pt, shorten >=1pt] (r2) -- (c2);
\node [trigl, anchor=right side] (f3)  at (16, -4) {$\mathfrak H_2$};
\node[cluster, right of=f3, xshift=-14, yshift=5](c3) {$\ \ \ $ };
\node[pnt, right of=f3, xshift=5, yshift=-25] (r3) {};
\draw[->, shorten <=1pt, shorten >=1pt] (r3) -- (m1);
\draw[->, shorten <=1pt, shorten >=1pt] (r3) -- (c3);
\node[draw=none, rectangle] at  (17,-5.3) {$\vdots$};
\end{scope}
\end{tikzpicture}
\caption{\label{f:gluing-2} Gluing of the frames $\HH_n$}
\end{figure}

This gluing is analogous to the wedge sum in algebraic topology. The resulting general {\bf S4}-frame $\HH$ is countable and path-connected. Moreover, since disjoint unions and p-morphic images of general frames preserve validity, $\HH$ validates $L$; and as each $\HH_n$ is a generated subframe of $\HH$, we see that $\HH$ refutes $\varphi_n$. Consequently, $L=L(\HH)$. This finishes the proof of (1)$\Rightarrow$(2).

\subsection{Proof of (1)$\Rightarrow$(5)}\label{ss:1->5}

As before we consider two cases.

\smallskip

\noindent\underline{{\bf Case 1:} $L$ is above ${\bf S4.2}$.}

\smallskip

Let $\G$ be the countable general {\bf S4}-frame constructed in Step 1.2 of the proof of (1)$\Rightarrow$(2). Then $\G$ has a unique maximal cluster $C$, which is accessible from each point $w$ in $\G$, and the logic of $\G$ is $L$.

\smallskip

\noindent{\bf Step 1.1:} For each non-theorem $\varphi_n$ of $L$, there are a valuation $\nu_n$ and a point $w_n$ in $\G$ such that $w_n\not\in\nu_n(\varphi_n)$. Let $\G_n=(W_n,R_n,\mathcal P_n)$ be the subframe of $\G$ generated by $w_n$. Then $\G_n$ is a general frame for $L$ that refutes $\varphi_n$. Furthermore, $\G_n$ has $C$ as its unique maximal cluster and $R_n(w)$ contains $C$ for each point $w$ in $\G_n$.

\smallskip

\noindent{\bf Step 1.2:}\label{SSsectionQuotientOfL2} For each $\G_n$ we construct a general space $\X_n=(X_n,\mathcal Q_n)$ such that $X_n$ is an interior image of $\mathbf L_2$ and $\mathcal P_n$ is isomorphic to $\mathcal Q_n$, yielding $L(\mathcal\G_n)=L(\X_n)$.

Consider a countable rooted {\bf S4}-frame, say $\F=(W,R)$, with a maximal cluster, say $C$. By Lemma~\ref{ImagesOfT2}, there is a p-morphism $f$ from $\T_2$ onto $\F$. Let $\alpha:\T_2^+\rightarrow{\bf L}_2^+$ be the closure algebra embedding defined in \cite[Lem~6.4]{Kre13}. We forego recalling the full details for $\alpha$ since we only need the existence of the embedding and the properties that $U\subseteq \alpha(U)$ and $\alpha(U)-U\subseteq L_2-T_2$ for each $U\subseteq T_2$. Since $C$ is a maximal cluster of $\F$, we have $f^{-1}(C)$ is an upset in $\T_2$. Therefore, $\alpha(f^{-1}(C))$ is open in ${\mathbf L}_2$. Consider the equivalence relation $\equiv$ on $L_2$ given by
$$
a\equiv b\mbox{ iff }a=b\mbox{ or }(\exists w\in C)(a,b\in\alpha(f^{-1}(w)).
$$
Let $X$ be the quotient space $\mathbf L_2/\equiv$ and let $\rho:L_2\rightarrow X$ be the quotient map.

\bigskip

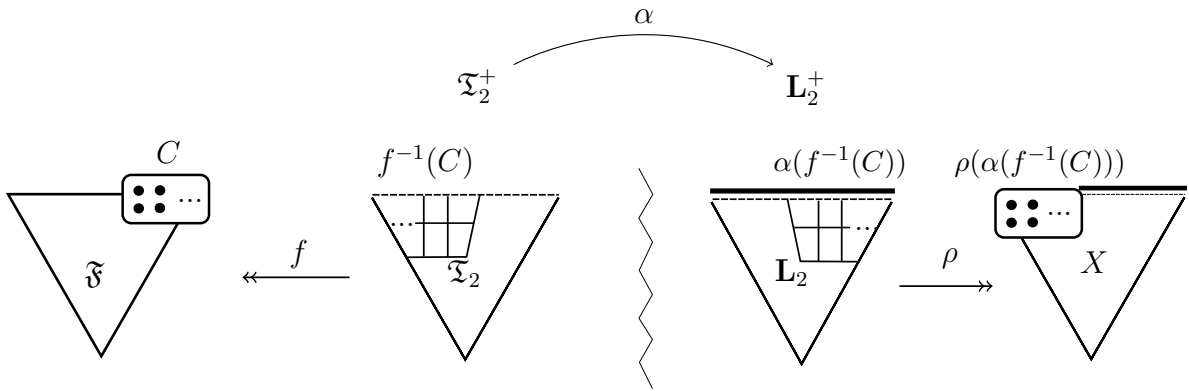
\begin{figure}[h]

\definecolor{cffffff}{RGB}{255,255,255}

\begin{tikzpicture}
  [y=0.80pt,x=0.80pt,yscale=-0.9,xscale=0.9, inner sep=0pt, outer sep=0pt]
  \path[cm={{0.30752719,0.0,0.0,0.30866997,(31.799472,126.05593)}},draw=black,line
    join=miter,line cap=butt,miter limit=4.00,line width=1pt]
    (90.9137,343.4126) -- (11.2284,205.3937) -- (-68.4568,67.3747) --
    (90.9137,67.3747) -- (250.2843,67.3747) -- (170.5990,205.3936) -- cycle;
  \path[draw=black,fill=cffffff,line join=round,line cap=butt,miter
    limit=4.00,nonzero rule,line width=0.965pt,rounded corners=0.1145cm]
    (70.9756,136.7431) rectangle (116.3476,162.1076);
  \path[cm={{0.84624856,0.0,0.0,0.80799141,(-1.9317919,48.086333)}},draw=black,fill=black,line
    join=round,line cap=butt,miter limit=4.00,nonzero rule,line width=1.229pt]
    (111.6218,119.7389)arc(0.000:180.000:2.273)arc(-180.000:0.000:2.273) -- cycle;
  \path[cm={{0.84624856,0.0,0.0,0.80799141,(-2.0042748,57.37054)}},draw=black,fill=black,line
    join=round,line cap=butt,miter limit=4.00,nonzero rule,line width=1.229pt]
    (111.6218,119.7389)arc(0.000:180.000:2.273)arc(-180.000:0.000:2.273) -- cycle;
  \path[cm={{0.84624856,0.0,0.0,0.80799141,(-13.117196,57.574588)}},draw=black,fill=black,line
    join=round,line cap=butt,miter limit=4.00,nonzero rule,line width=1.229pt]
    (111.6218,119.7389)arc(0.000:180.000:2.273)arc(-180.000:0.000:2.273) -- cycle;
  \path[cm={{0.84624856,0.0,0.0,0.80799141,(-13.117196,48.273968)}},draw=black,fill=black,line
    join=round,line cap=butt,miter limit=4.00,nonzero rule,line width=1.229pt]
    (111.6218,119.7389)arc(0.000:180.000:2.273)arc(-180.000:0.000:2.273) -- cycle;
  \path[draw=black,line join=miter,line cap=butt,miter limit=4.00,line
    width=0.844pt] (250.9597,233.8577) -- (226.4543,191.2554) --
    (201.9489,148.6531) -- (250.9597,233.7413) -- (299.9704,148.6531) --
    (275.4651,191.2554) -- cycle;
  \path[draw=black,fill=cffffff,dash pattern=on 2.65pt off 0.66pt,line
    join=miter,line cap=butt,miter limit=4.00,line width=0.662pt]
    (202.2338,147.1808) -- (300.1130,147.1808);
  \path[draw=black,line join=miter,line cap=butt,line width=0.662pt]
    (219.9717,180.0327) -- (251.6008,179.8286) -- (258.6532,147.1809);
  \path[draw=black,line join=miter,line cap=butt,line width=0.662pt]
    (229.1612,147.7930) -- (229.1612,179.6246);
  \path[draw=black,line join=miter,line cap=butt,line width=0.662pt]
    (241.7701,147.9971) -- (241.7701,179.8287);
  \path[draw=black,line join=miter,line cap=butt,line width=0.617pt]
    (224.6214,162.0764) -- (255.0721,162.0764);
  \path[draw=black,line join=miter,line cap=butt,miter limit=4.00,line
    width=0.829pt] (427.6714,236.1755) -- (451.3511,193.5732) --
    (475.0308,150.9709) -- (427.6714,236.0592) -- (380.3119,150.9709) --
    (403.9916,193.5732) -- cycle;
  \path[draw=black,fill=cffffff,dash pattern=on 2.60pt off 0.65pt,line
    join=miter,line cap=butt,miter limit=4.00,line width=0.650pt]
    (474.7555,149.4987) -- (380.1742,149.4987);
  \path[draw=black,line join=miter,line cap=butt,line width=0.650pt]
    (457.6153,182.3506) -- (427.0518,182.1465) -- (420.2370,149.4987);
  \path[draw=black,line join=miter,line cap=butt,line width=0.650pt]
    (448.7353,150.1109) -- (448.7353,181.9425);
  \path[draw=black,line join=miter,line cap=butt,line width=0.650pt]
    (436.5513,150.3149) -- (436.5513,182.1465);
  \path[draw=black,line join=miter,line cap=butt,line width=0.606pt]
    (453.1222,164.3943) -- (423.6975,164.3943);
  \path[draw=black,line join=miter,line cap=butt,miter limit=4.00,line
    width=1.951pt] (476.6576,145.1716) -- (379.5981,145.1716);
  \path[draw=black,line join=miter,line cap=butt,miter limit=4.00,line
    width=0.844pt] (579.8185,233.6369) -- (555.3131,191.0346) --
    (543.1992,169.7864) -- (579.8185,233.5206) -- (628.8293,148.4323) --
    (604.3239,191.0346) -- cycle;
  \path[draw=black,fill=cffffff,dash pattern=on 1.46pt off 0.37pt,line
    join=miter,line cap=butt,miter limit=4.00,line width=0.366pt]
    (575.0293,146.9601) -- (629.1609,146.9601);
  \path[draw=black,line join=miter,line cap=butt,miter limit=4.00,line
    width=1.771pt] (573.3031,143.6532) -- (630.1193,143.6532);
  \path[draw=black,fill=cffffff,line join=round,line cap=butt,miter
    limit=4.00,nonzero rule,line width=0.965pt,rounded corners=0.1145cm]
    (529.5215,144.2608) rectangle (574.8935,169.6253);
  \path[cm={{0.84624856,0.0,0.0,0.80799141,(456.61414,55.604038)}},draw=black,fill=black,line
    join=round,line cap=butt,miter limit=4.00,nonzero rule,line width=1.229pt]
    (111.6218,119.7389)arc(0.000:180.000:2.273)arc(-180.000:0.000:2.273) -- cycle;
  \path[cm={{0.84624856,0.0,0.0,0.80799141,(456.54166,64.888246)}},draw=black,fill=black,line
    join=round,line cap=butt,miter limit=4.00,nonzero rule,line width=1.229pt]
    (111.6218,119.7389)arc(0.000:180.000:2.273)arc(-180.000:0.000:2.273) -- cycle;
  \path[cm={{0.84624856,0.0,0.0,0.80799141,(445.42874,65.092294)}},draw=black,fill=black,line
    join=round,line cap=butt,miter limit=4.00,nonzero rule,line width=1.229pt]
    (111.6218,119.7389)arc(0.000:180.000:2.273)arc(-180.000:0.000:2.273) -- cycle;
  \path[cm={{0.84624856,0.0,0.0,0.80799141,(445.42874,55.79167)}},draw=black,fill=black,line
    join=round,line cap=butt,miter limit=4.00,nonzero rule,line width=1.229pt]
    (111.6218,119.7389)arc(0.000:180.000:2.273)arc(-180.000:0.000:2.273) -- cycle;
  \path[draw=black,line join=miter,line cap=butt,miter limit=4.00,line
    width=0.198pt] (342.1361,133.2458) -- (349.4800,146.9768) --
    (342.1361,159.4835) -- (349.3807,172.0850) -- (342.3391,186.8508) --
    (349.2815,199.6461) -- (342.2094,212.0623) -- (349.3074,223.2747) --
    (342.4642,235.2786) -- (349.3333,249.1539);
  \path[->>][draw=black,line join=miter,line cap=butt,miter limit=4.00,line
    width=0.662pt] (190.1526,190.5842) -- (133.9375,190.5842);
  \path[->>][draw=black,line join=miter,line cap=butt,line width=0.662pt]
    (479.0860,195.2013) -- (528.0476,195.2013);
  \path[->][draw=black,line join=miter,line cap=butt,line width=0.425pt]
    (276.4395,78.7641) .. controls (276.4395,78.7641) and (339.7083,43.4858) ..
    (412.7464,78.7641);
  \path[xscale=1.023,yscale=0.977,fill=black] (50,200) node[above
    right] (text4186) {${\mathfrak F}$     };
   \path[xscale=1.023,yscale=0.977,fill=black] (155,190) node[above
    right] (text4186) {${f}$     };
  \path[xscale=1.023,yscale=0.977,fill=black] (490,190) node[above
    right] (text4186) {$\rho$     };
  \path[xscale=1.023,yscale=0.977,fill=black] (332,58) node[above
    right] (text4186) {$\alpha$     };
  \path[xscale=1.023,yscale=0.977,fill=black] (241.32503,101.13189) node[above
    right] (text4186) {${\mathfrak T}_2^+$     };
  \path[xscale=1.023,yscale=0.977,fill=black] (410,102.24284) node[above
    right] (text4186-2) {${\bf L}_2^+$     };
  \path[xscale=1.023,yscale=0.977,fill=black] (236.11348,200) node[above
    right] (text4186-3) {${\mathfrak T}_2$     };
  \path[xscale=1.023,yscale=0.977,fill=black] (404.31546,199.76852) node[above
    right] (text4186-3-1) {${\bf L}_2$     };
  \path[xscale=1.023,yscale=0.977,fill=black] (561.07819,192.6808) node[above
    right] (text4186-3-3) {$X$     };
  \path[fill=black] (99.621376,151.17986) node[above right] (text4080) {...     };
  \path[fill=black] (210.92415,163.40961) node[above right] (text4080-1) {...
    };
  \path[fill=black] (455.50732,165.80872) node[above right] (text4080-4) {...
    };
  \path[fill=black] (556.27002,157.60123) node[above right] (text4080-9) {...   };
  \node (C0) at (95,125) {$C$};
  \node (C1) at (230,130) {$f^{-1}(C)$};
  \node (C2) at (448,130) {\small$\alpha(f^{-1}(C))$};
  \node (C3) at (553,130) {\small$\rho(\alpha(f^{-1}(C)))$};
\end{tikzpicture}
\caption{\label{f:space_subst_for_frame} Constructing $X$ and $\rho:{\bf L}_2\to X$}
\end{figure}

\bigskip

\begin{lemma}\label{QuotOfL2isIntrImage}
The space $X$ is an interior image of $\mathbf L_2$ under $\rho$.
\end{lemma}

\begin{proof}
Since a quotient map is always continuous and onto, we only need to show that $\rho$ is open. It is sufficient to show that $U\in\tau$ implies $\rho^{-1}(\rho(U))\in\tau$. Let $U\in\tau$. If $U\cap \alpha(f^{-1}(C))=\varnothing$, then $\rho^{-1}(\rho(U))=U\in\tau$. Suppose that $U\cap \alpha(f^{-1}(C))\not=\varnothing$.

\medskip

\noindent{\bf Claim:} $U\cap \alpha(f^{-1}(w))\not=\varnothing$ for each $w\in C$.

\smallskip

\noindent{\em Proof:} Since $\varnothing\not=U\cap \alpha(f^{-1}(C))\in\tau$, there is $a\in U\cap \alpha(f^{-1}(C)) \cap T_2$. As both $U$ and $\alpha(f^{-1}(C))$ are upsets in $\LL_2$, we have ${\uparrow}a\cap T_2 \subset {\uparrow}a\subseteq U\cap \alpha(f^{-1}(C))\subseteq \alpha(f^{-1}(C))$. Since $a\in \alpha(f^{-1}(C))$, there is $w\in C$ such that $a\in \alpha(f^{-1}(w))$. Moreover,
\begin{eqnarray*}
\alpha(f^{-1}(w))\cap T_2&=&\left(f^{-1}(w)\cup\left(\alpha(f^{-1}(w))-f^{-1}(w)\right)\right)\cap T_2\\
&=&\left(f^{-1}(w)\cap T_2\right)\cup\left(\left(\alpha(f^{-1}(w))-f^{-1}(w)\right)\cap T_2\right)\\
&=&f^{-1}(w)\cup\varnothing=f^{-1}(w).
\end{eqnarray*}
So $a\in f^{-1}(w)$, which implies that $a\in f^{-1}(C)$. Therefore, $f({\uparrow}a\cap T_2)\subseteq C$. In fact, $f({\uparrow}a\cap T_2)=C$ because $C$ is a cluster, ${\uparrow}a\cap T_2$ is an upset in $\T_2$, and $f$ is a p-morphism. Since ${\uparrow}a\cap T_2 \subset {\uparrow}a\subseteq U$, we see that $U\cap f^{-1}(v)\ne\varnothing$ for each $v\in C$. Thus, as $f^{-1}(v)\subseteq\alpha(f^{-1}(v))$, we conclude that $U\cap \alpha(f^{-1}(v))\ne\varnothing$ for each $v\in C$, proving the claim.

\medskip

Consequently, $\rho^{-1}(\rho(U))=U\cup\alpha(f^{-1}(C)) \in \tau$, completing the proof of the lemma.
\end{proof}

\begin{lemma}\label{SubaglebraOfQuotOfL2}
Let $\F$ and $X$ be as above. Then $\F^+$ is isomorphic to a subalgebra of $X^+$.
\end{lemma}

\begin{proof}
Since both $f^{-1}:\F^+\rightarrow \T_2^+$ and $\alpha:\T_2^+\rightarrow\mathbf L_2^+$ are closure algebra embeddings, $\alpha\circ f^{-1}:\F^+\rightarrow\mathbf L_2^+$ is a closure algebra embedding. By Lemma~\ref{QuotOfL2isIntrImage}, $\rho:L_2\to X$ is an onto interior map. Therefore, $\rho^{-1}:X^+\to\mathbf L_2^+$ is a closure algebra embedding. We show that if $A\in\F^+$, then $\rho^{-1}(\rho(\alpha(f^{-1}(A))))=\alpha(f^{-1}(A))$. Clearly $\alpha(f^{-1}(A))\subseteq\rho^{-1}(\rho(\alpha(f^{-1}(A))))$. For the converse, recalling that $C$ is a maximal cluster of $\F$, since $A=(A\cap C)\cup (A-C)$, we have
$$
f^{-1}(A)=f^{-1}(A-C)\cup\bigcup\{f^{-1}(w):w\in A\cap C\}.
$$
Therefore,
$$
\alpha(f^{-1}(A))=\left(\alpha(f^{-1}(A))-\alpha(f^{-1}(C))\right)\cup\bigcup\{\alpha(f^{-1}(w)):w\in A\cap C\}.
$$
Now suppose $a\in\rho^{-1}(\rho(\alpha(f^{-1}(A))))$. Then there is $b\in \alpha(f^{-1}(A))$ such that $\rho(a)=\rho(b)$. If $\rho(a)$ is a singleton, then $b=a$, so $a\in \alpha(f^{-1}(A))$. If $\rho(a)$ is not a singleton, then there is $w\in A\cap C$ such that $b\in\alpha(f^{-1}(w))$. Therefore, $a\in\alpha(f^{-1}(w))$. Since $w\in A$, it follows that $a\in \alpha(f^{-1}(A))$. Thus, $\rho^{-1}(\rho(\alpha(f^{-1}(A))))=\alpha(f^{-1}(A))$.

Consequently, $\alpha\circ f^{-1}$ embeds $\F^+$ into the image of $X^+$ under $\rho^{-1}$. This implies that the image of $\F^+$ under $\alpha\circ f^{-1}$ is a subalgebra of the image of $X^+$ under $\rho^{-1}$. Thus, $\F^+$ is isomorphic to a subalgebra of $X^+$.
\end{proof}

\smallskip

\begin{figure}[h]
\begin{tikzpicture}[description/.style={fill=white,inner sep=2pt}]
\matrix (m) [matrix of math nodes, row sep=3em,
column sep=2.5em, text height=1.5ex, text depth=0.25ex]
{ & \mathbf L_2^+ &  \\
\F^+ & & X^+  \\ };
\path[-,font=\scriptsize]
(m-2-1) edge[right hook->] node[auto] {$ \alpha\circ f^{-1} $} (m-1-2)
(m-2-3) edge[left hook->] node[above] {$\ \ \ \ \rho^{-1} $} (m-1-2)
(m-2-1) edge[right hook->, dashed] node[below] {Lemma~\ref{SubaglebraOfQuotOfL2}} (m-2-3);
\end{tikzpicture}
\caption{An embedding of $\F^+$ into $X^+$\label{f:diagram_for_X+_F+}}
\end{figure}
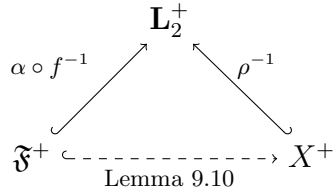

By the above construction, we may associate an interior image $X_n$ of ${\bf L}_2$ with each $\G_n$. Let $f_n:\T_2\to\G_n$ be the onto p-morphism used in defining $X_n$, and let $\rho_n:{\bf L}_2\to X_n$ be the quotient map. By Lemma~\ref{SubaglebraOfQuotOfL2}, each $\mathcal P_n$ is isomorphic to a subalgebra $\mathcal Q_n$ of $X_n^+$. So $\X_n=(X_n,\mathcal Q_n)$ is a general space satisfying $L(\G_n)=L(\X_n)$, and hence $\X_n$ is a general space for $L$ refuting $\varphi_n$. Moreover, the maximal cluster $C$ of $\G_n$ is realized as the open set $\rho_n(\alpha\circ f_n^{-1}(C))\subseteq X_n$. Note that since $L$ is above {\bf S4.2}, $C$ is a unique maximal cluster accessible from each point of $\G_n$, so the closure of $\rho_n(\alpha\circ f_n^{-1}(C))$ is $X_n$. We now perform the gluing of $\X_n$ along $\rho_n(\alpha\circ f_n^{-1}(C))$ to yield a general space $\X=(X,\mathcal Q)$. Since sums and interior images preserve validity, $\X\models L$. Moreover, since each $\X_n$ is an open subspace of $\X$ refuting $\varphi_n$, it follows that $\X$ refutes $\varphi_n$. Thus, $L=L(\X)$.

\smallskip

\noindent{\bf Step 1.3:} We next show that each $\X_n$ is an interior image of any non-trivial real interval.

\begin{lemma}\label{XisIntrImgOfClosedInterval}
Let $X$ be an interior image of ${\bf L}_2$ constructed above and let $I$ be a non-trivial interval in $\mathbf R$. Then there is an onto interior map $f:I\to X$ such that $f$ maps the endpoints of $I$ (if present) to the root of $X$.
\end{lemma}

\begin{proof}
By Theorem~\ref{L2isIntrImageR}, $\mathbf L_2$ is an interior image of $I$ and the endpoints get mapped to the root. By~Lemma~\ref{QuotOfL2isIntrImage}, $X$ is an interior image of $\mathbf L_2$, and the root of $\LL_2$ is mapped to the root of $X$. Taking the composition yields that $X$ is an interior image of $I$, and the endpoints are mapped to the root of $X$.
\end{proof}

\smallskip

\noindent{\bf Step 1.4:} As the final step, we produce an interior map from $\mathbf R$ onto $X$. Since $\X$ is obtained by gluing along the image of $C$ in $X_n$, we may identify $C$ as an open subset of $X$ that is countable.

\begin{lemma}\label{AlphaClusterAndForkIntrImgOfInvl}
Let $I$ be a non-trivial interval in $\mathbf R$ and let $\alpha\in\omega+1$ be nonzero. Then
\begin{enumerate}
\item $\mathfrak C_\alpha$ is an interior image of $I$.
\item $\F_\alpha$ is an interior image of $I$.
\end{enumerate}
\end{lemma}

\begin{proof}
(1) The case where $\alpha$ is finite is well known. If $\alpha=\omega$, then take any partition $\{Z_n:n\in\omega\}$ of $I$ into $\omega$-many dense and nowhere dense sets. It is routine to check that $f:I\rightarrow \mathfrak C_\omega$ is an onto interior map, where $f(x)=w_n$ whenever $x\in Z_n$; see \cite[Lem.~4.3]{JLB13}.

(2) Again the case where $\alpha$ is finite is well known. Let $\alpha=\omega$. Choose $z$ in $I$ such that $z$ is not an endpoint of $I$. Let $I_0=\{x\in I:x>z\}$. By (1), there is an onto interior mapping $f_0:I_0\rightarrow \mathfrak C_\omega$.
Define $f:I\rightarrow\F_\omega$ by
$$
f(x)=\left\{
\begin{array}{ll}
m_\omega&\mbox{ if }x<z,\\
r_\omega&\mbox{ if }x=z,\\
f_0(x)&\mbox{ if }x>z.
\end{array}\right.
$$
It is straightforward to check that $f$ is an onto interior map.
\end{proof}

For each $n\in\omega$, Lemma~\ref{AlphaClusterAndForkIntrImgOfInvl}(1) gives an onto interior map $f_n:(2n,2n+1)\rightarrow C$. By Lemma~\ref{XisIntrImgOfClosedInterval}, there is an onto interior map $g_n:[2n+1,2(n+1)]\rightarrow X_n$ that sends the endpoints $2n+1$ and $2(n+1)$ to the root of $X_n$. Define $f:(0,\infty)\rightarrow Y$ by
$$
f(x)=\left\{
\begin{array}{ll}
f_n(x)&\mbox{ if }x\in(2n,2n+1),\\
g_n(x)&\mbox{ if }x\in[2n+1,2(n+1)].
\end{array}\right.
$$

\begin{lemma}\label{Case2IntrOntoX}
The map $f:\mathbb (0,\infty)\rightarrow Y$ is an onto interior map.
\end{lemma}

\begin{proof}
Since each $g_n$ is onto, $f$ is onto. For an open interval $I\subseteq (0,\infty)$, we have
$$
f(I)=\bigcup\limits_{n\in\omega}\left(f_n(I\cap (2n,2n+1))\cup g_n(I\cap [2n+1,2(n+1)])\right).
$$
Each $f_n(I\cap (2n,2n+1))$ is either $C$ or $\varnothing$, both of which are open in $X_n$, and hence open in $X$. Since $I\cap [2n+1,2(n+1)]$ is open in $[2n+1,2(n+1)]$, we see that $g_n(I\cap [2n+1,2(n+1)])$ is open in $X_n$, and hence open in $X$. Thus, $f$ is open.

The basic open sets in $X$ arise from sets in $X_n$ of the form $\rho_n({\uparrow}a)$, where $a\in T_2$. We have
\begin{equation}\label{Eqn}
f^{-1}(\rho_n({\uparrow}a))=g_n^{-1}(\rho_n({\uparrow}a))\cup\bigcup\limits_{k\in\omega}(2k,2k+1) \cup \bigcup\limits_{k\in\omega} g_k^{-1}(C).
\end{equation}
Either $g_n^{-1}(\rho_n({\uparrow}a))$ is a proper subset of $[2n+1,2(n+1)]$ or not. If $g_n^{-1}(\rho_n({\uparrow}a))$ is proper, then the root of $X_n$ is not in $\rho_n({\uparrow}a)$, giving $g_n^{-1}(\rho_n({\uparrow}a))$ is open in $(2n+1,2(n+1))$, and hence $g_n^{-1}(\rho_n({\uparrow}a))$ is open in $(0,\infty)$. If $g_n^{-1}(\rho_n({\uparrow}a))=[2n+1,2(n+1)]$, then we may replace $g_n^{-1}(\rho_n({\uparrow}a))$ by $(2n,2(n+1)+1)$ in Equation~\ref{Eqn} and equality remains. Similarly, for each $k\in\omega$, either $g_k^{-1}(C)$ is a proper subset of $[2k+1,2(k+1)]$ or not. If $g_k^{-1}(C)$ is proper, then the root of $X_k$ is not in $C$, giving $g_k^{-1}(C)$ is open in $(2k+1,2(k+1))$, and hence $g_k^{-1}(C)$ is open in $(0,\infty)$. If $g_k^{-1}(C)=[2k+1,2(k+1)]$, then we may replace $g_k^{-1}(C)$ by $(2k,2(k+1)+1)$ in Equation~\ref{Eqn} and retain equality. Since $(2j,2(j+1)+1)$ is open in $(0,\infty)$ for any $j\in \omega$, when replacing as prescribed, we get that $f$ is continuous. Thus, $f$ is an onto interior map.
\end{proof}

Since $\X$ is an interior image of $(0,\infty)$ and $(0,\infty)$ is homeomorphic to $\mathbf R$, it follows that $\X$ is an interior image of $\mathbf R$. Since $L=L(\X)$ and $\X$ is an interior image of $\mathbf R$, the proof for the case $L\supseteq{\bf S4.2}$ is completed by applying Lemma~\ref{GenTopFromGFAndIntMap}.

\medskip

\noindent\underline{{\bf Case 2:} $L$ is not above {\bf S4.2}.} For each non-theorem $\varphi_n$ of $L$, let $\G_n=(W_n,R_n,\mathcal P_n)$ be the countable rooted general {\bf S4}-frame which was constructed in Step 2.2 of the proof of (1)$\Rightarrow$(2). Recall that $\G_n$ is a general frame for $L$ that refutes $\varphi_n$ at a root and that $\G_n$ has a maximal cluster $C_n$ that is isomorphic to $\mathfrak C_{\alpha_n}$.

\smallskip

\noindent{\bf Step 2.1:} By the construction in Step 1.2 (of (1)$\Rightarrow$(5)) and Lemma ~\ref{SubaglebraOfQuotOfL2}, there is a general space $\X_n=(X_n,\mathcal Q_n)$ such that $X_n$ is an interior image of $\mathbf L_2$ arising from $\G_n$ and $L(\X_n)=L(\G_n)$. Thus, $\X_n$ is a general space for $L$ refuting $\varphi_n$. We point out the maximal cluster $C_n$ of $\G_n$ is realized as the open subset $\rho_n(\alpha\circ f_n^{-1}(C_n))$, which we identify with $\mathfrak C_{\alpha_n}$.

\smallskip

\noindent{\bf Step 2.2:} We view the $\alpha_n$-fork $\F_{\alpha_n}$ as a general space. Let $\Y_n$ be the result of gluing $\X_n$ and $\F_{\alpha_n}$ along $\mathfrak C_{\alpha_n}$. In each $\Y_n$ there is the isolated point $m_{\alpha_n}$ coming from $\F_{\alpha_n}$. Let $\Y=(Y,\mathcal Q)$ be obtained by gluing the $\Y_n$ along a homeomorphic copy of $\{m_{\alpha_n}\}$. Then each $\Y_n$ is open in $\Y$ and hence $\Y$ refutes each $\varphi_n$. Moreover, since $\mathfrak F_{\alpha_n}\models L$ and $\Y_n\models L$ for each $n$, we have that $\Y\models L$. It follows that $L=L(\Y)$.

\smallskip

\noindent{\bf Step 2.3:} Lastly, we need to observe that $Y$ is an interior image of $\mathbf R$.

\begin{lemma}\label{GluedForksIntrImgOfInterval}
Let $\F$ be obtained by gluing the forks $\F_\alpha$ and $\F_\beta$ along their maximal points $m_\alpha$ and $m_\beta$. Let $I$ be a non-trivial interval and $y<z$ in $I$ be such that neither $y$ nor $z$ is an endpoint of $I$. There is an onto interior map $f:I\rightarrow \F$ such that $f(x)\in\mathfrak C_\alpha$ when $x<y$ and $f(x)\in\mathfrak C_\beta$ when $x>z$.
\end{lemma}

\begin{figure}[h]
\centering
\begin{tikzpicture}
\draw[rounded corners]  (-4,2) rectangle (-3,1.5);
\node (C0) at (-3.5,1.75) {\tiny $\mathfrak C_{\alpha}$};
\node[circle, fill, inner sep=1.5] (r1) at (-2.5,0) {};
\node[circle, fill, inner sep=1.5] (m) at (-1,1.5) {};
\node[circle, fill, inner sep=1.5] (r2) at (0.5,0) {};
\node at (-2.5,-0.5) {\small $r_\alpha$};
\node (mnew) at (-1,1) {\small $m$};
\node at (0.5,-0.5) {\small $r_\beta$};
\draw[rounded corners]  (2,2) rectangle (1,1.5);
\node (C1) at (1.5,1.75) {\tiny $\mathfrak C_{\beta}$};
\draw[->=latex, shorten <=2pt, shorten >=2pt] (r1) -- (m);
\draw[->=latex, shorten <=2pt, shorten >=2pt] (r2) -- (m);
\draw[->=latex, shorten <=2pt, shorten >=2pt] (r1) -- (-3.5,1.47);
\draw[->=latex, shorten <=2pt, shorten >=2pt] (r2) -- (1.5,1.47);
\node (l) at (-4,4) {\bf (};
\node (b1) at (-2.5,4) {$\bullet$};
\node at (-2.5,4.3) {\small $y$};
\node (b2) at (0.5,4) {$\bullet$};
\node at (0.5,4.3) {\small $z$};
\node (r) at (2,4) {\bf )};
\node (I0) at (-3.5,4.3) {\small $I_0$};
\node (M) at (-1,4.3) {\small $(y,z)$};
\node (I1) at (1.5,4.3) {\small $I_1$};
\draw[shorten <=-8pt, shorten >=-8pt, thick] (l) -- (r);
\draw[dashed, ->, shorten <=-2pt, shorten >=43pt ] (b1) -- (r1);
\draw[dashed, ->, shorten <=-2pt, shorten >=43pt] (b2) -- (r2);
\draw[dashed, ->, shorten <=0pt, shorten >=13pt ] (I0) -- (C0);
\draw[dashed, ->, shorten <=0pt, shorten >=13pt ] (I1) -- (C1);
\draw[dashed, ->, shorten <=0pt, shorten >=28pt ] (M) -- (mnew);
\end{tikzpicture}
\caption{\label{f:mapping_to_W} Mapping $I$ to $\F$}
\end{figure}

\begin{proof}
Let $I_0=\{x\in I:x<y\}$ and $I_1=\{x\in I:x>z\}$. By Lemma~\ref{AlphaClusterAndForkIntrImgOfInvl}(1), there are interior mappings $f_0:I_0\rightarrow \mathfrak C_\alpha$ and $f_1:I_1\rightarrow \mathfrak C_\beta$. Let $m\in \F$ be the image of $m_\alpha$ and $m_\beta$. Define $f:I\rightarrow \F$ by
$$
f(x)=\left\{
\begin{array}{ll}
m&\mbox{ if }x\in(y,z),\\
r_\alpha&\mbox{ if }x=y,\\
r_\beta&\mbox{ if }x=z,\\
f_i(x)&\mbox{ if }x\in I_i.
\end{array}\right.
$$
It is easy to check that $f$ is an onto interior map. Clearly $f(I_0)=\mathfrak C_\alpha$ and $f(I_1)=\mathfrak C_\beta$.
\end{proof}

We are ready to show that there is an interior map from {\bf R} onto $Y$. For each $n\in\omega$ we consider $I_{n,0}=(2n,2n+1)$ and $I_{n,1}=[2n+1,2(n+1)]$. Let $f_{0,0}:I_{0,0}\rightarrow \F_{\alpha_0}$ be the interior mapping as defined in Lemma~\ref{AlphaClusterAndForkIntrImgOfInvl}(2) with $z=\frac23$. For $n\in\omega-\{0\}$, let $f_{n,0}$ be the interior mapping of the interval $I_{n,0}$ onto the frame obtained by gluing $\F_{\alpha_{n{-}1}}$ and $\F_{\alpha_{n}}$ along the maximal point that is defined in the proof of Lemma~\ref{GluedForksIntrImgOfInterval}, where $y=2n+\frac13$ and $z=2n+\frac23$, such that $f_{n,0}(2n,2n+\frac13)=\mathfrak C_{\alpha_{n-1}}$ and $f_{n,0}(2n+\frac23,2(n+1))=\mathfrak C_{\alpha_n}$. Let $f_{n,1}:I_{n,1}\rightarrow X_n$ be given by Lemma~\ref{XisIntrImgOfClosedInterval}. Then the endpoints of $I_{n,1}$ are sent to the root of $X_n$. Define $f:(0,\infty)\rightarrow Y$ by $f(x)=f_{n,k}(x)$ when $x\in I_{n,k}$.

\begin{figure}[h]
\centering
\begin{tikzpicture}
\draw[rounded corners]  (-4,2) rectangle (-3,1.5);
\node at (-3.5,1.75) {};
\node[circle, fill, inner sep=1.5] (r1) at (-2.5,0) {};
\node[circle, fill, inner sep=1.5] (m) at (-1,1.5) {};
\node[circle, fill, inner sep=1.5] (r2) at (0.5,0) {};
\node[circle, fill, inner sep=1.5] (prer1) at (-2.5,4) {};
\node[circle, fill, inner sep=1.5] (prer2) at (0.5,4) {};
\draw[rounded corners]  (2,2) rectangle (1,1.5);
\node at (1.5,1.75) {};
\draw[->=latex, shorten <=2pt, shorten >=2pt] (r1) -- (m);
\draw[->=latex, shorten <=2pt, shorten >=2pt] (r2) -- (m);
\draw[->=latex, shorten <=2pt, shorten >=2pt] (r1) -- (-3.5,1.47);
\draw[->=latex, shorten <=2pt, shorten >=2pt] (r2) -- (1.5,1.47);
\path[draw] (-4.75,1.75) -- (-5.75,0) -- (-6.75,1.75) -- (-4.75,1.75);
\draw[rounded corners, fill=white]  (-5.25,2) rectangle (-4.25,1.5);
\path[draw] (2.75,1.75) -- (3.75,0) -- (4.75,1.75) -- (2.75,1.75);
\draw[rounded corners, fill=white]  (2.25,2) rectangle (3.25,1.5);
\node (2n-2) at (-7,4) {\bf [};
\node at (-4.2,4) {\bf ]};
\node at (-4.16,4) {\bf (};
\node at (2.06,4) {\bf )};
\node at (2.1,4) {\bf [};
\node (2n+2) at (5,4) {\bf ]};
\draw[shorten <=-6pt, shorten >=-6pt, thick] (2n-2) -- (2n+2);
\node (In1) at (-5.75,4.3) {\small $I_{n,1}$};
\node (Xn) at (-5.75,1) {\small $\mathfrak X_n$};
\node (In+10) at (-1.05,4.3) {\small $I_{n+1,0}$};
\node (In+11) at (3.75,4.3) {\small $I_{n+1,1}$};
\node (Xn+1) at (3.75,1) {\small $\mathfrak X_{n+1}$};
\node (XnCn) at (-4.75,1.75) {\small $\mathfrak C_n$};
\node (Cn) at (-3.5,1.75) {\small $\mathfrak C_n$};
\node (preCn) at (-3.5,4) {};
\node (Cn+1) at (1.5,1.75) {\small $\mathfrak C_{n+1}$};
\node (preCn+1) at (1.5,4) {};
\node (Xn+1Cn+1) at (2.75,1.75) {\small $\mathfrak C_{n+1}$};
\draw[dashed, ->, shorten <=0pt, shorten >=28pt ] (In1) -- (Xn);
\draw[dashed, ->, shorten <=-4pt, shorten >=18pt ] (preCn) -- (Cn);
\draw[dashed, ->, shorten <=0pt, shorten >=28pt ] (prer1) -- (r1);
\draw[dashed, ->, shorten <=0pt, shorten >=28pt ] (In+10) -- (m);
\draw[dashed, ->, shorten <=0pt, shorten >=28pt ] (prer2) -- (r2);
\draw[dashed, ->, shorten <=-4pt, shorten >=18pt ] (preCn+1) -- (Cn+1);
\draw[dashed, ->, shorten <=0pt, shorten >=28pt ] (In+11) -- (Xn+1);
\end{tikzpicture}
\caption{\label{f:mapping_intervals} Depiction of $f$}
\end{figure}
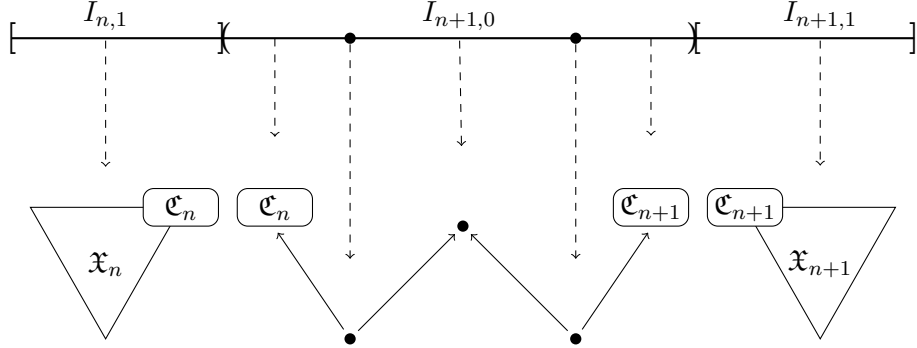

\begin{lemma}\label{Case1IntrOntoX}
The map $f:\mathbb (0,\infty)\rightarrow Y$ is an onto interior map.
\end{lemma}

\begin{proof}
It is clear that $f$ is onto since $f|_{I_{n,0}}=f_{n,0}$ is onto $\F_{\alpha_n}$ and $f|_{I_{n,1}}=f_{n,1}$ is onto $X_n$. Let $I\subseteq \mathbb (0,\infty)$ be open. Then $I\cap I_{n,k}$ is open in $I_{n,k}$, and hence $f(I\cap I_{n,k})=f_{n,k}(I\cap I_{n,k})$ is open in $f_{n,k}(I_{n,k})$. Therefore, $f(I\cap I_{n,k})$ is open in $Y$. Thus,
$$
f(I)=\bigcup\limits_{n\in\omega}f_{n,0}(I\cap I_{n,0})\cup f_{n,1}(I\cap I_{n,1})
$$
is open in $Y$. This implies that $f$ is an open map.

Let $U\subseteq Y$ be open. Then
\begin{equation}\label{fIsCon}f^{-1}(U)=\bigcup\limits_{n\in\omega}(f_{n,0})^{-1}(U\cap \F_{\alpha_n})\cup\bigcup\limits_{X_n\not\subseteq U}(f_{n,1})^{-1}(U\cap X_n)\cup\bigcup\limits_{X_n\subseteq U}(f_{n,1})^{-1}(U\cap X_n).\end{equation}
We have each $(f_{n,0})^{-1}(U\cap \F_{\alpha_n})$ is open in $(2n,2n+1)$ and hence open in $(0,\infty)$. For $X_n\not\subseteq U$ we have each $(f_{n,1})^{-1}(U\cap X_n)$ is open in $(2n+1,2n+2)$ and hence open in $(0,\infty)$. When $X_n\subseteq U$ we can replace the closed interval $(f_{n,1})^{-1}(U\cap X_n)$ by the open interval $(2n+1-\frac13,2n+2+\frac13)$  in Equation \ref{fIsCon} and still retain equality. This shows $f$ is continuous.
\end{proof}

Since $(0,\infty)$ is homeomorphic to $\mathbf R$, it follows that $Y$ is an interior image of $\mathbf R$. Applying Lemma~\ref{GenTopFromGFAndIntMap} finishes the proof of (1)$\Rightarrow$(5), and hence the proof of the Main Result.

We conclude this section by mentioning the following useful consequence of the Main Result. Recall that ${\bf S4.1}={\bf S4}+\Box\Diamond\varphi\to\Diamond\Box\varphi$. By \cite[Thm.~5.3]{BG11}, each logic above {\bf S4.1} is connected. Also, ${\bf S4.1}\subseteq{\bf S4.Grz}$, where ${\bf S4.Grz}={\bf S4}+\Box(\Box(\varphi\to\Box\varphi)\to\varphi)\to\varphi$ is the Grzegorczyk logic. As an immediate consequence of the Main Result, we obtain:

\begin{corollary}
If $L$ is a logic above {\bf S4.1}, then $L$ is the logic of a general space over $\mathbf R$ or equivalently $L$ is the logic of a subalgebra of $\mathbf R^+$. In particular, if $L$ is a logic above {\bf S4.Grz}, then $L$ is the logic of a general space over $\mathbf R$ or equivalently $L$ is the logic of a subalgebra of~$\mathbf R^+$.
\end{corollary}

\section{Intermediate logics}\label{SecIntLogs}

In this section we apply our results to intermediate logics. We recall that intermediate logics are the logics that are situated between the intuitionistic propositional calculus {\bf IPC} and the classical propositional calculus {\bf CPC}; that is, $L$ is an intermediate logic if ${\bf IPC}\subseteq L\subseteq{\bf CPC}$. There is a dual isomorphism between the lattice of intermediate logics and the lattice of non-degenerate varieties of Heyting algebras, where we recall that a Heyting algebra is a bounded distributive lattice $A$ equipped with an additional binary operation $\to$ that is residual to $\wedge$; that is, $a\wedge x\le b$ iff $x\le a\to b$.

There is a close connection between closure algebras and Heyting algebras. Each closure algebra $\mathfrak A=(A,\Box)$ gives rise to the Heyting algebra $\HH(\mathfrak A)=\{a\in A:a=\Box a\}$ of open elements of $\mathfrak A$, and each Heyting algebra $\HH=(H,\to)$ generates the closure algebra $\mathfrak A(\HH)=(B(H),\Box)$, where $B(H)$ is the free Boolean extension of $H$ and for $x\in B(H)$, if $x=\bigwedge_{i=1}^n(\neg a_i\vee b_i)$, then $\Box x=\bigwedge_{i=1}^n(a_i\to b_i)$ \cite{MT46}. Also, if $(W,R,\mathcal P)$ is a general {\bf S4}-frame and $R$ is a partial order, then $(W,R,\mathcal P_R)$ is a general intuitionistic frame, where we recall that $\mathcal P_R=\{A\in\mathcal P:A$ is an $R$-upset$\}$, and there is an isomorphism between partially ordered descriptive {\bf S4}-frames and descriptive intuitionistic frames (see, e.g., \cite[Sec.~8]{CZ97}).

This yields the well-known correspondence between intermediate logics and logics above {\bf S4}. Namely, each intermediate logic can be viewed as a fragment of a consistent logic above {\bf S4}, and this can be realized through the G\"odel translation (which translates each formula $\varphi$ of the language of {\bf IPC} to the modal language by adding $\Box$ to every subformula of $\varphi$). Then the lattice of intermediate logics is isomorphic to an interval in the lattice of logics above {\bf S4}, and the celebrated Blok-Esakia theorem states that this interval is exactly the lattice of consistent logics above {\bf S4.Grz} (see, e.g., \cite[Sec.~9]{CZ97}).

An element $a$ of a Heyting algebra $\HH$ is \emph{complemented} if $a\vee\neg a=1$, and $\HH$ is \emph{connected} if $0,1$ are the only complemented elements of $\HH$. Also, $\HH$ is \emph{well-connected} if $a\vee b=1$ implies $a=1$ or $b=1$. Then it is easy to see that a closure algebra $\mathfrak A$ is connected iff the Heyting algebra $\HH(\mathfrak A)$ is connected, and that $\mathfrak A$ is well-connected iff $\HH(\mathfrak A)$ is well-connected.

An intermediate logic $L$ is \emph{connected} if $L=L(\HH)$ for some connected Heyting algebra $\HH$, and $L$ is \emph{well-connected} if $L=L(\HH)$ for some well-connected Heyting algebra $\HH$. By \cite[Thm.~8.1]{BG11}, each intermediate logic is connected. Since a Heyting algebra $\mathfrak A$ is well-connected iff its dual descriptive intuitionistic frame is rooted (\cite{Esa79,BB08}), by \cite[Thm.~15.28]{CZ97}, $L$ is well-connected iff $L$ is Hallden complete, where we recall that $L$ is Hallden complete provided $\varphi\vee\psi\in L$ and $\varphi,\psi$ have no common propositional letters imply that $\varphi\in L$ or $\psi\in L$.

For a topological space $X$, let $\Omega(X)$ denote the Heyting algebra of open subsets of $X$. Similarly, for a partially ordered frame $\F=(W,\le)$, let $\mathrm{Up}(\F)$ denote the Heyting algebra of upsets of $\F$. For a general space $(X,\tau,\mathcal P)$, recall that $\mathcal P_\tau=\mathcal P\cap\tau$. Then $\mathcal P_\tau$ is a Heyting algebra, and for each Heyting algebra $\HH$, there is a descriptive space $(X,\tau,\mathcal P)$ such that $\HH$ is isomorphic to $\mathcal P_\tau$. For a general space $(X,\tau,\mathcal P)$, we call $(X,\tau,\mathcal P_\tau)$ a \emph{general intuitionistic space}.

The Blok-Esakia theorem together with the results obtained in this paper yield the following theorems.

\begin{theorem}
The following are equivalent.
\begin{enumerate}
\item $L$ is an intermediate logic.
\item $L$ is the logic of a countable path-connected general intuitionistic frame.
\item $L$ is the logic of a general intuitionistic space over $\mathbf R$.
\item $L$ is the logic of a general intuitionistic space over $\mathbf Q$.
\item $L$ is the logic of a general intuitionistic space over $\mathbf C$.
\item $L$ is the logic of a Heyting subalgebra of the Heyting algebra $\Omega(\mathbf R)$.
\item $L$ is the logic of a Heyting subalgebra of the Heyting algebra $\Omega(\mathbf Q)$.
\item $L$ is the logic of a Heyting subalgebra of the Heyting algebra $\Omega(\mathbf C)$.
\end{enumerate}
\end{theorem}

\begin{theorem}
Let $L$ be an intermediate logic. The following are equivalent.
\begin{enumerate}
\item $L$ is well-connected.
\item $L$ is Hallden complete.
\item $L$ is the logic of a general intuitionistic frame over $\T_2$.
\item $L$ is the logic of a general intuitionistic frame over $\LL_2$.
\item $L$ is the logic of a Heyting subalgebra of the Heyting algebra $\mathrm{Up}(\T_2)$.
\item $L$ is the logic of a Heyting subalgebra of the Heyting algebra $\mathrm{Up}(\LL_2)$.
\end{enumerate}
\end{theorem}

\bigskip

\noindent {\bf Acknowledgement:} The first two authors acknowledge the support of the grant \# FR/489/5-105/11 of the Rustaveli Science Foundation of Georgia.

\bibliographystyle{amsplain}
\def\cprime{$'$}
\providecommand{\bysame}{\leavevmode\hbox to3em{\hrulefill}\thinspace}
\providecommand{\MR}{\relax\ifhmode\unskip\space\fi MR }
\providecommand{\MRhref}[2]{%
  \href{http://www.ams.org/mathscinet-getitem?mr=#1}{#2}
}
\providecommand{\href}[2]{#2}

\bigskip

\noindent Guram Bezhanishvili: Department of Mathematical Sciences, New Mexico State University, Las Cruces NM 88003, USA, gbezhani@math.nmsu.edu

\bigskip

\noindent David Gabelaia: A.~Razmadze Mathematical Institute, Tbilisi State University, University St.~2, Tbilisi 0186, Georgia, gabelaia@gmail.com

\bigskip

\noindent Joel Lucero-Bryan: Department of Applied Mathematics and Sciences, Khalifa University, Abu Dhabi, UAE, joel.lucero-bryan@kustar.ac.ae

\end{document}